\documentclass{amsart} 
\usepackage[T1]{fontenc}
\usepackage{indentfirst}
\usepackage{amsmath}
\usepackage{amssymb}
\usepackage{amsthm}
\usepackage{cancel}
\usepackage{dsfont}
\usepackage{stmaryrd}
\SetSymbolFont{stmry}{bold}{U}{stmry}{m}{n}
\usepackage{hyperref}
\usepackage{empheq}
\usepackage{cleveref}
\usepackage{bm}
\usepackage{enumitem}
\usepackage{cancel}
\usepackage[pdftex]{graphicx}
\usepackage{pgf, tikz}
\usepackage{pgfplots}
\usepackage[toc,page]{appendix}

\title[Instability in Vlasov-Poisson around rough profiles]{Nonlinear instability in Vlasov type equations around rough velocity profiles}
\author{Aymeric \textsc{Baradat}}
\address{CMLS, \'Ecole Polytechnique, 91128 Palaiseau, France}
\email{aymeric.baradat@polytechnique.edu}
\date{}

\theoremstyle{plain}
\newtheorem{Thm}{Theorem}[section]
\newtheorem{Cor}[Thm]{Corollary}
\newtheorem{Prop}[Thm]{Proposition}
\newtheorem{Lem}[Thm]{Lemma}
\theoremstyle{definition}
\newtheorem{Def}[Thm]{Definition}
\newtheorem{Ass}[Thm]{Assumption}
\theoremstyle{remark}
\newtheorem{Rem}[Thm]{Remark}

\newcommand{\R}{\mathbb{R}}
\newcommand{\Z}{\mathbb{Z}}
\newcommand{\T}{\mathbb{T}}

\newcommand{\N}{\mathbb{N}}
\newcommand{\C}{\mathbb{C}}
\newcommand{\X}{\mathcal{X}}

\newcommand{\cg}{\langle}
\newcommand{\cd}{\rangle}
\newcommand{\1}{\mathds{1}}
\newcommand{\eps}{\varepsilon}
\newcommand{\vv}{\boldsymbol{v}}
\newcommand{\uu}{\boldsymbol{u}}
\newcommand{\rr}{\boldsymbol{r}}
\newcommand{\rhorho}{\boldsymbol{\rho}}
\newcommand{\sigsig}{\boldsymbol{\sigma}}
\newcommand{\xixi}{\boldsymbol{\xi}}
\newcommand{\FF}{\boldsymbol{f}}
\newcommand{\GG}{\boldsymbol{g}}
\newcommand{\XX}{\boldsymbol{X}}
\newcommand{\pp}{\boldsymbol{p}}
\newcommand{\qq}{\boldsymbol{q}}

\newcommand{\VERT}{\big| \hspace{-1pt} \big| \hspace{-1pt} \big|}

\DeclareMathOperator{\Div}{div}
\DeclareMathOperator{\bDiv}{\textbf{div}}

\DeclareMathOperator{\D}{d\!}

\DeclareMathOperator{\Id}{Id}

\DeclareMathOperator{\DD}{\!D\!}
 
\begin{document} 
\begin{abstract}
In the Vlasov-Poisson equation, every configuration which is homogeneous in space provides a stationary solution. Penrose gave in \cite{pen60} a criterion for such a configuration to be linearly unstable. While this criterion makes sense in a measure-valued setting, the existing results concerning nonlinear instability always suppose some regularity with respect to the velocity variable. Here, thanks to a multiphasic reformulation of the problem, we can prove an "almost Lyapounov instability" result for the Vlasov-Poisson equation, and an ill-posedness result for the kinetic Euler equation and the Vlasov-Benney equation (two quasineutral limits of the Vlasov-Poisson equation), both around any unstable measure. 
\end{abstract}
\maketitle
 \tableofcontents
 \section{Introduction}
 The purpose of this paper is to study the nonlinear instability of a general class of Vlasov equations (see Section \ref{model}), \emph{i.e.} evolution equations for the density of a system of particles in the phase space. This class of equations contains the Vlasov-Poisson equation and some of its asymptotic limits. The specificity of this work is the fact that we deal with the case when for all $(t,x)$, $f(t,x,\bullet)$ is only supposed to be a measure. Let us first describe the physical models we have in mind. 
 \subsection{Presentation of the physical models}
 \label{physicalmodels}
 We will apply our abstract result to three models studied in the field of plasma physics: the Vlasov-Poisson equation for electrons, the kinetic Euler equation and the Vlasov-Benney equation. Let us present them one by one. 
 \subsubsection*{The Vlasov-Poisson equation for electrons} A population of electrons of unit mass and unit negative charge moving in a homogeneous environment of fixed particles of positive charge can be described by a Vlasov-Poisson type equation. If the domain is the $d$-dimensional torus $\T^d := \R^d / \Z^d$, this equation governs the evolution over time of the density of electrons $f = (f(t,x,v), t \in [0,T], x \in \T^d, v \in \R^d)$ in the phase space $\T^d \times \R^d$. They write in the following way:
\begin{equation}
\label{VP}
\left\{ \begin{gathered}
\partial_t f(t,x,v) + v\cdot \nabla_x f(t,x,v) - \nabla_x U(t,x) \cdot \nabla_v f(t,x,v) = 0,\\
- \Delta_x U(t,x) = \int f(t,x,v) \D v - 1,\\
f(0,x,v)= f_0(x,v).
\end{gathered} \right.
\end{equation} 
It means that the electrons follow the Newton dynamics in the electric potential $U$ they induce together with the fixed charges. This potential is obtained through an elliptic equation involving the density of electrons in space. 

This equation is of major interest in plasma physics, and so has been extensively studied. Among the huge literature about it, global existence of classical solutions to the Cauchy problem has been obtained in dimension 2 by Ukay--Okabe in \cite{ukai1978classical}, and in dimension 3 by Lions--Perthame in \cite{lions1991propagation} and by Pfaffelmoser in \cite{pfaffelmoser1992global}. We refer to \cite{glassey1996cauchy} for an overview of the subject.

\subsubsection*{The kinetic Euler equation} This equation is deduced from the previous one in the regime of \emph{small Debye length}, also called \emph{quasineutral limit} (see \cite{grenier1996oscillations}). It reads
\begin{equation}
\label{KEu}
\left\{ \begin{gathered}
\partial_t f(t,x,v) + v\cdot \nabla_x f(t,x,v) - \nabla_x p(t,x) \cdot \nabla_v f(t,x,v) = 0,\\
\int f(t,x,v) \D v = 1,\\
f(0,x,v) =f_0(x,v).
\end{gathered} \right.
\end{equation} 
It can be seen as a kinetic version of the Euler equation for incompressible fluids: as in the hydrodynamic case, the particles follow the Newton dynamics in a \emph{pressure field} $p$, which is the Lagrange multiplier associated to the constraint
\begin{equation}
\label{incompressibilityconstraint}
\int f(t,x,v) \D v = 1.
\end{equation}
Incidentally, to any monokinetic solutions to \eqref{KEu} corresponds a solution to the Euler equation and \emph{vice versa}.

This analogy goes further. Indeed, this equation is linked to an optimization problem, the so-called Brenier model (see for example \cite{bre89,bre93,amb09}). Following ideas by Arnold (in \cite{arn66,arn99}), this model aims to understand the behaviour of incompressible fluids as the geodesics of the set of measure-preserving diffeomorphisms, which is seen as a formal Riemannian manifold of infinite dimension. In the smooth case (considered by Arnold), the geodesic equation is nothing but the Euler equation, whereas in general, as shown by Shnirelman in \cite{shn87}, we cannot prevent particles from crossing each other, and we obtain solutions to the kinetic Euler equation (at least in a weak sense). A study of \eqref{KEu} with PDE techniques provides information on the optimization problem: using the present paper, the author shows in \cite{Baradat2018kinetic} that the optimal action in the Brenier model, although continuous (see \cite{baradat2018continuous}) cannot be Lipschitz continuous with respect to the data.

\subsubsection*{The Vlasov-Benney equation}  This equation is another formal limit of the Vlasov-Poisson equation in the quasineutral limit. But this time, it corresponds to the case when we look at the evolution of the population of ions whose masses are far higher than electrons ones. It reads
\begin{equation}
\label{VB}
\left\{ \begin{gathered}
\partial_t f(t,x,v) + v\cdot \nabla_x f(t,x,v) - \nabla_x \rho(t,x) \cdot \nabla_v f(t,x,v) = 0,\\
\rho(t,x) = \int f(t,x,v) \D v ,\\
f(0,x,v) = f_0(x,v).
\end{gathered} \right.
\end{equation}
We refer for instance to \cite{han2011quasineutral} for its derivation. The study of this Cauchy problem has aroused great interest in the last few years, as evidenced by the works of Bardos \cite{bardos2012variant}, Bardos--Besse \cite{bardos2013cauchy,bardos2015hamiltonian}, Han-Kwan--Rousset \cite{han2016quasineutral} and references therein.
 \subsection{Homogeneous profiles and the Penrose condition}
 \label{presentationpenrosecondition}
 The three equations \eqref{VP}, \eqref{KEu} and \eqref{VB} admit stationary solutions of a particular form: those which depend only on the velocity variable. In each case, any smooth profile $\mu = (\mu(v))$ satisfying
\[
\int \mu(v) \D v = 1
\] 
 gives rise to a stationary homogeneous solution. The goal of the present work is to study the nonlinear instability of the three models around such profiles.
 
 At the linear level, the question of linear stability dates back to the late 50's and resulted in the seminal paper \cite{pen60}. In this article, Penrose gave in the context of the Vlasov-Poisson equation \eqref{VP} a necessary and sufficient condition on a profile $\mu$ to be linearly unstable. Let us present this condition. For given $n \in \Z^d$ and $\omega \in \C^d$, the linearization of the Vlasov-Poisson equation \eqref{VP} around a smooth profile $\mu$ admits a solution of the form
 \begin{equation}
 \label{defexponentialgrowingmode}
a(v) \exp \Big(in\cdot(x - \omega t) \Big)
\end{equation}
for some function $a$ if and only if $(n, \omega)$ satisfies the equation
  \begin{equation}
 \label{penroseVP}
\int_{\R^d} \frac{n \cdot \nabla_v \mu(v)}{n \cdot (v - \omega)} \D v = |n|^2.
 \end{equation}
If in addition $\Im(n \cdot \omega) >0$, then this solution is an \emph{exponential growing mode}, and the stationary solution $\mu$ turns out to be linearly unstable. Therefore, for the Vlasov-Poisson equation \eqref{VP}, we can give the following criterion for exponential growing modes to exist.
\begin{Def}[Penrose instability condition for Vlasov-Poisson]
\label{def:Penrose_VP}
 The smooth profile $\mu$ is said to be Penrose unstable for the Vlasov-Poisson \eqref{VP} equation if there exist $n \in \Z^d$ and $\omega \in \C^d$ such that $\Im(n \cdot \omega) >0$ and satisfying \eqref{penroseVP}. 
\end{Def}
 
 In the other models, similar formulae can be found, and lead to the following definitions. 
 \begin{Def}[Penrose instability condition for kinetic Euler]
 \label{def:Penrose_KEu}
  The smooth profile $\mu$ is said to be Penrose unstable for the kinetic Euler equation \eqref{KEu} if there exist $n \in \Z^d$ and $\omega \in \C^d$ such that $\Im(n \cdot \omega) >0$ and satisfying  
   \begin{equation}
   \label{penroseKEu}
  \int_{\R^d} \frac{n \cdot \nabla_v \mu(v)}{n \cdot (v - \omega)} \D v = 0.
   \end{equation}
 \end{Def}
 \begin{Def}[Penrose instability condition for Vlasov-Benney]
 \label{def:Penrose_VB}
  The smooth profile $\mu$ is said to be Penrose unstable for the Vlasov-Benney equation \eqref{VB} if there exist $n \in \Z^d$ and $\omega \in \C^d$ such that $\Im(n \cdot \omega) >0$ and satisfying  
   \begin{equation}
   \label{penroseVB}
  \int_{\R^d} \frac{n \cdot \nabla_v \mu(v)}{n \cdot (v - \omega)} \D v = 1.
   \end{equation}
 \end{Def}
 
Once again, in the three cases, $\mu$ is Penrose unstable for one model if and only if it is linearly unstable when considered as a stationary solution to this model.
 
In these three cases, classical examples of stable profiles are the ones admitting a unique maximum. For example, a Maxwellian is always stable. On the contrary, profiles with two bumps like the superposition of two sufficiently distant Maxwellian are unstable. We refer to \cite{pen60} to see how to deduce from formulae \eqref{penroseVP}, \eqref{penroseKEu} and \eqref{penroseVB} if a profile $\mu$ is stable or not using complex analysis.
 
 \subsection{Known results for nonlinear instability}
 In the case when $\mu$ is Penrose unstable, it is possible to derive nonlinear instability results. We will present some known results in this subsection. 
 
 But before doing it, let us point out a crucial difference between formula \eqref{penroseVP} and the two formulae \eqref{penroseKEu} and \eqref{penroseVB}. In the two last ones, as soon as we can find $n \in \Z^d$ and $\omega \in \C^d$ such that $\Im(n \cdot \omega) >0$ and satisfying \eqref{penroseKEu} or \eqref{penroseVB}, then for all $k \in \N^*$, $kn$ and $\omega$ satisfy the same properties (this is a consequence of the scale invariance of these equations as explained in Subsection \ref{paragraphhiglyunbounded}). In view of \eqref{defexponentialgrowingmode}, it means that for any unstable profile, we can find exponential growing modes with arbitrary large frequency $n$ and with growing rate $\Im(n \cdot \omega)$ proportional to this large frequency. The instability is therefore far more violent in these cases. This additional property is the reason why the results that we will present are not the same for the Vlasov-Poisson equation and for the two other models.
 
 We also insist on the fact that in all the results presented below, $\mu$ is supposed to be smooth ($C^1$ in the case of Guo and Strauss, and analytic in the other cases). It also have to satisfy a technical assumption on the way it cancels (see the so-called $\delta$ and $\delta'$-conditions in \cite{han2015stability}, designed to ensure that the solutions built are nonnegative). We will see in Section \ref{sec:new_results} that we can drop these assumptions: we are able to recover some of these results only assuming that $\mu$ is a measure. 
 
\subsubsection*{Lyapounov instability for Vlasov-Poisson} To our knowledge, the first result of nonlinear instability for the Vlasov-Poisson equation \eqref{VP} was proved by Guo and Strauss in \cite{guo1995nonlinear}. It consists in a Lyapounov instability result in the $C^1$ norm in both variables $x$ and $v$.
 
More recently, Han-Kwan--Hauray (in \cite{han2015stability}, in the case of dimension one) and Han-Kwan--Nguyen (in \cite{han2016nonlinear}, in any dimension) showed that the Penrose instability of a smooth profile $\mu= (\mu(v))$ can be used to build a family $(f_{k})_{k \in \N}$ of solutions to \eqref{VP} and a family of times $(T_k)_{k \in \N}$ with 
\[
f_{k}|_{t=0} \underset{k \to + \infty}{\longrightarrow}\mu 
\]
strongly in any $H^{s}_{x,v}$ but,
\[
\|f_{k} - \mu\| \underset{k \to + \infty}{\cancel{\longrightarrow}} 0, 
\]
where the norm is the one of $L^{\infty}([0,T_k); H^{s'}_{x,v})$ whatever $s'\in\Z$. Roughly speaking, the Lyapounov instability holds even if the initial data is taken close to the equilibrium in a very strong topology, and even if we measure the distance to the equilibrium at further times in a very weak norm. In that case, the sequence $(T_k)$ is of the following order:
\begin{equation*}
T_k \underset{k \to + \infty}{\sim} \big( |\log \eps_k | \big) \quad \mbox{with} \quad \eps_k := \| f_k|_{t=0} - \mu \|_{H^{s'}_{x,v}} .
\end{equation*}
It means that the exponential growing rate of the solutions to the linearized problem prevails.

In a slightly different context, let us also mention Cordier--Grenier--Guo \cite{cordier2000two} who proved a similar result for several systems of equations governing plasmas with two phases in one space dimension (related to the Euler-Poisson system). Somehow, we present here a framework that encompasses the classical kinetic setting and this kind of multiphase settings.  
\subsubsection*{Ill-posedness for kinetic Euler and Vlasov-Benney}  In \cite{han16}, following ideas by M\'etivier (see \cite{metivier2005remarks}), Han-Kwan and Nguyen proved that \eqref{KEu} and \eqref{VB} are ill-posed in any $H^s_{x,v}$ in the following sense: they show for instance that for any $s \in \N$ and any $T>0$, the map
\[
\begin{array}{c c c}
H^s_{x,v} & \to& L^2([0,T)\times \T^d \times \R^d) \\
f_0 & \mapsto& f \mbox{ solution to \eqref{KEu} or \eqref{VB}},
\end{array}
\]
if exists, cannot be H\"older continuous with any exponent in $(0,1]$ in the neighbourhood of any smooth linearly unstable profile. To do so, for a fixed analytic Penrose unstable profile $\mu$, $s>0$ and $\alpha \in (0,1]$, they build a sequence of times $(T_k)$ tending to $0$ and a sequence $(f_k)$ of analytic solutions, such that for all $k$, $f_k$ is well defined up to time $T_k$, and such that
\begin{equation*}
\lim_{k \to + \infty} \frac{\| f_k - \mu \|_{L^2([0,T_k) \times \T^d \times \R^d)}}{\| f_k|_{t=0} - \mu \|^{\alpha}_{H^s_{x,v}}} = + \infty.
\end{equation*}
This time, $(T_k)$ is of order:
\begin{equation*}
T_k = \underset{k \to + \infty}{\mathcal{O}} \left( \frac{|\log \eps_k |}{|n_k|} \right), 
\end{equation*}
where $\eps_k := \| f_k|_{t=0} - \mu \|_{L^2_{x,v}}$ and $n_k$ is the spatial frequency of the nearest exponential growing mode. The solution $f_k$ is of size $\eps_k$ at time $0$ and close to an exponential growing mode of spatial frequency $n_k$ and of proportional growing rate, and once again the exponential growing rate of the solution to the linearized problem prevails.

This result is a quantitative extension of \cite[Theorem 4.1]{bardos2012vlasov} by Bardos and Nouri, where it is proved that \eqref{VB} is ill-posed from $H_{x,v}^m$ to $H_{x,v}^1$ for any $m \in \N^*$. 
\section{New result: the case of non-smooth stationary profiles}
\label{sec:new_results}

As already said, the aim of this paper is to generalize these results in the case when the velocity profile $\mu$ and the density $f$ are no longer smooth in the variable $v$ but only measures. We start by defining a notion of solution in that setting. These solutions are regular with respect to the time variable and the space variable and measures with respect to the velocity variable.
 \subsection{Measure-valued solutions}
 We will be dealing with functions $f: [0,T] \times \T^d \to \mathcal{P}(\R^d)$ which are smooth when integrated against smooth functions of the variable $v$. If $\varphi$ is a smooth and bounded function on $\R^d$, we define for all $t$ and $x$
\[
\cg f,\varphi\cd(t,x) := \int \varphi(v) f(t,x,\D v).
\]
The function $\cg f, \varphi \cd$ is called the \emph{macroscopic observable} corresponding to $\varphi$. The class of solutions to \eqref{VP} that we will consider is defined as follows.
\begin{Def}[Weak in $v$ and strong in $x$ solutions]
\label{def:ws_sol}
We will say that $f: [0,T] \times \T^d \to \mathcal{P}(\R^d)$ is a weak in $v$ and strong in $x$ solution to \eqref{VP} if it satisfies in the classical sense for all test function $\varphi$ the system
\begin{equation}
\label{VPweak}
\left\{ \begin{gathered}
\partial_t \cg f, \varphi \cd (t,x) + \Div \cg f, v\varphi \cd (t,x) + \nabla_x U(t,x) \cdot \cg f, \nabla \varphi \cd (t,x) = 0,\\
- \Delta_x U(t,x) = \cg f, 1 \cd (t,x) - 1,\\
f(0,x, \D v) =f_0(x, \D v).
\end{gathered} \right.
\end{equation} 
Equations \eqref{KEu'} and \eqref{VB} have straightforward similar formulations.
\end{Def}

This is motivated by the following fact. If $\mu$ is any probability measure, then it is a weak solution to \eqref{VP} (resp. \eqref{KEu} or \eqref{VB}). Moreover, an integration by parts leads to
 \begin{equation}
 \label{integrationbyparts}
 \int_{\R^d} \frac{n \cdot \nabla_v \mu(v)}{n \cdot (v - \omega)} \D v = |n|^2 \int_{\R^d} \frac{\D \mu(v) }{\{n\cdot(v - \omega)\}^2},
 \end{equation}
which only involves $\mu$ and not its derivatives. (We make no difference between the density $\mu$ and the measure it induces $\D \mu(v) = \mu(v) \D v$.) Therefore, the Penrose instability condition  of Definition \ref{def:Penrose_VP} (resp. \ref{def:Penrose_KEu} or \eqref{def:Penrose_VB}) makes sense for any probability measure $\mu$, and it is a natural question to know whether the stability can be studied around such profiles. 

To give examples of unstable profiles in this setting, we show in Appendix \ref{penrosediracs} that a superposition of a finite number of distinct Dirac masses is always unstable for the three physical models we have presented. This is coherent with the classical setting where profiles with one bump are stable and profiles with several sufficiently large and sufficiently distant bumps are unstable.

The natural question that is asked is the following: do there exist unstable weak solutions to \eqref{penroseVP}, \eqref{penroseKEu} and \eqref{penroseVB} in the neighbourhood of any probability measure $\mu$ that satisfies the corresponding Penrose condition. In the present paper, we answer affirmatively to this question. Let us state the results precisely.

 \subsection{Our new results}
 In the measure-valued setting, we are only able to evaluate the size of the solutions when integrated against smooth functions of $v$. So we will state the results in terms of macroscopic observables. These results might be understood as follows: whatever the number of macroscopic observables we control at the initial time in very strong norms, one specific macroscopic quantity  will be likely to grow along the flow of the equation even in weak norms. This macroscopic quantity will be the electric potential in the case of the Vlasov-Poisson equation, the pressure in the case of the kinetic Euler equation and the density in the case of the Vlasov-Benney equation. 
 \subsubsection*{Almost Lyapounov instability for Vlasov-Poisson} In this case, the result we show can be stated in the following way. 
 \begin{Thm}
\label{kineticstatementVP} 
Take $\mu$ an unstable profile, $N \in \N^*$, $\varphi_1, \dots, \varphi_N \in C_c^{\infty}(\R^d)$, $s \in \N$ and $\alpha \in (0,1]$. Then there exists, $(T_k)\in (\R_+^*)^{\N}$ and $(f_0^k)$ a family of measure-valued initial data such that:
\begin{itemize}
\item for all $k$, there is a weak in $v$ and strong in $x$ solution $f^k$ to \eqref{VP} starting from $f^k_0$ up to time $T_k$,
\item if we denote by $U_k$ the corresponding electric potential, we have:
\[
\frac{\|U_k \|_{L^1([0,T_k) \times \T^d)}}{\sum_{i=1}^N \| \cg f^k_0 , \varphi_i \cd - \cg \mu, \varphi_i\cd \|_{W^{s, \infty}(\T^d)}^{\alpha}}\underset{k \to + \infty}{\longrightarrow} + \infty.
\]
\end{itemize}
Moreover
\begin{equation*}
T_k \underset{k \to + \infty}{\sim}  |\log \eps_k | \quad \mbox{with} \quad \eps_k := \| U_k|_{t=0} \|_{L^1(\T^d)} .
\end{equation*} 
 \end{Thm}
Remark that there is no contribution of the stationary solution in the numerator because the electric potential of the stationary solution is $0$.
 
We could not prove with our method a Lyapounov instability result: in our proof, we build solutions that actually satisfy
 \[
 \|U_k \|_{L^1([0,T_k) \times \T^d)} \underset{k \to + \infty}{\longrightarrow} 0,
 \]
 whereas Lyapounov instability would correspond to the following property:
 \[
 \sum_{i=1}^N \| \cg f^k_0 , \varphi_i \cd - \cg \mu, \varphi_i\cd \|_{W^{s, \infty}(\T^d)} \underset{k \to + \infty}{\longrightarrow} 0,
 \]
 but:
\[
 \liminf_{k \to + \infty} \|U_k \|_{L^1([0,T_k) \times \T^d)} >0.
\]
This point will be developed in Remark \ref{remalmostLyapounov}.

In conclusion, our method makes it possible to deal with measure-valued solutions. It also allows to drop the so-called $\delta$ and $\delta'$-conditions in \cite{han2015stability} that we already talked about. But on the other hand, the instability result is a bit weaker than the one of Han-Kwan--Hauray in \cite{han2015stability} and Han-Kwan--Nguyen in \cite{han2016nonlinear}.
 
 \subsubsection*{Ill-posedness for kinetic Euler and Vlasov-Benney}
 The statement in these cases is similar to the previous one, but we can take a sequence $(T_k)$ tending to zero: the instabilities can develop arbitrarily fast.
  \begin{Thm}
\label{kineticstatementKEuVB} 
Take $\mu$ an unstable profile, $N \in \N^*$, $\varphi_1, \dots, \varphi_N \in C_c^{\infty}(\R^d)$, $s \in \N$ and $\alpha \in (0,1]$. Then there exists, $(T_k)\in (\R_+^*)^{\N}$ \textbf{tending to zero} and $(f_0^k)$ a family of measure-valued initial data such that:
\begin{itemize}
\item for all $k$, there is a weak in $v$ and strong in $x$ solution $f^k$ to \eqref{KEu} starting from $f^k_0$ up to time $T_k$,
\item if we denote by $p_k$ the corresponding pressure, we have:
\[
\frac{\|p_k \|_{L^1([0,T_k) \times \T^d)}}{\sum_{i=1}^N \| \cg f^k_0 , \varphi_i \cd - \cg \mu, \varphi_i\cd \|_{W^{s, \infty}(\T^d)}^{\alpha}}\underset{k \to + \infty}{\longrightarrow} + \infty.
\]
\end{itemize}
Moreover
\begin{equation*}
T_k \underset{k \to + \infty}{\sim} \left( \frac{|\log \eps_k |}{|n_k|} \right),
\end{equation*}
where $\eps_k := \| p_k|_{t=0} \|_{L^1}$ and $n_k$ is the spatial frequency of the nearest exponential growing mode.

The same result holds for \eqref{VB} instead of \eqref{KEu}, replacing the pressure by 
\[
\rho_k - 1,
\]
$\rho_k$ being the density of $f_k$.
 \end{Thm}
 
 In \cite{han16}, Han-Kwan and Nguyen show a similar result with the additional assumption that $\mu$ is analytic. So in that case, our result is a strict generalization. Once again, we can drop the $\delta$ or $\delta'$-condition.
 
 \section{Ideas of proof}
 Before going into the details of the proof, let us present how to build weak in $v$ and strong in $x$ solutions. The idea is to look for a particular class of solutions: the ones that admit a multiphasic decomposition. The weak in $v$ and strong in $x$ solutions that we will build will be induced by strong solutions to a different system. Let us explain this idea.  
 
 \subsection{A multiphasic representation}
 We will present in this subsection how to build weak in $v$ and strong in $x$ solutions to the Vlasov-Poisson equation \eqref{VP}. The other models can be treated the same way. Let us rewrite \eqref{VP} here for clarity: 
\begin{equation*}
\left\{ \begin{gathered}
\partial_t f(t,x,v) + v\cdot \nabla_x f(t,x,v) - \nabla_x U(t,x) \cdot \nabla_v f(t,x,v) = 0,\\
- \Delta_x U(t,x) = \int f(t,x,v) \D v - 1,\\
f(0,x,v)= f_0(x,v).
\end{gathered} \right.
\end{equation*}  
 Assume that the initial data can be decomposed into a superposition of smooth graphs (with densities): there exists $\X$ a polish space, $\nu$ a Borel probability measure on this set, $\rhorho_0 = (\rho^{\alpha}_0)_{\alpha \in \X}$ a family of smooth functions on $\T^d$ (the densities) and $\vv_0 = (v^{\alpha}_0)_{\alpha \in \X}$ a family of smooth vector fields on $\T^d$ (which provide the graphs), such that for all smooth and bounded function $\varphi$ and for all position $x$,
\[
\int \varphi(v) f_0(x, \D v) = \int \varphi(v^{\alpha}_0(x)) \rho^{\alpha}_0(x) \D \nu(\alpha).
\]
Also suppose that we are able to solve (say classically) the following system:
\begin{equation}
\label{multiphasic}
\left\{ \begin{gathered}
\forall \alpha \in \X, \quad \partial_t \rho^{\alpha}(t,x) + \Div ( \rho^{\alpha}(t,x) v^{\alpha}(t,x) ) = 0, \\
\forall \alpha \in \X, \quad \partial_t v^{\alpha}(t,x) + (v^{\alpha}(t,x) \cdot \nabla)v^{\alpha}(t,x) = - \nabla U(t,x),\\
- \Delta U(t,x) = \int \rho^{\alpha}(t,x) \D \nu(\alpha) - 1,\\
\forall \alpha \in \X, \quad \rho^{\alpha}|_{t=0} = \rho^{\alpha}_0 \mbox{ and } v^{\alpha}|_{t=0} = v^{\alpha}_0.
\end{gathered} \right.
\end{equation}
Then, at time $t$ and position $x$, we can define the measure $f(t,x,\bullet)$ though the macroscopic observable: for all $\varphi$ sufficiently smooth,
\begin{equation}
\label{multiphasictokinetic}
\int \varphi(v) f(t,x, \D v) = \cg f, \varphi \cd (t,x)  := \int \varphi(v^{\alpha}(t,x)) \rho^{\alpha}(t,x) \D \nu (\alpha).
\end{equation}
Straightforward computations show that this density is a weak in $v$ and strong in $x$ solution to \eqref{VP}, as defined in Definition \ref{def:ws_sol}.

Roughly speaking, the multiphasic representation corresponds to the case when the whole population of particles can be divided into distinguishable phases, each of which can be described by its pointwise density and velocity. According to the first equation in \eqref{multiphasic}, each density is transported by the corresponding velocity, according to the second one, each phase is accelerated by the same potential, and according to the third one, the potential is calculated by taking into account all the phases.

This formulation was already used by Grenier in \cite{grenier1996oscillations} to prove a small time existence result and to analyse precisely the quasineutral limit. In another direction, Brenier built in \cite{bre97} and \cite{bre99} some low regularity solutions to the kinetic Euler equations in this formulation as minimizers of the mean kinetic action
\[
\frac{1}{2}\int \hspace{-5pt} \int_0^T \hspace{-7pt} \int |v^{\alpha}(t,x)|^2 \rho^{\alpha}(t,x) \D x \D t \D \nu(\alpha)
\]
with prescribed $(\rho^{\alpha}|_{t=0})_{\alpha \in \X}$ and $(\rho^{\alpha}|_{t=T})_{\alpha \in \X}$. Let us point out that the ill-posedness for this kind of multifluid system answer questions that were left open in \cite[Introduction]{bre97} and \cite[Remark 2.2]{han16}.

Here, the specificity is the fact that we label the phases by the velocity space $\R^d$ in the following way. To any probability measure $\mu$ on $\R^d$ corresponds a stationary homogeneous solution to \eqref{VP}, \eqref{KEu}, \eqref{VB} (in their weak formulations of type \eqref{VPweak}) given by $f (x,\D v) = \D \mu(v)$. It has a smooth multiphasic decomposition indexed by $\mu$ itself: defining for all $w$, $\rho^w \equiv 1$ and $v^w \equiv w$, then for all admissible test function $\varphi$ and all position $x$,
\begin{equation}
\label{homogeneoussolutionmultiphasic}
\int \varphi(v) f(x,\D v) = \int \varphi(v^w(x)) \rho^w(x) \D \mu(w) = \int \varphi(v) \D \mu(v).
\end{equation}
Remark that these $\rhorho$ and $\vv$ are stationary solutions to \eqref{multiphasic} with $U\equiv0$. We can then ask the question of linear stability in this multiphasic formulation. Doing so, we will recover in Subsection \ref{spectralstudy} the Penrose condition. We provide an illustration on the notion of having a smooth multiphasic decomposition indexed by $\mu$ at Figure \ref{figmultiphasic}. In fact, in this paper, we will mainly work with multiphasic formulations. 
\begin{figure}
\centering
\begin{tikzpicture}[scale=0.8]
\draw[very thick, blue] (0,1) -- (4,1);
\draw[very thick, red] (0,2.5) -- (4,2.5);
\draw[very thick, violet] (0,3) -- (4,3);
\draw (0,0) rectangle (4,4);
\draw (-0.3,2) node{$v$};
\draw (2,-0.3) node{$x$};
\end{tikzpicture}
\hspace{0.3cm}
\begin{tikzpicture}[scale=0.8]
\draw[very thick, blue] (0,1) .. controls (1,2) and (3,0) .. (4,1);
\draw[very thick, red] (0,2.5) .. controls (2.5,2) and (3,4) .. (4,2.5);
\draw[very thick, violet] (0,3) .. controls (1,2) and (3,4) .. (4,3);
\draw (0,0) rectangle (4,4);
\draw (-0.3,2) node{$v$};
\draw (2,-0.3) node{$x$};
\end{tikzpicture}
\hspace{0.3cm}
\begin{tikzpicture}[scale=0.8]
\draw[very thick, blue] (0,1) to[out=0, in=180] (1,0.7);
\draw[very thick, blue] (1,0.7) to[out=0, in=180] (2,1.5);
\draw[very thick, blue] (2,1.5) to[out=0, in=180] (3,1.2);
\draw[very thick, blue] (1,0.7) to[out=0, in=180] (2,0.2);
\draw[very thick, blue] (2,0.2) to[out=0, in=180] (3,1.2);
\draw[very thick, blue] (3,1.2) to[out=0, in=180] (4,1);
\draw[very thick, red] (0,2.5) to[out=0, in=180] (1.7,3);
\draw[very thick, red] (1.7,3) to[out=0, in=170] (1.3,2);
\draw[very thick, red] (1.3,2) to[out=-10, in=180] (4,2.5);
\draw[very thick, violet] (0,3) .. controls (1,4) .. (4,3);
\draw (0,0) rectangle (4,4);
\draw (-0.3,2) node{$v$};
\draw (2,-0.3) node{$x$};
\end{tikzpicture}
\caption{\label{figmultiphasic} On the left, a stationary homogeneous density. Here $\mu$ is a superposition of three Diracs. In the middle, another density that has a smooth multiphasic structure indexed by $\mu$. The velocities depend on the position and the density on each graph may not be uniform anymore. On the right, the density does not have a smooth multiphasic structure indexed by $\mu$. The velocity of the two lower "phases" are no longer graphs. To get a multiphasic decomposition, we should add some labels.}
\end{figure}
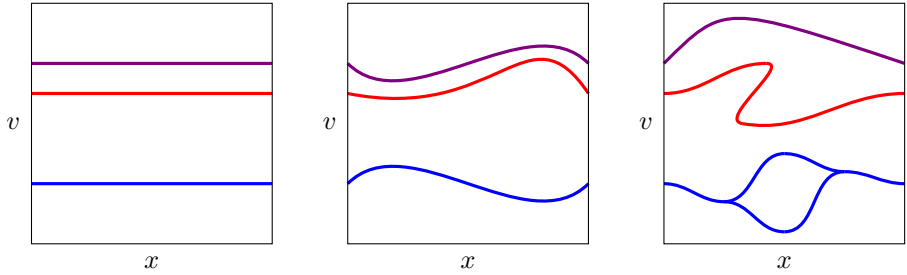

We are now ready to describe briefly the structure of the proof.

 \subsection{Sketch of the proof}
\subsubsection*{Analytic regularity with respect to the position} The proof consists in studying the linearized multiphasic system to get an estimate on the corresponding semigroup, and then to use this estimate to get a nonlinear solution through a fixed point argument. As in the works \cite{grenier1996oscillations, han2015stability, han16}, we work in an analytic framework. The densities and velocity fields in the multiphasic formulation will be analytic functions of $x$. This is the relevant level of regularity to handle the fact that in the kinetic Euler equation and in the Vlasov-Benney equation, the force field ($- \nabla p$ and $-\nabla \rho$ respectively) are one derivative less regular than the density. So for instance, there is no hope \emph{a priori} to perform a fixed point proof of existence in any Sobolev space (besides, our ill-posedness result makes the feasibility of such proof very unlikely). In our work, this lack of regularity will appear in the fact that the semi-group of the linearized operator will be continuous only in analytic functional spaces.
 
\subsubsection*{Outline of the paper} 
Let us present the content of each section of the paper.
\paragraph{\underline{Section \ref{model}}.} We introduce the abstract multiphasic model we will work with, and the assumptions we make to perform the analysis. The three examples presented in Subsection \ref{physicalmodels} in their multiphasic formulations are particular cases of this model. The homogeneous solutions presented in \eqref{homogeneoussolutionmultiphasic} are still stationary solutions in this framework.

\paragraph{\underline{Section \ref{sectionlinear}}.} We study the linearization of the abstract model around these homogeneous stationary solutions. This section is divided in two parts: in Subsection \ref{spectralstudy}, we compute the unstable eigenvectors and eigenvalues of the linearized system, and in Subsection \ref{sharpsemigroupbound}, we derive some sharp estimates for the corresponding semigroup in analytic regularity following \cite{han16}. These estimates are crucial to get sufficiently large times of existence for the instabilities to develop (see the beginning of Subsection \ref{sharpsemigroupbound} for more detail).

\paragraph{\underline{Section \ref{nonlinearinstability}}.} We show that their exist analytic solutions to the abstract model of the form
\[
\mbox{stationary solution} + \eta f + \mbox{remainder}
\]
where $\eta$ is a small parameter, $f$ is a solution to the linearized system (typically an exponential growing mode of spatial frequency $n$) and the remainder is small with respect to $\eta$ and cancels at $t=0$. We also bound from below the time of existence of such solutions with respect to $\eta$ and $n$ using the estimate derived in the previous section. This is done at Theorem \ref{existenceofsolutions} which is the main result of this paper. The strategy is the same as in \cite{grenier1996oscillations} and \cite{han16}: we decompose the operator as a linear term and an at least quadratic term, and we consider the latter as a source term in a Duhamel formulation. After a fine analysis of the properties of the analytic norms we use (Subsection \ref{propertiesanalyticnorms}), and of the size of each term in the Duhamel formulation (Subsection \ref{estimatesduhamelformulation}), we can perform at Subsection \ref{cauchykovalevskaiatheorem} a fixed point argument as in Caflish' proof of the Cauchy-Kovalevskaia theorem (see \cite{caflisch1990simplified}).

\paragraph{\underline{Section \ref{consequences}}.} We show how to deduce from these existence results Theorem \ref{kineticstatementVP} and Theorem \ref{kineticstatementKEuVB}. Theorem \ref{almostLyapounovThm} asserts that the Penrose instability condition always implies almost Lyapounov instability in the abstract multiphasic model. The only thing we need to do is to use the form of the eigenvalues and the estimates obtained in Theorem \ref{existenceofsolutions} to evaluate precisely the size of the initial data in Sobolev types norms and of the solutions in Lebesgue type norms. Corollary \ref{almostLyapounovkinetic} is a kinetic version of Theorem \ref{almostLyapounovThm} and directly implies Theorem \ref{kineticstatementVP}.

On the other hand, ill-posedness around Penrose unstable profiles only holds in the abstract multiphasic model when a further assumption is made on the spectrum of the linearized operator. This is the content of Theorem \ref{illposednesstheorem}, and of Corollary \ref{illposednesskinetic}, its kinetic counterpart. This assumption is true in the kinetic Euler equation and in the Vlasov-Benney equation thanks to their scaling properties already discussed in Subsection \ref{presentationpenrosecondition}. Apart from this new ingredient, the proof is very similar to the one of Theorem \ref{kineticstatementVP}. However, if Corollary \ref{illposednesskinetic} directly implies Theorem \ref{kineticstatementKEuVB} in the Vlasov-Benney case, we need to work a little bit more to adapt it to the case of the kinetic Euler equation. The reason is the fact that the initial data of the exponential growing modes we build in the abstract setting do not satisfy the incompressibility constraint. In Subsection \ref{sectionkineticEuler}, we present how to fix this problem, and thus how to prove Theorem \ref{kineticstatementKEuVB} in the case of the kinetic Euler equation.

\paragraph{\underline{Appendices}.} In Appendix \ref{penrosediracs}, we show that any superposition of at least two Dirac masses is unstable for the three physical models. In Appendix \ref{appendixproof}, we give the proofs of the properties of the analytic norms stated in Subsection \ref{propertiesanalyticnorms}. 
\section{Presentation of the abstract model}
\label{model}
Let us describe the model we will study throughout the paper.
\subsection{The abstract model}
 First, we model the evolution of several phases indexed by a probability measure $\mu$ on $\R^d$ and described by their densities $(\rho^w)_{w \in \R^d}$ and velocity fields $(v^w)_{w \in \R^d}$ which are functions of time $t \in \R_+$ and position $x \in \T^d$. The torus is normalized, so that the total mass of its Lebesgue measure is supposed to be equal to one. The notation $\rhorho(t)$ and $\vv(t)$ will stand for the whole families $(\rho^w(t, \bullet))_{w \in \R^d}$ and $(v^w(t, \bullet))_{w \in \R^d}$. These phases follow the Newton dynamics in a potential $U$: 
\begin{equation}
\label{abstracteq}
\left\{ \begin{gathered}
\forall w \in \R^d, \quad \partial_t \rho^{w}(t,x) + \Div ( \rho^{w}(t,x) v^{w}(t,x) ) = 0, \\
\forall w \in \R^d, \quad \partial_t v^{w}(t,x) + (v^{w}(t,x) \cdot \nabla)v^{w}(t,x) = - \nabla U[\rhorho(t), \vv(t)](x),\\
\forall w \in \R^d, \quad \rho^{w}|_{t=0} = \rho^{w}_0 \mbox{ and } v^{w}|_{t=0} = v^{w}_0.
\end{gathered} \right.
\end{equation}
We need now to describe how the phases generate the potential. We suppose it is in the following form:
\begin{equation}
\label{abstractforce}
U[\rhorho, \vv](x) := A\left[  \int \Phi(v^w) \rho^w \D \mu(w) \right](x).
\end{equation}
where: $\Phi: \R^d \to E$ is a smooth function, $E$ is a normed $\R$-vector space with finite dimension and $A$ is a homogeneous Fourier multiplier of symbol $P: \Z^d \to \mathcal{L}(E^{\mathbb{C}}; \C) = \mathcal{L}(E; \R)^{\mathbb{C}}$. (The notations $E^{\mathbb{C}}$ and $ \mathcal{L}(E; \R)^{\mathbb{C}}$ stand for the complexifications of $E$ and $\mathcal{L}(E; \R)$ respectively.)
\begin{Rem}[Stationary solution]
 Defining for all $w$ $\rho^w \equiv 1$ and $v^w \equiv w$, the potential is well defined thanks to \eqref{intPhi}, its gradient vanishes, and we get a stationary solution (corresponding to the stationary homogeneous solution in Vlasov-Poisson).
\end{Rem} 
\begin{Rem}
It would be natural to solve the two first lines of \eqref{abstracteq} only for $\mu$-almost all $w$. In addition, it could seem artificial to prescribe initial conditions for all $w$ and not only for $\mu$-almost all $w$ because it would mean describing the distribution of the particles belonging to a phase that does not contain any particle. However, in this paper, we build solutions starting from very specific initial conditions (the eigenvectors of the linearized operators around homogeneous solutions) that have a meaning for all $w$. So we will indeed solve \eqref{abstracteq} for all $w$.
\end{Rem}

\subsection{Gradient structure} We will solve the system \eqref{abstracteq}-\eqref{abstractforce} for a particular class of initial data, where the total mass is the same as the one of the stationary solution, and where the velocity is a gradient.

Formally, a solution to \eqref{abstracteq}-\eqref{abstractforce} of the form 
\[
(\rhorho(t), \vv(t)) = (1 + r^w(t), w + u^w(t))_{w \in \R^d}
\]
with
\[
\forall w \in \R^d, \quad \int_{\T^d} r^w(0,x) \D x = 0 \quad \mbox{and} \quad u^w(0) \mbox{ is a gradient}
\]
keeps this structure along the flow: we expect that for all $t$ for which the solution exists, 
\[
\forall w \in \R^d, \quad \int_{\T^d} r^w(t,x) \D x = 0 \quad \mbox{and} \quad u^w(t) \mbox{ is a gradient}.
\]
We will see in the sequel that this is true for our solutions. We give a name to this type of families of functions.
\begin{Def}
\label{defL0}
Let $(\rr, \uu) = (r^w, u^w)_{w \in \R^d}$ a family of pairs of analytic functions. We write $(\rr, \uu) \in \boldsymbol{L}_0$ if
\[
\forall w \in \R^d, \quad \int_{\T^d} r^w(x) \D x = 0 \quad \mbox{and} \quad u^w \mbox{ is a gradient}.
\]
\end{Def}

\subsection{Assumptions}
Let us give a few assumptions to be made to perform the analysis. We will need several quantities depending on $A$, $\Phi$ and $\mu$ to be finite. We will take a large number $M>0$ that bounds all of them.
\begin{Ass}[Assumptions on $\mu$ and $\Phi$] We suppose that $\Phi$ is a power series on $\R^d$, \emph{i.e.} there is $(a_k)_{k \in \N^d} \in E^{\N^d}$ such that for all $w \in \R^d$,
\begin{equation}
\label{defPhi}
\Phi(w) = \sum_{k \in \N^d} w^k a_k
\end{equation}
where if $k = (k_1, \dots, k_d)$ and $w = (w_1, \dots, w_d)$, $w^k$ stands for the real number $w_1^{k_1} \times \dots \times w_d^{k_d}$. We will also use the notation $|k| := k_1 + \dots +k_d$. Moreover, we suppose that there exists $r_0>0$ such that the following quantities are finite and bounded by $M$:
\begin{gather}
\label{intPhi} \int |\Phi(w)| \D \mu(w) \leq M, \\
\label{intdPhi} \int |\D \Phi(w)| \D \mu(w) \leq M,\\
\label{intdPhiseries} \sum_{k \in \N^d, |k|\geq 1} |a_k| |k| \int (|w| + r_0)^{|k|-1} \D \mu(w) \leq M , \\
\label{intd2Phiseries} \sum_{k \in \N^d, |k|\geq 2} |a_k| |k|(|k| - 1) \int (|w| + r_0)^{|k|-2} \D \mu(w) \leq M.
\end{gather}
\end{Ass}
These quantities are linked together: for instance, \eqref{intdPhiseries} clearly implies \eqref{intdPhi}. However we will not develop much these links, especially since in all the physical models presented in the introduction, $\Phi$ is polynomial, and in that case, all these estimates hold with $M$ big enough as soon as
\[
\int |w|^p \D \mu(w) < + \infty
\]
where $p$ is the degree of $\Phi$. We will nevertheless write the proof for analytic $\Phi$ because the estimates are the same as in the polynomial case.
\begin{Ass}[Assumptions on $P$] We suppose that $A$ is real, which means in terms of its symbol $P$:
\begin{equation}
\label{Areal}
\forall n \in \Z^d, \qquad P(-n) = \overline{P(n)},
\end{equation}
where the conjugate is understood via the identification $\mathcal{L}(E^{\mathbb{C}}; \C) = \mathcal{L}(E; \R)^{\mathbb{C}}$. 

We also suppose that $P$ is uniformly bounded:
\begin{equation}
\label{assumptionP}
\sup_{n \in \Z^d} |P(n)| \leq M.
\end{equation}
\end{Ass}
This assumption will be crucial for the semi-group of the linearized operator to be continuous at our level of analytic regularity. It means that the force field should not involve more than one derivative of the macroscopic observable
\[
\int \Phi(v^w) \rho^w \D \mu(w).
\]

\subsection{Examples}
Let us give $E$, $\Phi$ and $P$ in our physical models.
\begin{itemize}
\item \textbf{The Vlasov-Poisson case.} Equation \eqref{VP} has a straightforward multiphasic formulation of the form \eqref{abstracteq}-\eqref{abstractforce}: we take $E = \R$, $\Phi \equiv 1$, $P(0)=0$ and 
\[
\forall n \in \Z^d\backslash\{0\}, \quad P(n) = \frac{1}{ |n|^2}.
\]
\item \textbf{The kinetic Euler case.} Equation \eqref{KEu} does not have \emph{a priori} a multiphasic version of the form \eqref{abstracteq}-\eqref{abstractforce} because the pressure field is not given by a formula as in \eqref{abstractforce}. In fact, we can derive one.  the Cauchy problem \eqref{KEu} makes sense only when $f_0$ satisfies some additional properties. First, of course, it must satisfy the incompressibility constraint. Furthermore, if one integrates formally the first equation with respect to $v$, one gets because of the constraint and the fact that the pressure does not depend on the velocity
\begin{equation}
\label{nulldivergence}
\Div_x \left( \int v f(t,x,v) \D v \right) = 0.
\end{equation}
This means that the macroscopic velocity is divergence-free. This property must hold at time $t=0$. Consequently, $f_0$ must satisfy 
\begin{equation}
\label{initialdataconstraint}
\begin{gathered}
\int f_0(x,v) \D v \equiv 1,\\
\Div_x \left( \int v f_0(x,v) \D v \right) \equiv 0.
\end{gathered}
\end{equation}

Then, multiplying the equation by $v$, integrating it with respect to $v$ and finally taking the divergence leads to
\begin{equation*}
- \Delta_x p(t,x) = \Div_x \bDiv_x \left( \int v \otimes v f(t,x,v) \D v \right).
\end{equation*}
In fact, in what follows, we will be dealing directly with:
\begin{equation}
\label{KEu'}
\left\{ \begin{gathered}
\partial_t f(t,x,v) + v\cdot \nabla_x f(t,x,v) - \nabla_x p(t,x) \cdot \nabla_v f(t,x,v) = 0,\\
- \Delta_x p(t,x) = \Div_x \bDiv_x \left( \int v \otimes v f(t,x,v) \D v \right), \\
f(0,x,v) =f_0(x,v),
\end{gathered}\right.
\end{equation}
and justify in Subsection \eqref{sectionkineticEuler} that when $f_0$ satisfies \eqref{initialdataconstraint}, then our solutions are indeed solutions to \eqref{KEu}. 

Now \eqref{KEu'} has a multiphasic version of the form \eqref{abstracteq}-\eqref{abstractforce}: it suffices to take $E = \mathcal{M}_d(\R)$, $\Phi: v \mapsto v \otimes v$, $P(0)=0$ and
\[
\forall n \in \Z^d\backslash\{0\},\, \forall X \in \mathcal{M}_d(\R), \quad P(n)\cdot X = -\frac{\cg n, X \cdot n \cd}{|n|^2}.
\]
\item \textbf{The Vlasov-Benney case.} In the case of the Vlasov-Benney equation \eqref{VB}, we take $E = \R$, $\Phi \equiv 1$ and for all $n \in \Z^d$,
\[
\forall n \in \Z^d, \quad P(n) = 1.
\]
\end{itemize}
In these three cases, all the assumptions are easy to check.
\section{The linearized system}
\label{sectionlinear}
In this section and in the following one, we study the multiphasic system \eqref{abstracteq} governed by the potential defined in \eqref{abstractforce} with the assumptions \eqref{intPhi}, \eqref{intdPhi} and \eqref{assumptionP}. In this setting, defining for all $w$ $\rho^w \equiv 1$ and $v^w \equiv w$ leads to a stationary solution. The linearized system around this stationary solution is
\begin{equation}
\label{abstractlineq}
\left\{ \begin{gathered}
\forall w \in \R^d, \quad \partial_t r^{w}(t,x) + w \cdot \nabla r^{w}(t,x) + \Div ( u^w(t,x)) = 0, \\
\forall w \in \R^d, \quad \partial_t u^{w}(t,x) + (w \cdot \nabla)u^{w}(t,x) = - \nabla V[\rr(t), \uu(t)](x),\\
V[\rr, \uu](x) := A\left[  \int \left\{ \Phi(w) r^w + \D \Phi(w)\cdot u^w \right\} \D \mu(w) \right](x), \\
\forall w \in \R^d, \quad r^{w}|_{t=0} = r^{w}_0 \mbox{ and } u^{w}|_{t=0} = u^{w}_0.
\end{gathered} \right.
\end{equation}
\subsection{Spectral analysis}
\label{spectralstudy}
We look for the exponential growing modes of system \eqref{abstractlineq} \emph{i.e.} the non-zero solutions of the form
\begin{equation}
\label{exponentialgrowingmode}
\left\{
\begin{aligned}
r^w(t,x) &= f(w) \exp(\lambda t)\exp( in \cdot x ),\\
u^w(t,x) &= g(w)  \exp(\lambda t)\exp( in \cdot x ).
\end{aligned}
\right.
\end{equation}
with $n \in \Z^d$, $\lambda \in \C$ such that $\Re(\lambda)>0$, and $f: \R^d \to \C$ and $g: \R^d \to \C^d$ in $L^{\infty}(\mu)$ (for $V$ to be well defined thanks to \eqref{intPhi} and \eqref{intdPhi}). Injecting this ansatz in \eqref{abstractlineq}, we get that for all $w \in \R^d$,
\[
\left\{
\begin{aligned}
(\lambda + in \cdot w) f(w) &=-i n \cdot g(w),\\
(\lambda + in \cdot w) g(w) &= -i \left( P(n) \cdot \int \Big\{ \Phi(w') f(w') + \D \Phi(w') \cdot g(w') \Big\} \D \mu(w') \right) n.
\end{aligned}
\right.
\]
As a consequence, if $(\rr, \uu)$ is a non-trivial solution, then $n \neq 0$ and 
\[
P(n) \cdot \int \Big\{ \Phi(w') f(w') + \D \Phi(w') \cdot g(w') \Big\} \D \mu(w') \neq 0.
\]

Up to dividing $f$ and $g$ by this number, we can suppose that it is equal to $1$. Then, we get for all $w$,
\begin{equation} 
\label{deffg}
f(w) =-\frac{|n|^2}{(\lambda + in\cdot w)^2} \quad \mbox{and} \quad g(w) = \frac{-i n}{\lambda + in\cdot w}.
\end{equation}
Such $f$ and $g$ are bounded.

Then, setting for all $w\in \R^d$
\[
\Psi(n,\lambda,w) := \frac{\Phi(w)}{\lambda + in\cdot w},
\]
we require:
\begin{align*}
1 &= P(n) \cdot \int \Big\{ \Phi(w) f(w) + \D \Phi(w) \cdot g(w) \Big\} \D \mu(w)\\
&= P(n) \cdot \int \left\{ -\Phi(w)\frac{|n|^2}{(\lambda + in\cdot w)^2} -i \frac{\D \Phi(w) \cdot n}{\lambda + in\cdot w} \right\} \D \mu(w) \\
&= -i P(n) \cdot \int \left\{ -i\Phi(w)\frac{|n|^2}{(\lambda + in\cdot w)^2} + \frac{\D \Phi(w) \cdot n}{\lambda + in\cdot w} \right\} \D \mu(w) \\
&= -i P(n) \cdot \int \partial_w \Psi(n,\lambda,w) \cdot n \D \mu(w)\\
&=-i \left(\int \partial_w \left\{\frac{ P(n) \cdot \Phi(w)}{\lambda + in\cdot w}\right\}\cdot n \D \mu(w) \right) .
\end{align*}
In particular, we get the following general Penrose condition:
\begin{equation}
\label{penrosecondition}
\int \partial_w \left\{\frac{ P(n) \cdot \Phi(w)}{\lambda + in\cdot w}\right\}\cdot n \D \mu(w)   = i,
\end{equation}
for instability to hold.

Conversely, if \eqref{penrosecondition} holds for some $n \in \Z^d$ and $\Re(\lambda)>0$, and if we define $f$ and $g$ by \eqref{deffg}, then the exponential growing modes \eqref{exponentialgrowingmode} are (classical, unstable) solutions to the linearized equation \eqref{abstractlineq}.

In the end, we have proved the following proposition.
\begin{Prop}
System \eqref{abstractlineq} admits exponential growing modes if and only if there exists $n \in \Z^d$ and $\lambda \in \C$ with $\Re(\lambda) >0$ satisfying \eqref{penrosecondition}. In that case, $f$ and $g$ are given up to a scalar by \eqref{deffg}.
\end{Prop}

Consequently, we define what is an unstable profile in the following way.
\begin{Def}[Unstable profile]
	\label{def:unstability} We say that the probability measure $\mu$ on $\R^d$ is unstable if there exist $n \in \Z^d$ and $\lambda \in \C$ with $\Re(\lambda)>0$ such that \eqref{penrosecondition} holds.
\end{Def}

\begin{Rem} \label{eigenpotential}
When \eqref{penrosecondition} holds, with this choice of $f$ and $g$, the potential takes a very simple form:
\begin{align*}
V[&\rr(t), \uu(t)](x)\\
& = \left( P(n) \cdot \int \Big\{ \Phi(w') f(w') + \D \Phi(w') \cdot g(w') \Big\} \D \mu(w') \right) \exp(\lambda t) \exp(in\cdot x)\\
&= \exp(\lambda t) \exp(in\cdot x).
\end{align*}
Hence, it is natural to use this quantity to evaluate the size of our solutions, as it is done in Theorem \ref{kineticstatementVP} and Theorem \ref{kineticstatementKEuVB}.
\end{Rem}
\subsubsection*{Examples}\begin{itemize}
\item In the case of the Vlasov-Poisson equation \eqref{VP}, \eqref{penrosecondition} reads
\begin{equation}
\label{penroseVPmeasure}
\int \frac{\D \mu(w)}{(\lambda + in \cdot w)^2} = - 1.
\end{equation}
as expected by the combination of \eqref{penroseVP} and \eqref{integrationbyparts} (with the correspondence $\lambda = -in\cdot \omega$). In particular, the multiphasic formulation and the kinetic one have the exact same unstable eigenvalues.
\item In the case of the kinetic Euler system \eqref{KEu'}, \eqref{penrosecondition} reads
\begin{align*}
i&= \frac{1}{|n|^2} \int \partial_w \left\{ \frac{(in\cdot w)^2}{\lambda + in\cdot w}\right\}\cdot n \D \mu(w) \\
&= \frac{1}{|n|^2} \int \left( \frac{-2|n|^2 n \cdot w}{\lambda + in \cdot w} + \frac{i |n|^2 (n \cdot w)^2}{(\lambda + in\cdot w)^2} \right) \D \mu(w) \\
&= i \int \frac{2(\lambda + in \cdot w) i n \cdot w + (n\cdot w)^2}{(\lambda + in \cdot w)^2}\D \mu(w) \\
&= i \int \frac{2i \lambda n \cdot w - (n\cdot w)^2}{(\lambda + in \cdot w)^2}\D \mu(w) \\
&= i \int \left( 1 - \frac{\lambda^2}{(\lambda + i n\cdot w)^2} \right)\D \mu(w)\\
&= i - \lambda^2 i \int \frac{\D \mu(w)}{(\lambda + i n \cdot w)^2}.
\end{align*}
So we get the expected Penrose condition in this context (the combination of \eqref{penroseKEu} and \eqref{integrationbyparts})
\begin{equation}
\label{penroseKEumeasure}
\int\frac{ \D \mu(w) }{(\lambda + in \cdot w)^2} = 0.
\end{equation}
\item Finally, in the case of the Dirac-Benney system \eqref{VB}, similar computations show that \eqref{penrosecondition} reads
\begin{equation}
\label{penroseVBmeasure}
\int\frac{ \D \mu(w) }{(\lambda + in \cdot w)^2} = -\frac{1}{|n|^2}.
\end{equation}
\end{itemize}
\subsection{Sharp semigroup bounds}
\label{sharpsemigroupbound}
In this subsection, we will derive sharp estimates in analytic regularity for the semigroup corresponding to system \eqref{abstractlineq}. The philosophy for this result is the following: to build solutions to \eqref{abstracteq}-\eqref{abstractforce}, we will consider the nonlinear part of the system as a perturbation of the linear part. As long as the linear part of the solution is small, we will be able to deduce that the perturbation is even smaller and to perform a fixed point proof. So we want the estimate on the semigroup to be sharp for the fixed point argument to work until the longest possible times. We work in analytic regularity because in general, the only bound that we can get for the spectrum of the linearized operator is the fact that the unstable spectrum increases proportionally with the frequency of the exponential growing modes, as stated in Proposition \ref{gamma0finite} below.

For each $n \in \Z^d$, we call
\begin{equation}
\label{defSn}
S_n := \{ \lambda \in \C \mbox{ such that } \Re(\lambda) >0 \mbox{ and the Penrose condition }\eqref{penrosecondition} \mbox{ holds} \}.
\end{equation}
We already saw that $S_0$ is empty.

We also call
\begin{equation}
\label{defLambda}
\begin{aligned}
\Lambda_n(\lambda) &:= \int \partial_w \left\{\frac{ P(n) \cdot \Phi(w)}{\lambda + in\cdot w}\right\}\cdot n \D \mu(w)\\
&=i P(n) \cdot \int \left\{ -\Phi(w)\frac{|n|^2}{(\lambda + in\cdot w)^2} -i \frac{\D \Phi(w) \cdot n}{\lambda + in\cdot w} \right\} \D \mu(w) .
\end{aligned}
\end{equation}

The first observation to be made is that under condition \eqref{assumptionP}, the size of $S_n$ does not grow too much with $n$. More precisely, we have the following proposition.
\begin{Prop}
\label{gamma0finite}
We have:
\[
\sup_{n \in \Z^d \backslash \{0\}} \sup_{\lambda \in S_n} \frac{\Re(\lambda)}{|n|} < + \infty. 
\]
\end{Prop}
\begin{proof}
For $n \in \Z^d \backslash\{0\}$, we call
\[
\widetilde{S_n} := \{ l \in \C \mbox{ such that } |n| l \in S_n \}.
\]
Now for all $n \in \Z^d \backslash \{0\}$, for all $w \in \R^d$ and for all $l \in \C$ such that $\Re(l) >0$, using \eqref{assumptionP},
\begin{align*}
\left| \partial_w \left\{ \frac{P(n) \cdot \Phi(x)}{|n|l + i n \cdot w} \right\}\cdot n \right| & \leq \big|P(n) \big| \left\{ \frac{|\D \Phi(w)|}{\Re(l)} + \frac{|\Phi(w)|}{\Re(l)^2 } \right\} \\
&\leq M\left\{ \frac{|\D \Phi(w)|}{\Re(l)} + \frac{|\Phi(w)|}{\Re(l)^2} \right\}
\end{align*}
In particular, integrating with respect to $\mu$ and using \eqref{intPhi}, \eqref{intdPhi} and \eqref{defLambda} leads to
\[
 |\Lambda_n(|n|l)| \leq \frac{M^2}{\Re(l)} + \frac{M^2}{\Re(l)^2} \underset{\Re(l) \to +\infty}{\longrightarrow} 0.
\]
As a consequence, there is $C>0$ (independent of $n$) such that if $\Re(l) \geq C$, then for any $n \in \Z^d\backslash\{0\}$, the modulus of $\Lambda_n(|n|l)$ is lower than $1/2$, and so the Penrose condition \eqref{penrosecondition} cannot hold with $\lambda = |n|l$.
\end{proof}

Form now on, we suppose that $\mu$ is unstable, that is $\cup_{n  \in \Z^d} S_n \neq \emptyset$. We set
\begin{equation}
\label{defgamma0}
\gamma_0 := \sup_{n \in \Z^d \backslash\{0\}} \sup_{\lambda\in S_n} \frac{\Re(\lambda)}{|n|}>0.
\end{equation}

We want to show some bounds on the semigroup related to system \eqref{abstractlineq} in analytic regularity. To do so, following Grenier in \cite{grenier1996oscillations}, we introduce the following Banach spaces of analytic functions (in $x$). Take $\delta >0$, and $f$ a function on $\T^d$. We say that $f$ belongs to $X_{\delta}$ if it can be written for all $x \in \T^d$:
\[
f(x) = \sum_{n \in \Z^d} \hat{f}_n \exp(in\cdot x),
\]
with
\begin{equation*}
|f|_{\delta} := \sum_{n \in \Z^d}|\hat{f}_n| \exp(\delta |n|) < +\infty.
\end{equation*}
Remark that this formula makes sense for $f$ with value in $\R$, $\R^d$ or $E$ taking for $|\bullet|$ any norm on the corresponding vector space. So with a slight abuse of notations, we will still write $f \in X_{\delta}$ in all these cases.

Now, if $\FF = (f^w)_{w \in \R^d}$ is a family of functions on $\T^d$, we say that $\FF$ belongs to $\XX_{\delta}$ if it can be written for all $w \in \R^d$ and $x \in \T^d$:
\[
f^w(x) = \sum_{n \in \Z^d} \hat{f}_n(w) \exp(i n \cdot x)
\]
with for all $n \in \Z^d$, $\hat{f}_n \in L^{\infty}(\R^d, \mu)$ and with
\begin{equation}
\label{defdoublenorm}
\| \FF \|_{\delta} := \sum_{n \in \Z^d} |\hat{f}_n|_{\infty} \exp(\delta |n|) < + \infty.
\end{equation}
Once again, we keep the same notations for the values of $f$ to be in $\R$, $\R^d$ or $E$.

Finally, to gain space, if $\rr = (r^w)_{w \in \R^d} \in \XX_{\delta}$ is a family of functions from $\T^d $ to $\R$ and $\uu = (u^w)_{w \in \R^d}\in \XX_{\delta}$ is a family of functions from $\T^d$ to $\R^d$, we write
\begin{equation}
\label{defnormcouple}
\|(\rr, \uu)\|_{\delta} := \max \big(\| \rr \|_{\delta}, \| \uu \|_{\delta}\big).
\end{equation}
More generally, if $\FF$ and $\GG$ are two families of functions, $\| \FF, \GG \|_{\delta}$ will stand for the $\max$ between $\| \FF \|_{\delta}$ and $\| \GG \|_{\delta}$.

Remark that all the exponential growing modes found in the previous subsection belong for all $t$ and all $\delta$ to $\XX_{\delta}$.

We also write
\[
\boldsymbol{L}^{\infty} := L^{\infty}((\R^d , \mu); \R) \times L^{\infty}((\R^d , \mu); \R^d).
\]
Its norm is defined by
\[
\forall (\hat{r},\hat{u}) \in \boldsymbol{L}^{\infty}, \quad |(\hat{r},\hat{u})|_{\infty} := \max( |\hat{r}|_{\infty} , |\hat{u}|_{\infty}\big).
\]
With the same notation as before, we can see that
\[
\frac{1}{2}\sum_{n \in \Z^d}|(\hat{r}_n, \hat{u}_n)|_{\infty} \exp(\delta |n|) \leq \|(\rr, \uu)\|_{\delta} \leq \sum_{n \in \Z^d}|(\hat{r}_n, \hat{u}_n)|_{\infty} \exp(\delta |n|).
\]

The rest of this subsection will be devoted to the proof of the following theorem.
\begin{Thm}
\label{sharpestimate}
Let $\delta_0>0$ and $(\rr_0, \uu_0) \in \XX_{\delta_0}$. There exists a unique classical solution $(\rr(t), \uu(t))$ to \eqref{abstractlineq} at time $t\in[0, \delta_0 / \gamma_0)$ starting from $(\rr_0, \uu_0)$. It satisfies the following properties:
\begin{itemize}
	\item for all $\gamma > \gamma_0$, there exists $C$ only depending on $M$ and $\gamma$ (and not $\delta_0$ nor $(\rr_0, \uu_0)$) such that for all $t\leq \delta_0/ \gamma$,
	\[
	\big\| (\rr(t), \uu(t)) \big\|_{\delta_0 - \gamma t} \leq C \big\|(\rr_0, \uu_0)\big\|_{\delta_0}.
	\]
	\item for all $\delta < \delta_0$, the map 
	\begin{equation*} 
	t \in \left[0, \frac{\delta_0 - \delta}{\gamma_0} \right) \mapsto (\rr(t), \uu(t)) \in \XX_{\delta}
	\end{equation*} is continuous,
	\item if $(\rr_0, \uu_0) \in \boldsymbol{L}_0$ (defined in Definition \ref{defL0}), then this property is propagated: for all $t < \delta_0 / \gamma_0$, $(\rr(t), \uu(t)) \in \boldsymbol{L}_0$.
\end{itemize}
\end{Thm}
\begin{Rem}
If we write
\begin{equation}
\label{defS}
(\rr(t), \uu(t)) =: S_t(\rr_0, \uu_0),
\end{equation}
then the theorem shows that
\begin{equation}
\label{estimS}
\big\| S_t (\rr_0, \uu_0) \big\|_{\delta_0 - \gamma t} \leq C \big\|(\rr_0, \uu_0)\big\|_{\delta_0}.
\end{equation}
\end{Rem}
\begin{proof} Take $\gamma$, $\delta$ and $(\rr_0, \uu_0)$ as in the statement of Theorem \ref{sharpestimate}.

Let $(\rr, \uu) = ((t,x)\mapsto r^w(t,x), u^w(t,x))_{w \in \R^d}$ be a time dependent family of $C^1$ functions. For all $n \in \Z^d$, we call $\hat{\rr}_n = (\hat{r}_n(t,w))$ and $\hat{\uu}_n = (\hat{u}_n(t,w))$ the Fourier coefficients of these functions, so that for all $(t,x,w)$,
\begin{equation}
\label{rufourier}
r^w(t,x) = \sum_{n \in \Z^d} \hat{r}_n(t,w) \exp(in\cdot x) \quad \mbox{and} \quad u^w(t,x) = \sum_{n \in \Z^d} \hat{u}_n(t,w) \exp(in\cdot x).
\end{equation}
Then $(\rr, \uu)$ is a solution to \eqref{abstractlineq} if and only if for all $n\in \Z^d$ and $w \in \R^d$, the pair $(\hat{\rr}_n, \hat{\uu}_n) $ is a solution to
\begin{equation}
\label{fouriereq}
\left\{ \begin{gathered}
\partial_t \begin{bmatrix}
\hat{r}_n(t,w) \\
\hat{u}_n(t,w)
\end{bmatrix}
+ i \begin{bmatrix}
n \cdot w & n  \\  
0 & n \cdot w \Id_{\R^d}
\end{bmatrix} \hspace{-3pt} \cdot \hspace{-3pt} \begin{bmatrix}
\hat{r}_n(t,w) \\
\hat{u}_n(t,w)
\end{bmatrix} = I_n(\hat{\rr}_n(t), \hat{\uu}_n(t))\begin{bmatrix}
0 \\
-in
\end{bmatrix},\\
\hat{r}_n(0,w) \mbox{ and } \hat{u}_n(0,w) \mbox{ are the }n\mbox{th Fourier coefficients of }r_0^w \mbox{ and } u_0^w,
\end{gathered} \right.
\end{equation}
with
\[
I_n(\hat{\rr}_n(t), \hat{\uu}_n(t)) := P(n) \cdot \int \left\{\Phi(w') \hat{r}_n(t,w') + \D \Phi(w')\cdot \hat{u}_n(t,w') \right\} \D \mu(w').
\]

Now it suffices to show the following lemma.
\begin{Lem}
\label{fourierlemma}
For each $n \in \Z^d$, equation \eqref{fouriereq} admits a unique solution $(\hat{\rr}_n, \hat{\uu}_n)$ for all times and it is a continuous map from $\R_+$ to $\boldsymbol{L}^{\infty}$. 

Moreover, for all $\gamma > \gamma_0$, there exists $C$ only depending on $M$ and $\gamma$ such that this solution satisfies
\begin{equation*}
|(\hat{\rr}_n(t),\hat{\uu}_n(t))|_{\infty} \leq C |(\hat{\rr}_n(0),\hat{\uu}_n(0))|_{\infty} \exp(\gamma |n| t).
\end{equation*}
\end{Lem}
Indeed, if the lemma is true, the unique classical solution to \eqref{abstractlineq} is given by \eqref{rufourier} with $(\hat{\rr}_n, \hat{\uu}_n)$ given by the lemma. Then, if $\gamma > \gamma_0$, by the lemma, we can find $C$ only depending on $M$ and $\gamma$ such that for all $t \leq \delta_0/\gamma$,
\begin{align*}
\big\| (\rr(t), \uu(t)) \big\|_{\delta_0 - \gamma t} &\leq \sum_{n \in \Z^d}|(\hat{\rr}_n(t),\hat{\uu}_n(t))|_{\infty} \exp\Big(\{\delta_0 - \gamma t\}|n|\Big) \\
&\leq C \sum_{n \in \Z^d}|(\hat{\rr}_n(0),\hat{\uu}_n(0))|_{\infty}  \exp(\delta_0 |n|) \\
&\leq C \|(\rr_0, \uu_0)\|_{\delta_0}.
\end{align*}
Hence, the first point of Theorem~\ref{sharpestimate} is proved.  

For the second point, let $\delta < \delta_0$ and $T < (\delta_0 - \delta) / \gamma_0$. It suffices to prove that
\begin{equation*}
t \in [0,T] \mapsto (\rr(t), \uu(t)) \in \XX_{\delta}
\end{equation*}
is continuous. Take $C$ as given by Lemma~\ref{fourierlemma} with $\gamma := (\delta_0 - \delta)/T$. If $t \in [0,T]$, we need to prove:
\begin{equation*}
\lim_{\substack{s \to t \\ s \in [0,T]}} \sum_{n \in \Z^d} |(\hat{\rr}_n(s),\hat{\uu}_n(s))-(\hat{\rr}_n(t),\hat{\uu}_n(t))|_{\infty} \exp(\delta |n|) = 0.
\end{equation*}
But on the one hand, for all $n \in \Z^d$, $(\hat{\rr}_n, \hat{\uu}_n)$ is continuous in $\boldsymbol{L}^{\infty}$, so that each term of the sum tends to $0$. On the other hand, for all $n \in \Z^d$ and $s \in [0,T]$,
\begin{align*}
|(\hat{\rr}_n(s),\hat{\uu}_n(s))|_{\infty} \exp(\delta |n|)&\leq C |(\hat{\rr}_n(0),\hat{\uu}_n(0))|_{\infty} \exp(\gamma |n| s)\exp(\delta |n|) \\
&\leq C |(\hat{\rr}_n(0),\hat{\uu}_n(0))|_{\infty} \exp(\gamma |n| T)\exp(\delta |n|) \\
& =C |(\hat{\rr}_n(0),\hat{\uu}_n(0))|_{\infty} \exp(\delta_0 |n| )
\end{align*}
where the last line is obtained by definition of $\gamma$. This bound does not depend on $s$ and is summable with respect to $n$. So the dominated convergence theorem applies and the result follows.

Finally, $S_t \boldsymbol{L}_0 \subset \boldsymbol{L}_0$ is a consequence of the fact that the first equation for $n=0$ reduces to:
\[
\partial_t \hat{r}_0(t,w)= 0,
\]
and that the second equation for any $n \in \Z^d$ ensures that for all $t$ and $w$, the vector $\partial_t \hat{u}_n(t,w) + in \cdot w \hat{u}_n(t,w)$ is collinear with $n$.
\end{proof}
In order to prove Lemma \ref{fourierlemma}, we need to state a result for the family of holomorphic functions $(\Lambda_n)_{n \in \Z^d}$ (which was defined in \eqref{defLambda}). By the definition \eqref{defgamma0} of $\gamma_0$, we already know that if $n \in \Z^d$ and $\lambda$ is such that $\Re(\lambda) > \gamma_0 |n|$, then $\Lambda_n(\lambda) \neq i$. We need a stronger result, which tells that if $\Re(\lambda) \geq \gamma |n| > \gamma_0 |n|$, then $\Lambda_n(\lambda)$ stays far from $i$ uniformly in $n$ and $\Im(\lambda)$. This is the content of the following proposition. We postpone its proof to the end of the subsection.
\begin{Prop}
\label{propfarfromi}
For all $\gamma > \gamma_0$, there exists $\delta>0$ such that for all $n \in \Z^d$, for all $\lambda$ with $\Re(\lambda) \geq \gamma |n|$, 
\[
|\Lambda_n(\lambda) - i| \geq \delta.
\]
\end{Prop}

\begin{proof}[Proof of Lemma \ref{fourierlemma}]
We fix $n \in \Z^d \backslash\{0\}$ (there is no evolution for $n = 0$, so the inequality is trivially true). To lighten the notations, we denote by $\alpha = (\alpha(t,w))$ and $\beta = (\beta(t,w))$ the functions that will play the roles of $(\hat{r}_n^w(t))$ and $(\hat{u}_n^w(t))$. More precisely, given $(\alpha_0, \beta_0) \in \boldsymbol{L}^{\infty}$, we look for solutions to 
\begin{equation}
\label{fouriereqalphabeta}
\left\{ \begin{gathered}
\partial_t \begin{bmatrix}
\alpha(t,w) \\
\beta(t,w)
\end{bmatrix}
+ i \begin{bmatrix}
n \cdot w & n  \\  
0 & n \cdot w \Id_{\R^d}
\end{bmatrix} \cdot \begin{bmatrix}
\alpha(t,w) \\
\beta(t,w)
\end{bmatrix} = I_n(\alpha(t), \beta(t))\begin{bmatrix}
0 \\
-in
\end{bmatrix},\\
\alpha(0) = \alpha_0 \quad \mbox{and} \quad \beta(0) = \beta_0,
\end{gathered} \right.
\end{equation}
with
\begin{equation}
\label{defIn}
I_n(\alpha(t), \beta(t)) := P(n) \cdot \int \left\{\Phi(w') \alpha(t, w') + \D \Phi(w')\cdot \beta(t, w') \right\} \D \mu(w').
\end{equation}
($\alpha(t)$ and $\beta(t)$ are notations for $\alpha(t, \bullet)$ and $\beta(t, \bullet)$ respectively.) 

It is easy to see that this equation generates a $C_0$ semigroup on $\boldsymbol{L}^{\infty}$. Indeed, call
\begin{gather*}
A_n(w) := -i \begin{bmatrix}
n \cdot w & n  \\  
0 & n \cdot w \Id_{\R^d}
\end{bmatrix},\\
\boldsymbol{A}_n : (\alpha, \beta) \in \mathcal{D}(\boldsymbol{A}_n) \mapsto A_n(w) \cdot\begin{bmatrix} \alpha(w) \\ \beta(w)\end{bmatrix},\\
\boldsymbol{B}_n : (\alpha, \beta) \in \boldsymbol{L}^{\infty} \mapsto  I_n(\alpha, \beta)\begin{bmatrix} 0 \\ -in \end{bmatrix}.
\end{gather*}
(The domain $\mathcal{D}(\boldsymbol{A}_n)$ is the set of couples $(\alpha, \beta)\in \boldsymbol{L}^{\infty}$ for which the formula in the definition of $\boldsymbol{A}_n$ provides an element of $\boldsymbol{L}^{\infty}$). Then, Equation \eqref{fouriereqalphabeta} can be reformulated as
\[
\partial_t (\alpha(t), \beta(t)) = \boldsymbol{A}_n \cdot (\alpha(t), \beta(t)) + \boldsymbol{B}_n \cdot (\alpha(t), \beta(t)).
\]
But on the one hand, $\boldsymbol{A}_n$ generates the $C_0$ semigroup $(e^{t \boldsymbol{A}_n})_{t \in \R_+}$ with for all $t \in \R_+$, for all $(\alpha, \beta) \in \boldsymbol{L}^{\infty}$ and for all $w \in \R^d$,
\begin{align*}
e^{t \boldsymbol{A}_n}\cdot (\alpha, \beta) (w) &= \exp(t A_n(w)) \cdot \begin{bmatrix} \alpha(w) \\ \beta(w)\end{bmatrix}\\
&= \exp(-i t n\cdot w) \begin{bmatrix} 1 & -it n \\ 0& \Id_{\R^d} \end{bmatrix} \cdot \begin{bmatrix}\alpha(w) \\ \beta(w)\end{bmatrix}\\
&=\exp(-i t n\cdot w) \begin{bmatrix} \alpha(w) - itn\cdot \beta(w) \\
\beta(w) \end{bmatrix}.
\end{align*}
Remark the following estimate of the operator norm of $e^{t\boldsymbol{A}_n}$:
\begin{equation}
\label{estimexptA}
\left\| e^{t \boldsymbol{A}_n} \right\| \leq 1 + t|n|.
\end{equation}
On the other hand, $\boldsymbol{B}_n$ is bounded on $\boldsymbol{L}^{\infty}$ and its operator norm satisfies
\begin{equation*}
\| \boldsymbol{B}_n \| \leq K |n|
\end{equation*}
where $K$ only depends on $M$.

Thus, by \cite[Chapter 3, Theorem 1.1]{pazy2012semigroups}, $\boldsymbol{A}_n + \boldsymbol{B}_n$ is the infinitesimal generator of a $C_0$ semigroup $(e^{t(\boldsymbol{A}_n + \boldsymbol{B}_n)})_{t \in \R_+}$ on $\boldsymbol{L}^{\infty}$, and taking a slightly bigger $K$, for all $t \geq 0$,
\[
\|e^{t(\boldsymbol{A}_n + \boldsymbol{B}_n)}\| \leq \exp(K|n|t).
\]
The continuity property stated in Lemma~\ref{fourierlemma} follows.
The aim is now to lower the constant $K$ down to any $\gamma> \gamma_0$ up to adding a multiplicative constant. 

We fix $(\alpha_0, \beta_0) \in \boldsymbol{L}^{\infty}$. We will compute the Laplace transform $p \mapsto H[p]$ of
\[
h: t \mapsto \Big(e^{t(\boldsymbol{A}_n + \boldsymbol{B}_n)} - e^{t \boldsymbol{A}_n}\Big)\cdot (\alpha_0, \beta_0) = \int_0^t e^{(t-s)\boldsymbol{A}_n} \cdot \boldsymbol{B}_n \cdot (\alpha(s), \beta(s)) \D s \in \boldsymbol{L}^{\infty}.
\] 
With the previous estimate, we can classically (see for example the proof of \cite[Chapter 3, Theorem 5.3]{pazy2012semigroups}) deduce that the Laplace transform $F$ of $(e^{t(\boldsymbol{A}_n + \boldsymbol{B}_n)})$ is its resolvent, namely for all $p \in \C$ with $\Re(p) > K|n|$,
\begin{gather}
\notag F[p] := \int_0^{+ \infty} e^{-pt} e^{t(\boldsymbol{A}_n+\boldsymbol{B}_n)}\cdot (\alpha_0, \beta_0) \D t \in \mathcal{D}(\boldsymbol{A}_n+\boldsymbol{B}_n) \Big(= \mathcal{D}(\boldsymbol{A}_n) \Big), \\
\label{resolventequationA+B} \mbox{and} \quad \Big(p\Id_{\boldsymbol{L}^{\infty}} - (\boldsymbol{A}_n + \boldsymbol{B}_n) \Big)\cdot F[p] = (\alpha_0, \beta_0).
\end{gather}
In the same way, thanks to \eqref{estimexptA}, for all $p \in \C$ with $\Re(p) > 0$ (remark that if $\omega >0$, then $1 + t|n| \leq (1 + |n| / \omega ) e^{\omega t}$),
\begin{gather}
\notag G[p] := \int_0^{+ \infty} e^{-pt} e^{t \boldsymbol{A}_n}\cdot (\alpha_0, \beta_0) \D t \in \mathcal{D}(\boldsymbol{A}_n), \\
\label{resolventequationA} \mbox{and} \quad \Big(p\Id_{\boldsymbol{L}^{\infty}} - \boldsymbol{A}_n \Big)G[p] = (\alpha_0, \beta_0).
\end{gather}
But then, solving the resolvent equations \eqref{resolventequationA} and \eqref{resolventequationA+B}, we easily find that for all $p \in \C$ with $\Re(p) > K|n| > 1$,
\begin{gather}
\label{formulaG} G[p](w) = \begin{bmatrix}\displaystyle{
\frac{\alpha_0(w)}{p + in\cdot w} - \frac{in \cdot \beta_0(w)}{(p + in\cdot w)^2}}\\
\displaystyle{\frac{\beta_0(w)}{p + in\cdot w}}
\end{bmatrix},\\
\label{formulaF} F[p](w) = G[p](w) + I_n(F[p]) \begin{bmatrix} \displaystyle{\frac{-|n|^2 }{(p + in\cdot w)^2}} \\ \displaystyle{\frac{-in}{p + in\cdot w} }\end{bmatrix}. 
\end{gather}
Applying $i I_n$ to \eqref{formulaF}, we get by the definition of $I_n$ in \eqref{defIn} and by the definition of $\Lambda_n(p)$ in \eqref{defLambda}:
\[
\Big( i - \Lambda_n(p) \Big) I_n(F[p]) = i I_n(G[p]).
\]
But as soon as $\Re(p) > \gamma_0 |n|$, one must have $\Lambda_n(p) \neq i$, and
\[
I_n(F[p]) = \frac{i}{i - \Lambda_n(p)}I_n(G[p])
\]
Finally, we get (at least when $\Re(p) > K|n|$)
\[
H[p](w) = F[p](w) - G[p](w) =  \frac{i}{i - \Lambda_n(p)}I_n(G[p]) \begin{bmatrix} \displaystyle{\frac{-|n|^2}{(p + in \cdot w)^2}} \\ \displaystyle{\frac{-in}{p + in\cdot w} }\end{bmatrix}.
\]
But this expression is well defined and analytic in $p$ on $\{ p \in \C \, | \,\Re(p) > \gamma_0 |n| \}$. We keep the notation $H[p]$ in this domain.

 In addition, if $\gamma > \gamma_0$, we have the following estimates.
\begin{itemize}
\item By Proposition \ref{propfarfromi}, there is $\delta>0$ only depending on $\gamma$ such that for all $n \in \Z^d$, for all $p$ with $\Re(p) \geq \gamma |n|$,
\[
|\Lambda_n(p) -i| \geq \delta.
\]
\item Now we give an estimate for $H[p](w)$ when $\Re(p) \geq \gamma |n|$ and $w \in \R^d$. We just use the previous consideration, the fact that when $\Re(p) \geq \gamma |n|$, then $|p + in \cdot w| \geq \gamma |n|$ and the formulae for $I_n$, \eqref{defIn} and for $G[p](w)$, \eqref{formulaG}. In the following computation, $\lesssim$ means "lower than, up to a multiplicative constant which only depends on $M$ and $\gamma$". We have:
\begin{align}
\notag |H[p](w)| &= \left|\frac{i}{i - \Lambda_n(p)}\right| \times |I_n(G[p])| \times \left| \begin{bmatrix} \displaystyle{\frac{-|n|^2}{(p + in \cdot w)^2}} \\ \displaystyle{\frac{-in}{p + in\cdot w} }\end{bmatrix} \right| \\
\notag &\lesssim |I_n(G[p])| \times \frac{|n|}{|p + i n \cdot w|}\\
\notag &= \frac{|n|}{|p + in\cdot w|} \bigg| \int \bigg\{ \Phi(w') \left( \frac{\alpha_0(w')}{p + in\cdot w'} - \frac{in \cdot \beta_0(w')}{(p + in\cdot w')^2} \right)\\
\notag &\hspace{150pt} + \frac{\D \Phi(w') \cdot \beta_0(w')}{p + in\cdot w'} \bigg\} \D \mu (w') \bigg| \\
\label{estimH}&\lesssim \frac{|n|}{|p + in \cdot w|}\hspace{-3pt} \left( \int \hspace{-3pt} \frac{|\Phi(w')| + |\! \D \Phi(w')|}{|p + in\cdot w'|} \D \mu(w') \right) |(\alpha_0, \beta_0)|_{\infty}.
\end{align}
\item There exists $C$ only depending on $M$ and $\gamma$ (and not depending on $w$) such that for all $\gamma \geq \gamma_0$ and for all $w \in \R^d$,
 \begin{equation}
 \label{estimlaplacetransform}
 \frac{1}{2\pi}\hspace{-3pt}\int_{-\infty}^{+\infty} \hspace{-6pt}\int \hspace{-3pt}\frac{|\Phi(w')| + |\!\D \Phi(w')|}{\Big|\gamma |n| + i(q+n\cdot w)\Big|\Big|\gamma|n| + i(q+n\cdot w')\Big|} \D \mu(w') \D q \leq \frac{C}{|n|}.
\end{equation}
This is an easy consequence of the explicit computation
\begin{align*}
\int_{- \infty}^{+\infty} &\frac{\D q}{\Big|\gamma |n| + i(q+n\cdot w)\Big|\Big|\gamma|n| + i(q+n\cdot w')\Big|}\\
& \leq 2\int_{-\infty}^{+ \infty} \frac{\D q}{\Big(\gamma |n| + |q+ n\cdot w|\Big)\Big(\gamma |n| + |q+ n\cdot w'|\Big)}\\
&= \frac{8}{\gamma |n|} \frac{1 + Q}{2 + Q} \frac{\log(1 + Q)}{Q} \mbox{ with } Q = \frac{|n \cdot w - n \cdot w'|}{\gamma|n|}\\
&\leq \frac{8}{\gamma |n|}.
\end{align*}
\end{itemize}

As a consequence, gathering \eqref{estimH} and \eqref{estimlaplacetransform}, we get that for all $w \in \R^d$,
\begin{equation}
\label{boundedinverselaplacetransform}
\int_{\gamma |n| - i\infty}^{\gamma |n| + i \infty} |H[p](w)| \D p = \int_{-\infty}^{+\infty} \Big|H\big[\gamma |n| + iq\big](w)\Big| \D q \leq C |(\alpha_0, \beta_0)|_{\infty},
\end{equation}
with $C$ only depending on $M$ and $\gamma$. Therefore, the abscissa of convergence of $h$ is lower or equal to $\gamma_0|n|$, and the inverse Laplace transform formula applies, that is for all $\gamma > \gamma_0$ and $t \geq 0$
\begin{align*}
h(t) &= \frac{1}{2 \pi i} \int_{\gamma |n| - i\infty}^{\gamma |n| + i \infty} e^{p t}H[p] \D p\\
&=\frac{e^{\gamma |n| t}}{2 \pi i} \int_{- \infty}^{+ \infty} e^{iqt} H\big[\gamma |n| + iq\big] \D q,
\end{align*}
and by \eqref{boundedinverselaplacetransform},
\begin{align}
\notag |h(t)|_{\infty} &\leq \frac{e^{\gamma |n| t}}{2 \pi} \left| \int_{-\infty}^{+\infty} \Big|H\big[\gamma |n| + iq\big]\Big| \D q \right|_{\infty} \\
\label{eq:estim_h} &\leq C e^{\gamma |n|t}|(\alpha_0, \beta_0)|_{\infty},
\end{align}
where $C$ only depends on $M$ and $\gamma$. But
\begin{equation*}
(\alpha(t), \beta(t)) = e^{t (\boldsymbol{A}_n + \boldsymbol{B}_n)}(\alpha_0, \beta_0) = e^{t \boldsymbol{A}_n}(\alpha_0, \beta_0) + h(t).
\end{equation*}
Hence, we get the result by gathering \eqref{estimexptA} and \eqref{eq:estim_h}
\end{proof}
We now prove Proposition \ref{propfarfromi}.
\begin{proof}[Proof of Proposition \ref{propfarfromi}]
Defining the variable $\xi$ by the formula, $\lambda = |n| \xi$, we define for all $n \in \Z^d\backslash \{0\}$, for all $\xi$ with $\Re(\xi)>0$,
\[
 F_n(\xi) := \Lambda_n (|n| \xi) =-i P(n) \cdot \int \left\{ \frac{\Phi(w)}{(\xi + iu_n\cdot w)^2} + i \frac{\D \Phi(w) \cdot u_n}{\xi + iu_n\cdot w} \right\} \D \mu(w),
\]
where $u_n$ stands for the vector of the sphere $n / |n| \in \mathbb{S}^{d-1}$. All these functions are holomorphic on the half-plane
\[
\C_+^* := \{ \xi \in \C \mbox{ such that } \Re(\xi)>0 \}. 
\]
We know that for all $\xi \in \C$ satisfying $\Re(\xi) > \gamma_0$, for all $n\in \Z^d$, $F_n(\xi) \neq i$, and the goal is to prove that for all $\gamma > \gamma_0$, there exists $\delta >0$ such that if $\xi \in \C$ satisfies $\Re(\xi) \geq \gamma$, then $|F_n(\xi) - i|\geq \delta$.

By contradiction, if the result does not hold, we can find $\gamma > \gamma_0$, $(n_k)_{k \in \N} \in (\Z^d)^{\N}$ and $(\xi_k)_{k \in \N} \in \C^{\N}$ with for all $k \in \N$, $\Re(\xi_k) \geq \gamma$ such that
\[
\lim_{k \to + \infty} F_{n_k}(\xi_k) = i.
\] 
We show in several steps that this leads to a contradiction.\\
\underline{Step one:} $(\xi_k)_{k \in \N}$ is bounded.

For all $w \in \R^d$ and all $\xi = \alpha + i \beta \in \C$ with $\Re(\xi) = \alpha \geq \gamma$, we define the nonnegative function
\[
H(\xi,w) :=\frac{|\Phi(w)|}{\alpha^2 + (|\beta| - |w|)^2} + \frac{|\!\D\Phi(w)|}{\sqrt{\alpha^2 + (|\beta| - |w|)^2}}.
\]
On the one hand, for all $w \in \R^d$,
\[
H(\xi, w) \underset{|\xi| \to + \infty}{\longrightarrow} 0,
\]
and on the other hand for all $\xi$ with $\Re(\xi) \geq \gamma$ and $w \in \R^d$,
\[
H(\xi, w) \leq \frac{|\Phi(w)|}{\gamma^2} + \frac{|\!\D \Phi(w)|}{\gamma} \in L^1(\mu).
\]
Hence, the dominated convergence theorem applies and
\[
\int H(\xi, w) \D \mu(w) \underset{\substack{|\xi| \to + \infty \\ \Re(\xi) \geq \gamma}}{\longrightarrow} 0.
\]
We conclude step one remarking that because of \eqref{assumptionP}, for all $n \in \Z^d$ and $\xi \in \C_+^*$,
\begin{equation}
\label{FleqH}
|F_{n}(\xi)| \leq M \int H(\xi, w) \D \mu(w).
\end{equation}
So as
\[
|F_{n_k}(\xi_k)| \underset{k \to + \infty}{\longrightarrow} 1, 
\]
$(\xi_k)$ must be bounded.\\
\underline{Step two:} convergence to a non constant holomorphic function.

Up to an extraction, we can suppose that as $k$ tends to $+\infty$, $\xi_k \to \xi_{\infty}$ with $\Re(\xi_{\infty}) \geq \gamma$, $u_{n_k} \to u \in \mathbb{S}^{d-1}$ and $P(n_k) \to P \in \mathcal{L}(E^{\mathbb{C}}; \C)$ (with $\| P \| \leq M$, thanks to \eqref{assumptionP}). Then, by the dominated convergence theorem, $(F_{n_k})_{k \in \N}$ converges pointwise to $F$ defined for all $\xi \in \C_+^*$ by
\[
 F(\xi) := -i P \cdot \int \left\{ \frac{\Phi(w)}{(\xi + iu\cdot w)^2} + i \frac{\D \Phi(w) \cdot u}{\xi + iu\cdot w} \right\} \D \mu(w).
\]
Furthermore, because of Montel's theorem, this convergence is locally uniform on $\C_+^*$. In particular,
\[
F(\xi_{\infty}) = \lim_{k \to + \infty} F_{n_k}(\xi_k) = i.
\]
Moreover, because of \eqref{FleqH}, 
\[
F(\xi) \underset{\substack{|\xi| \to + \infty \\ \Re(\xi) \geq \gamma}}{\longrightarrow} 0.
\]
So $F$ cannot be constant.
\\ \underline{Step three:} conclusion applying Rouch\'e's theorem. 

If $\xi \in \C$ and $r>0$, we denote by $\mathcal{D}(\xi, r)$ the closed disk centered at $\xi$ and of radius $r$, and $\mathcal{C}(\xi, r) = \partial \mathcal{D}(\xi, r)$ the circle centered at $\xi$ and of radius $r$. Chose $r>0$ such that:
\begin{itemize}
\item for all $\xi \in \mathcal{D}(\xi_{\infty}, r)$, $\Re(\xi) > \gamma_0$,
\item the only zero of $F - i$ in $\mathcal{D}(\xi_{\infty}, r)$ is $\xi_{\infty}$.
\end{itemize}
Call 
\[
a := \inf_{\xi \in \mathcal{C}(\xi_{\infty}, r)} |F(\xi)|.
\]
If $k$ is sufficiently large because of the locally uniform convergence of $(F_{n_k})$ toward $F$, 
\[
\sup_{\xi \in \mathcal{C}(\xi_{\infty}, r)} |F_{n_k}(\xi) - F(\xi)| < a,
\]
so Rouch\'e's theorem applies. For such a $k$, $F_{n_k}-i$ and $F-i$ have the same number of zeroes (counted with multiplicity) on $\mathcal{D}(\xi_{\infty}, r)$. So $F_{n_k}-i$ cancels at least once on $\mathcal{D}(\xi_{\infty}, r)$, and so there exists $\xi$ with $\Re(\xi) > \gamma_0$ such that $F_{n_k}(\xi) = i$, which contradicts the definition of $\gamma_0$.
\end{proof}
\section{Nonlinear instability}
\label{nonlinearinstability}
\subsection{Statement of the main result}
The purpose of this subsection is to prove the existence of solutions to the nonlinear system \eqref{abstracteq}-\eqref{abstractforce} for any initial data in the neighbourhood of an unstable stationary solution. The nonlinear system is viewed as a perturbation of the linearized system \eqref{abstractlineq} for which Theorem \ref{sharpestimate} gives the existence of solutions.

 Fix $\delta_0>0$, consider $\gamma_0$ as defined in \eqref{defgamma0} and suppose $\gamma_0>0$. The initial condition will be taken of the form
 \[
 (\rhorho_0, \vv_0) = (\boldsymbol{1} + \rr_0, \boldsymbol{w} + \uu_0),
 \]
with $(\rr_0, \uu_0) \in \boldsymbol{L}_0$ and $(\nabla \rr_0, \nabla \uu_0) \in X_{\delta_0}$ for some $\delta_0>0$. 
 
We look for solutions of the form
\begin{equation}
\label{linear+rest}
\rho^w(t,x) = 1 + r^w(t,x) + \sigma^w(t,x), \quad v^w(t,x) = w + u^w(t,x) + \xi^w(t,x)
\end{equation}
where here and in the whole section, $(\rr, \uu)$ is a solution to the linear problem: $(\rr(t), \uu(t)) = S_t(\rr_0, \uu_0)$, and where $(\sigsig(0), \xixi(0)) = 0$. 

Injecting this ansatz in \eqref{abstracteq}-\eqref{abstractforce}, we find that $(\sigsig, \xixi)$ must be solution to the following system (where we omit the dependence of each function in $(t,x)$ to gain space, and where the equations must hold for all $w \in \R^d$)
\begin{equation}
\label{abstractcontractiveeq}
\left\{\hspace{-2pt} \begin{gathered}
 \partial_t \sigma^{w} + w\cdot \nabla \sigma^w + \Div(\xi^w) = -\Div\Big((r^w + \sigma^w)(u^w + \xi^w)\Big) , \\
\partial_t \xi^{w} \! + \! (w \cdot \nabla)\xi^{w}\! =\! -\nabla V[\sigsig, \xixi]\! - \! \nabla W_{\rr,\uu}[\sigsig, \xixi] \! - \! \Big[(  u^w \! + \! \xi^w)\! \cdot \! \nabla\Big]\! (  u^w \! + \! \xi^w)  ,\\
V[\sigsig, \xixi] := A\left[  \int \left\{ \Phi(w) \sigma^w + \D \Phi(w)\cdot \xi^w \right\} \D \mu(w) \right], \\
\begin{gathered} W_{\rr, \uu}[\sigsig, \xixi]:= A\left[\int \Big\{ \Phi(w +   u^w + \xi^w) - \Phi(w)\Big\} \big(  r^w + \sigma^w \big) \D \mu(w)\right]\\
\hspace{13pt}+ A\left[\int \Big\{ \Phi(w +   u^w + \xi^w) - \Phi(w) - \D \Phi(w)\cdot\big(   u^w + \xi^w \big)\Big\} \D \mu(w)\right]\hspace{-3pt},
\end{gathered}\\
 \sigma^{w}|_{t=0} = 0 \mbox{ and } \xi^{w}|_{t=0} = 0.
\end{gathered} \right.
\end{equation}
As expected, we recognize the linear system \eqref{abstractlineq} plus terms that are at least quadratic. So we give a Duhamel formulation of this system which is clearly equivalent at the level of regularity at which we work
\begin{equation}
\label{duhamelabstracteq}
\begin{bmatrix}
\sigsig(t) \\ \xixi(t)
\end{bmatrix}
\!\! = \! \!\int_0^t\!\! S_{t-s}\!\! \begin{bmatrix} -\Div\Big((  \rr(s) + \sigsig(s))(  \uu(s) + \xixi(s))\Big) \\
\!-\nabla W_{\rr, \uu}[\sigsig(s), \xixi(s)] \! - \! \Big[(  \uu(s) + \xixi(s))\! \cdot \! \nabla\Big](  \uu(s) + \xixi(s)) \! \end{bmatrix}\!\! \D s.
\end{equation}
The derivatives are taken pointwise in $w$: for instance, the notation $\Div (\vv)$ stands for $(\Div v^w)_{w \in \R^d}$. We have written the couples of density and velocity fields in column to gain space. 

We will prove the following theorem.
\begin{Thm}
\label{existenceofsolutions}
Suppose that $\mu$ is unstable as defined in Definition \ref{def:unstability}, so that $\gamma_0$ defined in \eqref{defgamma0} is positive. Take $\Gamma > \gamma_0$.

Then there exists $\eps_0>0$ only depending on $\Gamma$, $r_0$ and $M$ (the last two appearing in \eqref{intPhi}, \eqref{intdPhi}, \eqref{intdPhiseries}, \eqref{intd2Phiseries} and \eqref{assumptionP}), such that for all $(\rr_0, \uu_0) \in \boldsymbol{L}_0$ (defined in definition \eqref{defL0}), if there is $\delta_0>0$ such that
\begin{equation}
\label{initialdatasmall}
\|\DD \rr_0, \DD \uu_0 \|_{\delta_0} \leq \eps_0,
\end{equation}
then:
\begin{itemize}
\item  \eqref{duhamelabstracteq} admits a solution $(\sigsig(t), \xixi(t)) \in \boldsymbol{L}_0$ for $t \in [0, \delta_0 / \Gamma ]$,
\item for all $\delta < \delta_0$, $(\sigsig, \xixi)$ is continuous from $[0, (\delta_0 - \delta)/ \Gamma]$ to $\XX_{\delta}$,
\item there holds:
\begin{equation}
\label{estimsigmaxi}
\sup_{t \leq \delta_0/\Gamma }\| \sigsig(t), \xixi(t) \|_{\delta_0 - \Gamma t} \leq \frac{\| \DD \rr_0, \DD \uu_0 \|_{\delta_0} ^2}{\eps_0} .
\end{equation}
\end{itemize}

Moreover, this solution is unique in the class of analytic solutions: if $(\tilde\sigsig, \tilde\xixi)$ is a solution to \eqref{duhamelabstracteq} which is continuous from $[0,T]$ to $\XX_{\delta}$ for some $T\leq \delta_0 / \Gamma$ and $\delta>0$, then for all $t \in [0,T]$, $(\tilde\sigsig(t), \tilde\xixi(t)) = (\sigsig(t), \xixi(t))$.

Consequently, for such $(\rr_0, \uu_0)$, $\delta_0$ and $\Gamma$, equations \eqref{abstracteq}-\eqref{abstractforce} admit a unique analytic solution $(\rhorho(t), \vv(t))$ of the form \eqref{linear+rest} with
\[
\rhorho_0 = 1 + \rr_0 \qquad \mbox{and} \qquad \vv_0 = w + \uu_0,
\]
and we can estimate thanks to \eqref{estimsigmaxi} the distance between the linear solution and the nonlinear one.

Finally, $(\rhorho(t), \vv(t))$  stays real if $(\rhorho_0, \vv_0)$ is real, and $\rhorho(t)$ stays nonnegative if $\rhorho_0$ is nonnegative.
\end{Thm}
\begin{Rem}\begin{itemize}
\item This is an existence result in a neighbourhood of the stationary solution: for any initial data (in $\boldsymbol{L}_0$) sufficiently close to the stationary solution, we are able to find a local in time solution to \eqref{abstracteq}-\eqref{abstractforce}. In fact, we could even drop the condition $(\rr_0, \uu_0) \in \boldsymbol{L}_0$ if we were not interested in finding the best time of existence that is possible with this method.
\item If $(\rr_0, \uu_0) \in \boldsymbol{L}_0$ and $\delta_0>0$, then
\begin{equation}
\label{ruleqDrDu}
\| \rr_0, \uu_0 \| \leq \| \DD \rr_0, \DD \uu_0 \|.
\end{equation}
In particular, \eqref{initialdatasmall} implies $\| \rr_0, \uu_0 \| \leq \eps_0$. This remark will be useful in the following.
\item In the Vlasov-Poisson case, a famous result by Loeper \cite[Theorem 1.2]{loeper2006uniqueness} asserts that there is at most one distributional solution (in space and velocity) with bounded macroscopic density. Compared to this result, the uniqueness part of our theorem is very weak: \emph{one reformulation} of the problem (the multiphasic system) admits a unique \emph{analytic} solution. However, in the other cases, the uniqueness of solutions with low regularity in the velocity variable is an open question.
\end{itemize}
\end{Rem}

The proof is based on a fixed point argument in the Duhamel formulation \eqref{duhamelabstracteq}. To perform it, we will first list a few standard (but useful) properties of the family of norms $(\|\bullet\|_{\delta})_{\delta \geq 0}$ defined in \eqref{defdoublenorm}. Secondly, we will use these properties to derive estimates on the source term in Equation \eqref{duhamelabstracteq}. We will then show a version of the Cauchy-Kovalevskaia theorem, very close to the one due to Caflisch in \cite{caflisch1990simplified}. But we will have to take into account the loss of regularity due to the presence of $S_{t-s}$ in \eqref{duhamelabstracteq} (as already done in \cite{grenier1996oscillations, han16}). Also, for the first time to our knowledge, our proof allows to get $\eps_0$ independent of $\delta_0$. It is interesting since if $(\rr(t), \uu(t))$ is an exponential growing mode of frequency $n \in \Z^d$ and $c \in \R$, as a function of $\delta$,
\[
\big\|c\DD \rr_0 , c\DD \uu_0\big\|_{\delta} \propto c|n| \exp(|n|\delta).
\]
Therefore, we have a precise control on the best $\delta$ (and corresponding $T$) we can take for a given $c$: once $\Gamma$ is fixed, $\eps_0$ is fixed and we can get solutions starting from $c (\rr_0, \uu_0)$ up to time $T=\delta/\Gamma$ as soon as $c |n|\exp(|n|\delta) \lesssim \eps_0$, that is
\[
T \propto -\frac{\log c}{|n|}.
\]
(Usually, $\eps_0$ is a decreasing function of $\delta$, and consequently, the condition of existence $c \exp(|n|\delta) \lesssim \eps_0(\delta)$ is stronger.) Therefore, \eqref{initialdatasmall} can be seen as a balance between the size of the initial condition and the time of existence given by the theorem. This is useful when we want to get large times of existence, as we will need in the Vlasov-Poisson case. To get this result, we will have to take advantage of the fact that a solution starting from $\boldsymbol{L}_0$ stays in $\boldsymbol{L}_0$. We will then be able to use Lemma \ref{lemdeltatodelta'} below. We will finally apply this theorem to the proof of Theorem \ref{existenceofsolutions} thanks to the estimates that were previously derived.

\subsection{Properties of the analytic norms}
\label{propertiesanalyticnorms}
The following properties are basic tools when working with analytic regularity (at least the first ones), and most of the proofs can for example be found in \cite{grenier1996oscillations}. However, we will recall them in Appendix \ref{appendixproof} because we have to obtain uniform estimates with respect to the variable $w$. The last lemma is more original and delicate. We have decided to postpone the proofs to the appendix to lighten the reading.

 First, we will introduce a notation to bound all the first derivatives of a function in $X_{\delta}$ or $\XX_{\delta}$.

\begin{Def}
Let $f \in X_{\delta}$ and $\GG \in \XX_{\delta}$ be written for all $x \in \T^d$
\[
\sum_{n \in \Z^d} \hat{f}_n \exp(i n\cdot x) \quad \mbox{and} \quad \sum_{n \in \Z^d} \hat{g}_n(w) \exp(i n\cdot x),
\]
with for all $n \in \Z^d$, $\hat{g}_n \in L^{\infty}$. We define for all $\delta >0$
\begin{equation*}
\begin{aligned}
|\DD f|_{\delta} &:= \sum_{n \in \Z^d} |n| |\hat{f}_n| \exp(\delta |n|),\\
\| \DD \GG \|_{\delta} &:= \sum_{n \in \Z^d} |n| |\hat{g}_n|_{\infty} \exp(\delta |n|).
\end{aligned} 
\end{equation*}
Remark that this definition makes sense whatever the normed vector space in which $f$ and $\GG$ take their values. That is the reason for this definition.
\end{Def}

We move on to the estimates. We write the following propositions for families of functions in $\XX_{\delta}$, but the results are obviously still true for functions in $X_{\delta}$. 

The first proposition asserts that for all $\delta$, $\| \bullet \|_{\delta}$ is a norm of algebra.
\begin{Prop}
\label{product}
Take $\delta >0$ and let $\FF = (f^w)_{w \in \R^d}$ and $\GG = (g^w)_{w \in \R^d}$ be two families functions in $\XX_{\delta}$ with $\FF$ $\R$-valued. Then $\FF \GG := (f^w g^w)_{w \in \R^d} $ is still in $\XX_{\delta}$, and
\[
\| \FF \GG \|_{\delta} \leq \| \FF \|_{\delta} \| \GG \|_{\delta}.
\]
\end{Prop}

The second proposition gives the behaviour of the norms $(\| \bullet \|_{\delta})$ with respect to differentiation.
\begin{Prop}
\label{propestimnabla}
Take $\FF = (f^w)_{w \in \R^d}$ in $\XX_{\delta}$ for some $\delta>0$. Then for all $0< \delta'< \delta$, we have the following estimate:
\[
\| \DD \FF \|_{\delta '} \leq \frac{1}{\delta - \delta'} \| \FF \|_{\delta}.
\]
\end{Prop}

The next proposition is a Leibniz type formula.
\begin{Prop}
\label{propestimchainrule}
Take $\delta >0$ and let $\FF = (f^w)_{w \in \R^d}$ and $\GG = (g^w)_{w \in \R^d}$ be two families of functions with $\FF$ $\R$-valued. Then for all $\delta>0$ we have the following estimate:
\[
\| \DD \, (\FF \GG) \|_{\delta} \leq \| \FF \|_{\delta} \| \DD \GG \|_{\delta} + \| \GG \|_{\delta} \| \DD \FF \|_{\delta}.
\]
In particular, if $\FF$ and $\GG$ take their values in $\R^d$ and if $\alpha$ and $\beta \in \N^d$,
\begin{equation}
\label{estimiteratechainrule}
\| \DD \, (\FF^{\alpha} \GG^{\beta}) \|_{\delta} \leq |\alpha| \| \FF \|^{|\alpha|-1}_{\delta} \| \GG \|^{|\beta|}_{\delta} \| \DD \FF \|_{\delta} + |\beta| \| \FF \|^{|\alpha|}_{\delta} \| \GG \|^{|\beta| - 1}_{\delta} \| \DD \GG \|_{\delta}.
\end{equation}
\end{Prop}

The following estimate will be useful when estimating the force field.
\begin{Prop}
\label{propestimA}
Take $f$ a $E$-valued function. Then we have the following estimate:
\[
| \nabla Af |_{\delta} \leq M |\DD f|_{\delta}.
\]
\end{Prop}

The next lemma asserts more or less that $\DD$ commutes with the semigroup $S$.
\begin{Lem}
\label{lemDS=SD}
Take $\gamma > \gamma_0$ and $C$ as in Theorem \ref{sharpestimate}. If $(\rr_0, \uu_0)$ is such that for some $\delta >0$, 
\[
\| \DD \,  (\rr_0, \uu_0) \|_{\delta} < + \infty,
\]
then the following estimates holds for all $t > \delta / \gamma$:
\[
\| \DD S_t (\rr_0, \uu_0) \|_{\delta - \gamma t} \leq C \| \DD \, (\rr_0, \uu_0) \|_{\delta}.
\]
\end{Lem}

We will work with families of functions that have no constant part: their Fourier coefficients of order $0$ will be $0$. This is crucial to get $\eps_0$ independent of $\delta_0$ as needed in Subsection \ref{sectionalmostlyapounov}. For such functions, we have the following estimates.
\begin{Lem}
\label{lemdeltatodelta'}
Let $\FF =(f^w)_{w \in \R^d} \in \XX_{\delta}$ for some $\delta >0$, and such that for all $w \in \R^d$,
\begin{equation}
\label{zeromass}
\int f^w(x) \D x = 0.
\end{equation}
Then for all $0 \leq \delta' \leq  \delta$,
\[
\| \FF \|_{\delta'} \leq \frac{ \| \FF \|_{\delta}}{\exp(\delta - \delta')}.
\]
\end{Lem}

Finally, we give some estimates dealing with the composition of analytic functions. We recall that $\Phi$ is given by formula \eqref{defPhi}. If $\lambda$ is a nonnegative number, we set
\begin{equation*}
|\Phi|(\lambda) := \sum_{k \in \N^d} |a_k| \lambda^{|k|}.
\end{equation*}
Remark that $|\Phi|$ and its derivatives are sums of nonnegative terms. Thus, they always have a meaning in $[0,+ \infty]$. With this definition, assumptions \eqref{intdPhiseries} and \eqref{intd2Phiseries} can be reformulated
\begin{gather}
\label{intdPhiseriesbis} \int |\Phi|'(|w| + r_0)\D \mu(w) \leq M,\\
\label{intd2Phiseriesbis} \int |\Phi|''(|w| + r_0)\D \mu(w) \leq M.
\end{gather}
We can now state the last lemma of the subsection.
\begin{Lem}
\label{lemestimPhi}
Take $a \in \R^d$, $f$ and $g$ two analytic functions from $\T^d$ to $\R^d$ and $\delta \geq 0$. Then, the following inequalities hold (with possible infinite values)
\begin{gather}
\label{Phi-Phi} |\Phi(a + f) - \Phi(a + g)|_{\delta} \leq |f-g|_{\delta} |\Phi|'(|a| + |f,g|_{\delta}),\\[0.5cm]
\label{DPhi} |\DD\, \Phi(a + f) |_{\delta} \leq |\DD f |_{\delta} |\Phi|'(|a| + |f|_{\delta}),\\[0.5cm]
\label{Phi-Phi-dPhi} |\Phi(a + f) - \Phi(a) - \D \Phi(a)\cdot f|_{\delta} \leq | f |^2_{\delta} |\Phi|''(|a| + |f|_{\delta}),\\[0.5cm]
\label{DPhi-Phi}\begin{aligned}
 &|\DD \, \{ \Phi(a + f) - \Phi(a + g) \}|_{\delta} \\
 &\leq |f-g|_{\delta} |\DD f, \DD g|_{\delta} |\Phi|''(|a| + |f,g|_{\delta}) + |\DD \, (f-g)|_{\delta}|\Phi|'(|a| + |f,g|_{\delta}),
 \end{aligned} \\[0.5cm]
\label{DPhi-Phi-dPhi}\begin{aligned} 
|\DD \, \{ \Phi(&a + f) - \Phi(a + g) - \D \Phi(a)\cdot (f-g)\}|_{\delta}\\
&\leq \{ |\DD \, (f-g)|_{\delta}|f,g|_{\delta} +  |f-g|_{\delta} |\DD f, \DD g|_{\delta} \} |\Phi|''(|a| + |f,g|_{\delta}).
\end{aligned}
\end{gather}
\end{Lem}

\subsection{Estimates for the source term in the Duhamel formulation}
\label{estimatesduhamelformulation}
We will estimate the different terms in \eqref{duhamelabstracteq} in order to apply a Cauchy-Kovalevskaia theorem. We take $(\rr_0, \uu_0) \in \boldsymbol{L}_0$ and $\delta_0>0$, and we call
\begin{equation*}
\eta := \| \DD \rr_0, \DD \uu_0 \|_{\delta_0}.
\end{equation*}
Because of \eqref{ruleqDrDu}, we have then
\begin{equation*}
\| \rr_0, \uu_0 \|_{\delta_0} \leq \eta.
\end{equation*}
We take $\Gamma >\gamma_0$ (as in the statement of Theorem \ref{existenceofsolutions}), and we define 
\begin{equation}
\label{defgamma1}
\gamma_1 := \frac{2}{3} \gamma_0 + \frac{1}{3} \Gamma = \gamma_0 + \frac{1}{3} (\Gamma - \gamma_0).
\end{equation}
In the sequel, $C$ is the constant appearing in Theorem \ref{sharpestimate} with $\gamma = \gamma_1$. It only depends on $M$ and $\Gamma$. In particular, with these definitions,
\begin{gather}
\label{linsolunitenorm} \sup_{t \in [0, \delta_0 / \gamma_1] }\big\|(\rr(t), \uu(t))\big\|_{\delta_0 - \gamma_1 t} \leq C \eta,\\
\label{lingradsolunitenorm} \sup_{t \in [0, \delta_0 / \gamma_1] }\big\|(\nabla \rr(t) , \nabla \uu(t))\big\|_{\delta_0 - \gamma_1 t} \leq C \eta.
\end{gather}

 We begin with the easiest terms.
\begin{Prop}
For all $t < \delta_0 / \gamma_1$, we have
\begin{equation}
\label{estimbasiseasyterms}
\begin{gathered}
\big\| \Div\big(\rr(t) \uu(t) \big) \big\|_{\delta_0 - \gamma_1 t} \leq 2C^2 \eta^2,\\
\big\| \big(\uu(t) \cdot \nabla \big) \uu(t) \big\|_{\delta_0 - \gamma_1 t} \leq C^2 \eta^2.
\end{gathered}
\end{equation}

For all $t < \delta_0/\gamma_1$, for all $\delta \leq \delta_0 - \gamma_1 t$, for all $\eta >0$, and for all families $\pp_1 :=(\sigsig_1, \xixi_1)$ and $\pp_2 := (\sigsig_2, \xixi_2)$, we have
\begin{equation}
\label{estiminductioneasyterms}
\begin{aligned}
\Big\|& \Div\Big\{(\rr(t) + \sigsig_1) ( \uu(t) + \xixi_1) \Big\} - \Div\Big\{(\rr(t) + \sigsig_2) (\uu(t) + \xixi_2 )\Big\} \Big\|_{\delta} \\
&\hspace{0.1cm}\leq \Big( C \eta +  \| \DD \pp_1 , \DD \pp_2 \|_{\delta} \Big) \| \pp_1 - \pp_2 \|_{\delta} + \Big( C\eta + \| \pp_1, \pp_2 \|_{\delta}\Big) \| \DD\,( \pp_1 -\pp_2) \|_{\delta},\\
\Big\| &\Big((\uu(t) + \xixi_1) \cdot \nabla \Big) ( \uu(t) + \xixi_1) - \Big((\uu(t) + \xixi_2) \cdot \nabla \Big) (\uu(t) + \xixi_2) \Big\|_{\delta}\\
&\hspace{0.1cm} \leq \Big( C\eta +  \| \DD \xixi_1 , \DD \xixi_2 \|_{\delta} \Big) \| \xixi_1 - \xixi_2 \|_{\delta} + \Big( C\eta + \| \xixi_1, \xixi_2 \|_{\delta}\Big) \| \DD \, ( \xixi_1 - \xixi_2 )\|_{\delta},
\end{aligned}
\end{equation}
where $\| \DD \pp \|_{\delta}$ stands for $\| \DD \sigsig, \DD \xixi  \|_{\delta}$.
\end{Prop}
\begin{proof}
Inequalities \eqref{estimbasiseasyterms} are easy consequences of \eqref{linsolunitenorm}, \eqref{lingradsolunitenorm}, Proposition \ref{product} and Proposition \ref{propestimchainrule} once remarked that for any $\delta \geq 0$ and any vector valued $\FF$, $\| \Div \FF \|_{\delta} \leq \| \DD \FF \|_{\delta}$. 

To show inequalities \eqref{estiminductioneasyterms}, just decompose the differences of products using the relation
\[
(a + b_1)(c + d_1) - (a + b_2)(c+d_2) = (a + b_1)(d_1-d_2) + (c+d_2)(b_1-b_2),
\]
and use \eqref{linsolunitenorm}, \eqref{lingradsolunitenorm} and Proposition \ref{propestimchainrule}.
\end{proof}

Let us move on the delicate part of estimating the force field. It is now that we must use the analyticity of $\Phi$. We will prove the following.
\begin{Prop}
There is $K$ only depending on $M$, $r_0$ and $\Gamma$ such that:
\begin{itemize}
\item For all $t \leq \delta_0 / \gamma_1$ and for all $\eta$ such that $C \eta \leq r_0$, we have
\begin{equation}
\label{estimbasisforcefield}
|\nabla W_{\rr,\uu}[0,0](t)|_{\delta_0 - \gamma_1 t} \leq K \eta^2.
\end{equation}
(We recall that $r_0$ is used in \eqref{intdPhiseries} and \eqref{intd2Phiseries}.) 
\item For all $t \leq \delta_0/\gamma_1$, for all $\delta \leq \delta_0 - \gamma_1 t$, for all $\eta$ such that $C \eta  \leq r_0$, and for all families $\pp_1 :=(\sigsig_1, \xixi_1)$ and $\pp_2 := (\sigsig_2, \xixi_2)$ such that
\[
C \eta + \| \pp_1, \pp_2 \|_{\delta} \leq r \leq r_0,
\]
we have
\begin{equation}
\label{estiminductionforcefield}
\begin{aligned}
\big|\nabla W_{\rr,\uu}&[\sigsig_1, \xixi_1](t) - \nabla W_{\rr,\uu}[\sigsig_2, \xixi_2](t)\big|_{\delta}\\
& \leq K \Big\{ \Big( C \eta + \| \DD \pp_1, \DD \pp_2 \|_{\delta} \Big)\| \pp_1 - \pp_2 \|_{\delta} + r \| \DD\, (\pp_1 - \pp_2) \|_{\delta} \Big\}.
\end{aligned}
\end{equation}
\end{itemize}
\end{Prop}
\begin{proof}
\underline{Proof of \eqref{estimbasisforcefield}}. By the definition of $W_{\rr,\uu}$ in \eqref{abstractcontractiveeq},
\begin{align*}
W_{\rr,\uu}[0,0] &=  A\left[ \int \Big\{ \Phi(w +  u^w) - \Phi(w) \Big\}r^w \D \mu (w) \right] \\
& \qquad+  A\left[ \int \Big\{ \Phi(w +  u^w) - \Phi(w) -  \D \Phi(w) u^w \Big\} \D \mu (w) \right].
\end{align*}
Thus, by Proposition \ref{propestimA}, if $t \leq \delta_0 / \gamma_1$ and $\delta := \delta_0 - \gamma_1 t$,
\begin{align*}
|\nabla W_{\rr,\uu}[0,0](t)|_{\delta} &\leq M  \bigg|\DD \int \Big\{ \Phi(w +  u^w(t)) - \Phi(w) \Big\}r^w(t) \D \mu (w)  \bigg|_{\delta} \\
&\quad + M \bigg| \DD \int \Big\{ \Phi(w +  u^w(t)) - \Phi(w) -  \D \Phi(w) u^w(t) \Big\} \D \mu (w) \bigg|_{\delta} \\
&\leq M \int \Big|\DD\, \Big(\Big\{ \Phi(w +  u^w(t)) - \Phi(w)  \Big\} r^w(t)\Big) \Big|_{\delta}\D \mu(w) \\
&\quad + M \int \Big| \DD \, \Big\{ \Phi(w +  u^w(t)) - \Phi(w) -  \D \Phi(w) u^w(t) \Big\}\Big|_{\delta} \D \mu (w) 
\end{align*}
But on the one hand, by  \eqref{linsolunitenorm}, \eqref{lingradsolunitenorm} and Proposition \ref{propestimchainrule},
\begin{align*}
\Big|\DD\, \Big( \Big\{ \Phi(w + & u^w(t)) - \Phi(w)  \Big\} r^w(t)\Big) \Big|_{\delta}\\
 &\leq C |\DD \, \{ \Phi(w +  u^w(t)) - \Phi(w)  \}|_{\delta} + C | \Phi(w +  u^w(t)) - \Phi(w)  |_{\delta} ,\\
 &=C |\DD \, \{ \Phi(w +  u^w(t)) \} |_{\delta} + C | \Phi(w +  u^w(t)) - \Phi(w)  |_{\delta} ,\\
 &\leq 2C^2  \eta |\Phi|'(|w| + C \eta ),
\end{align*}
the last line being obtained using \eqref{linsolunitenorm}, \eqref{lingradsolunitenorm}, \eqref{Phi-Phi} and \eqref{DPhi}.

On the other hand, by \eqref{DPhi-Phi-dPhi},
\[
\Big| \DD \, \Big\{ \Phi(w +  u^w(t)) - \Phi(w) -  \D \Phi(w) u^w(t) \Big\}\Big|_{\delta} \leq 2 C^2 \eta^2 |\Phi|''( |w| + C\eta).
\]
In the end, we find
\begin{align*}
|\nabla W_{\rr,\uu}[0,0](t)|_{\delta} &\leq 2M C^2 \eta^2 \int \Big\{  |\Phi|'(|w| +C \eta) +  |\Phi|''(|w| + \eta) \Big\} \D \mu(w).
\end{align*}
We conclude by using \eqref{intdPhiseriesbis}, \eqref{intd2Phiseriesbis} and $C\eta \leq r_0$.\\
\underline{Proof of \eqref{estiminductionforcefield}}. Take $t \leq \delta_0/\gamma_1$, $\delta \leq \delta_0 - \gamma_1 t$, $\eta$ such that $C\eta \leq r_0$, $\pp_1 :=(\sigsig_1, \xixi_1)$ and $\pp_2 := (\sigsig_2, \xixi_2)$ as in the statement. We have by the definition of $W_{\rr,\uu}$ in \eqref{abstractcontractiveeq}:
\begin{align*}
&W_{\rr,\uu}[\pp_1] - W_{\rr,\uu}[\pp_2]\\
&= A \bigg[ \int \!\!\Big\{ \Phi(w +  u^w\!\! + \xi_1^w) - \Phi(w) \Big\}(\sigma_1^w - \sigma_2^w) \D \mu(w) \bigg] \\
&\hspace{5pt} +A \bigg[ \int\!\! \Big\{ \Phi(w +  u^w\!\! + \xi_1^w) - \Phi(w +  u^w\!\! + \xi_2^w) \Big\} ( r^w + \sigma_2^w)\D \mu(w) \bigg] \\
&\hspace{5pt} +A \bigg[ \int \!\!\Big\{ \Phi(w +  u^w\!\! + \xi_1^w) - \Phi(w +  u^w \!\!+ \xi_2^w) - \D \Phi(w)\cdot(\xi_1^w - \xi_2^w)\Big\} \D \mu(w) \bigg]. 
\end{align*}
So by Proposition \ref{propestimA}, omitting the time variable,
\begin{align*}
&|\nabla W_{\rr,\uu}[\pp_1] - \nabla W_{\rr,\uu}[\pp_2]|_{\delta}\\
&\leq \!\!M\!\!  \int \!\!\underbrace{\Big| \DD \Big(\Big\{ \Phi(w +  u^w \!\!+ \xi_1^w) - \Phi(w) \Big\}(\sigma_1^w - \sigma_2^w)\Big)\Big|_{\delta}}_{=:\Lambda_1} \D \mu(w) \\
&\hspace{5pt} +\!\!M\!\! \int\!\! \underbrace{\Big| \DD  \Big( \Big\{ \Phi(w +  u^w\!\! + \xi_1^w) - \Phi(w +  u^w\!\!+ \xi_2^w) \Big\} ( r^w + \sigma_2^w)\Big)\Big|_{\delta}}_{=: \Lambda_2} \D \mu(w) \\
&\hspace{5pt} +\!\!M\!\!\int \!\!\underbrace{\Big|\DD  \Big( \Phi(w +  u^w \!\!+ \xi_1^w) - \Phi(w +  u^w\!\! + \xi_2^w) - \D \Phi(w)\cdot(\xi_1^w - \xi_2^w)\Big)\Big|_{\delta}}_{=: \Lambda_3} \D \mu(w). 
\end{align*}
Let us control these terms one by one. First, by Proposition \ref{propestimchainrule},
\begin{align*}
\Lambda_1 &\leq |\DD\,( \Phi(w +  u^w \!\!+ \xi_1^w) - \Phi(w)) |_{\delta} \| \pp_1 - \pp_2 \|_{\delta} \\
&\hspace{4cm}+ | \Phi(w +  u^w \!\!+ \xi_1^w) - \Phi(w) |_{\delta} \| \DD \,(\pp_1 - \pp_2) \|_{\delta}\\
&=|\DD \Phi(w +  u^w \!\!+ \xi_1^w) |_{\delta} \| \pp_1 - \pp_2 \|_{\delta} \\
&\hspace{4cm}+ | \Phi(w +  u^w \!\!+ \xi_1^w) - \Phi(w) |_{\delta} \| \DD \,(\pp_1 - \pp_2) \|_{\delta}.
\end{align*}
With \eqref{Phi-Phi}, \eqref{DPhi}, \eqref{linsolunitenorm}, \eqref{lingradsolunitenorm} and $C\eta + \| \pp_1, \pp_2  \|_{\delta} \leq r \leq r_0$, we obtain
\begin{equation}
\label{Lambda1}
\Lambda_1 \leq \Big\{ (C\eta + \| \DD \pp_1, \DD \pp_2\|_{\delta}) \| \pp_1 - \pp_2 \|_{\delta} + r  \DD \,(\pp_1 - \pp_2) \|_{\delta}\Big\} |\Phi|'(|w| + r_0).
\end{equation}
Then, still by Proposition \ref{propestimchainrule}, \eqref{linsolunitenorm}, \eqref{lingradsolunitenorm} and $C\eta + \| \pp_1, \pp_2  \|_{\delta} \leq r \leq r_0$,
\begin{align*}
\Lambda_2 &\leq r \Big| \DD \, \Big\{ \Phi(w +  u^w\!\! + \xi_1^w) - \Phi(w +  u^w\!\!+ \xi_2^w) \Big\}\Big|_{\delta}\\
&\qquad + (C\eta + \| \DD \pp_1, \DD \pp_2 \|_{\delta})\Big| \Phi(w +  u^w\!\! + \xi_1^w) - \Phi(w +  u^w\!\!+ \xi_2^w)\Big|_{\delta}.
\end{align*}
Using \eqref{DPhi-Phi} and \eqref{Phi-Phi}, we get
\begin{align}
\notag\Lambda_2 &\leq r  \|\DD \, (\pp_1 - \pp_2)\|_{\delta} |\Phi|'(|w| + r_0) \\
\notag&\qquad + r_0\| \pp_1 - \pp_2\|(C\eta + \| \DD \pp_1 , \DD \pp_2 \|_{\delta}) |\Phi|''(|w| + r_0)) \\
\notag &\qquad + (C\eta + \| \DD \pp_1 , \DD \pp_2 \|_{\delta})\| \pp_1 - \pp_2 \|_{\delta} |\Phi|'(|w| + r_0)\\
\notag &=(C\eta + \| \DD \pp_1 , \DD \pp_2 \|_{\delta})\| \pp_1 - \pp_2 \|_{\delta} \Big\{ |\Phi|'(|w| + r_0) + r_0 |\Phi|''(|w| + r_0) \Big\} \\
\label{Lambda2} &\qquad + r \|\DD \, (\pp_1 - \pp_2)\|_{\delta} |\Phi|'(|w| + r_0).
\end{align}
For $\Lambda_3$, we just have to use \eqref{DPhi-Phi-dPhi}. We get
\begin{equation}
\label{Lambda3}
\Lambda_3 \leq \Big\{ r \|\DD \, (\pp_1 - \pp_2) + (C\eta + \| \DD \pp_1, \DD \pp_2 \|_{\delta}) \| \pp_1 - \pp_2 \| \Big\} |\Phi|''(|w| + r_0).
\end{equation}

The result is obtained by integrating \eqref{Lambda1}, \eqref{Lambda2} and \eqref{Lambda3} with respect to $w$ and by using \eqref{intdPhiseriesbis} and \eqref{intd2Phiseriesbis}.
\end{proof}
In the end, the results of this subsection can be summarized in the following way.
\begin{Thm}[Conclusion of the subsection]
Take $\Gamma > \gamma_0$, $\gamma_1$ as in \eqref{defgamma1} and $C$ the constant given by Theorem \ref{sharpestimate} with $\gamma = \gamma_1$. Equation \eqref{duhamelabstracteq} can be rewritten as
\begin{equation}
\label{abstractintegraleq}
\pp(t) = \int_0^t S_{t-s} \mathcal{F}_{\eta}(s, \pp(s)) \D s,
\end{equation}
with $S$ defined in \eqref{defS} and satisfying \eqref{estimS}, and $\mathcal{F}_{\eta}$ satisfying the following estimates for some $K>0$ only depending on $M$, $r_0$ and $\Gamma$.
\begin{itemize}
\item For all $t \leq \delta_0 / \gamma_0$ and for all $\eta$ such that $C\eta \leq r_0$, we have
\begin{equation}
\label{estimbasisabstract}
\|\mathcal{F}_{\eta}(t,0)\|_{\delta_0 - \gamma_0 t} \leq K \eta^2.
\end{equation}
\item For all $t \leq \delta_0/\gamma_1$, for all $\delta \leq \delta_0 - \gamma_1 t$, for all $\eta$ such that $C \eta \leq r_0$, and for all $\pp_1$ and $\pp_2$ such that
\[
C \eta + \| \pp_1, \pp_2 \|_{\delta} \leq r \leq r_0,
\]
we have
\begin{equation}
\label{estiminductionabstract}
\begin{aligned}
\big\|\mathcal{F}_{\eta}&(t, \pp_1) - \mathcal{F}_{\eta}(t, \pp_2)\big\|_{\delta} \\
&\leq K \Big\{ \Big( C\eta + \| \DD \pp_1, \DD \pp_2 \|_{\delta} \Big)\| \pp_1 - \pp_2 \|_{\delta} + r \| \DD\, (\pp_1 - \pp_2) \|_{\delta} \Big\}.
\end{aligned}
\end{equation}
\item For any $\pp \in \XX_{\delta}\cap \boldsymbol{L}_0$, $\delta \geq 0$, for all $\eta >0$ and for all $t \leq \delta_0 / \gamma_1$,
\begin{equation}
\label{FetaL0}
\mathcal{F}_{\eta}(t, \pp) \in \boldsymbol{L}_0.
\end{equation}
\end{itemize}
\end{Thm}
\begin{proof}
Looking at \eqref{duhamelabstracteq}, we see that
\[
\mathcal{F}_{\eta}(t, \pp) = \begin{bmatrix} -\Div\Big(( \rr(t) + \sigsig)( \uu(t) + \xixi)\Big) \\
-\nabla W_{\rr,\uu}[\pp] - \Big[( \uu(t) + \xixi) \cdot \nabla\Big]( \uu(t) + \xixi)  \end{bmatrix}
\]
Estimates \eqref{estimbasisabstract} and \eqref{estiminductionabstract} are obvious consequences of \eqref{estimbasiseasyterms}, \eqref{estiminductioneasyterms}, \eqref{estimbasisforcefield} and \eqref{estiminductionforcefield}.

To see \eqref{FetaL0}, remark that as for all $w \in \R^d$, $u^w$ and $\xi^w$ are gradients, 
\[
\Big[( \uu(t) + \xixi) \cdot \nabla\Big]( \uu(t) + \xixi) = \frac{1}{2} \nabla \Big( | \uu(t) + \xixi|^2\Big).
\]
\end{proof}

\subsection{A Cauchy-Kovalevskaia theorem}
\label{cauchykovalevskaiatheorem}
We want to derive from the estimates \eqref{estimS}, \eqref{estimbasisabstract} and \eqref{estiminductionabstract} and from property \eqref{FetaL0} an existence result for equation \eqref{abstractintegraleq}.

 The theorem is the following.
\begin{Thm}
\label{cauchykovalevskaia}
For all $\Gamma > \gamma_0$ there exists $\eps_0>0$ only depending on $r_0$, $M$ and $\Gamma$ (and not $\delta_0$) such that if
\begin{equation}
\label{conditionetadelta0}
\eta \leq \eps_0,
\end{equation}
then:
\begin{itemize}
\item  equation \eqref{abstractintegraleq} admits a solution $\pp$ for $t \in [0, \delta_0/\Gamma]$ with values in $\boldsymbol{L}_0$,
\item for all $\delta < \delta_0$, $\pp$ is continuous from $[0, (\delta_0 - \delta)/ \Gamma]$ to $\XX_{\delta}$,
\item there holds:
\begin{equation}
\label{estimsolution}
\sup_{t \in [0, \delta_0 / \Gamma]} \| \pp(t) \|_{\delta_0 - \Gamma t} \leq \frac{\eta^2}{\eps_0}.
\end{equation}
\end{itemize}

Moreover, this solution is unique in the class of analytic solutions: if $\qq$ is a solution to \eqref{abstractintegraleq} which is continuous from $[0,T]$ to $\XX_{\delta}$ for some $T \leq  \delta_0 / \Gamma$ and $\delta>0$, then for all $t \in [0,T]$, $\qq(t) = \pp(t)$.
\end{Thm}
\begin{Rem}
The proof also gives for free
\begin{equation}
\label{estimsolutionbis}
\sup_{t \in [0, \delta_0 / \Gamma)} \sup_{\delta < \delta_0 - \Gamma t} \sqrt{\delta_0} \left( \delta_0 - \delta - \Gamma t \right)^{1/2} \big\| \DD \pp(t) \big\|_{\delta} \leq \frac{\eta^2}{\eps_0}.
\end{equation}
\end{Rem}
\begin{proof}
Take $\Gamma > \gamma_0$, and define as previously $\gamma_1$ by \eqref{defgamma1}. As before, $C$ will stand for the constant appearing in Theorem \ref{sharpestimate} with $\gamma = \gamma_1$. Define as well
\[
\gamma_2 := \frac{1}{3}\gamma_0 + \frac{2}{3} \Gamma = \gamma_0 + \frac{2}{3} (\Gamma - \gamma_0).
\]
 In the whole proof, equation \eqref{estimS} will be used with $\gamma_2$, and the corresponding constant will be considered as a function of $\Gamma$. Also, in the whole proof, $K$ will denote a large constant which will be likely to grow from line to line, but only depending on $r_0$, $M$ and $\Gamma$.

First we define the following norm:
\begin{equation}
\label{deftriplenorm}
\VERT \pp \VERT := \sup_{t \in [0, \delta_0 / \Gamma]} \| \pp(t) \|_{\delta_0 - \Gamma t} + \sqrt{\delta_0} \sup_{\substack{t < \delta_0 / \Gamma \\ \delta < \delta_0 - \Gamma t}}  \big( \delta_0 - \delta - \Gamma t \big)^{1/2} \| \DD \pp(t) \|_{\delta}.
\end{equation}
This type of norm has been used for the first time by Caflish in \cite{caflisch1990simplified}. Here, we have chosen the exponent in the derivative part equal to $1/2$ and we have added a factor $\sqrt{\delta_0}$. All these choices are made to obtain $\eps_0$ independent of $\delta_0$ as the following computations will show.

For a given $\eta$, we introduce the scheme:
\begin{gather*}
\pp_0 = 0, \\ 
\forall n \geq 0, \quad \pp_{n+1}(t) := \int_0^t S_{t-s} \mathcal{F}_{\eta}(s, \pp_n(s)) \D s.
\end{gather*}

First of all, $0 \in \boldsymbol{L}_0$. Moreover, if $\pp \in \boldsymbol{L}_0$ is sufficiently regular, and if $0 \leq s \leq t$, then $\mathcal{F}_{\eta}(s, \pp) \in \boldsymbol{L}_0$ by \eqref{FetaL0}, and so $S_{t-s} \mathcal{F}_{\eta}(s, \pp) \in \boldsymbol{L}_0$ by Theorem \ref{sharpestimate}. Consequently, it is easy to prove by induction that for all $n\in\N$ and all $t \geq 0$ such that the definition of $\pp_n(t)$ makes sense, then $\pp_n(t) \in \boldsymbol{L}_0$. In particular, $\pp_n(t)$ will satisfy \eqref{zeromass} and, we will be able to use Lemma \ref{lemdeltatodelta'}. 

Now we suppose that \eqref{conditionetadelta0} holds and we will show that as soon as $\eps_0$ is small enough, then the scheme will converge to a certain $\pp$ for which we will give an estimate.\\
\underline{Step one:} estimating $\pp_1$.

 Let us begin the computations by estimating $\VERT \pp_1 \VERT$. If $t \leq \delta_0 / \Gamma$,
\begin{align*}
\| \pp_1(t) \|_{\delta_0 - \Gamma t} &\leq \int_0^t \| S_{t-s} \mathcal{F}_{\eta} (s, 0) \|_{\delta_0 - \Gamma t} \D s\\
&\leq K \int_0^t \| \mathcal{F}_{\eta} (s, 0) \|_{\delta_0 - ( \Gamma - \gamma_2)t - \gamma_2 s} \D s &\mbox{by \eqref{estimS},}\\
&\leq K \int_0^t\!\frac{ \| \mathcal{F}_{\eta} (s, 0) \|_{\delta_0 - \gamma_1 s}}{ \exp\big((\delta_0 - \gamma_1 s) - (\delta_0 - (\Gamma - \gamma_2) t - \gamma_2 s\big)} \D s  &\mbox{by Lem. \ref{lemdeltatodelta'}},\\
&\leq \frac{K \eta^2}{ \exp\big((\Gamma- \gamma_2) t \big)} \int_0^t \exp\big(- (\gamma_2- \gamma_1)s \big) \D s &\mbox{by \eqref{estimbasisabstract},} \\ 
&\leq \frac{K}{\gamma_2 - \gamma_1} \frac{ \eta^2}{\exp\big((\Gamma - \gamma_2 ) t\big)}.
\end{align*}
Putting the $\gamma_2 - \gamma_1$ in the constant $C$ and using $\Gamma > \gamma_2$, we get
\begin{equation}
\label{estimp1}
\| \pp_1(t) \|_{\delta_0 - \Gamma t} \leq C \eta^2.
\end{equation}

On the other hand, choose $t < \delta_0 / \Gamma$ and $\delta < \delta_0 - \Gamma t$. Taking for all $s \in [0,t]$
 \[
\delta'(s) := \frac{\delta + \gamma_2 (t-s) + \delta_0 - \gamma_1 s}{2},
\]
 we get
\begin{align*}
\| \DD \pp_1(t) \|_{\delta} &\leq \int_0^t \| \DD S_{t-s} \mathcal{F}_{\eta} (s, 0) \|_{\delta} \D s\\
&\leq K \int_0^t \| \DD \mathcal{F}_{\eta} (s, 0) \|_{\delta + \gamma_2(t-s)} \D s &\mbox{by Lem. \ref{lemDS=SD}},\\
&\leq K \int_0^t \frac{\| \mathcal{F}_{\eta} (s, 0) \|_{\delta'(s)}}{\delta'(s) - \delta - \gamma_2(t-s)} \D s &\mbox{by Prop. \ref{propestimnabla}},\\
&\leq K \int_0^t \frac{\| \mathcal{F}_{\eta} (s, 0) \|_{\delta_0 - \gamma_1 s}}{\delta'(s) - \delta - \gamma_2(t-s)} e^{-(\delta_0 - \gamma_1 s - \delta'(s))}\D s &\mbox{by Lem. \ref{lemdeltatodelta'}},\\
&\leq K \eta^2 \int_0^t \frac{e^{-(\delta_0 - \delta - \gamma_2 t + (\gamma_2 - \gamma_1)s)/2}}{\delta_0 - \delta - \gamma_2 t + (\gamma_2 - \gamma_1)s} \D s &\mbox{by \eqref{estimbasisabstract}},\\
&\leq K \eta^2 e^{-(\delta_0 - \delta - \gamma_2 t)/2} \int_0^t \frac{\D s}{\delta_0 - \delta - \gamma_2 t + (\gamma_2 - \gamma_1)s}\\
&= K \eta^2  e^{-(\delta_0 - \delta - \gamma_2 t)/2} \log\left( \frac{\delta_0 - \delta - \gamma_1 t}{\delta_0 - \delta - \gamma_2 t} \right).
\end{align*}
But remark that for $t \leq \delta_0/\Gamma$ and $\delta < \delta_0 - \Gamma t$,
\begin{align*}
 \sqrt{\delta_0}\big( &\delta_0 - \delta - \Gamma t \big)^{1/2}e^{-(\delta_0 - \delta - \gamma_2 t)/2} \log\left( \frac{\delta_0 - \delta - \gamma_1 t}{\delta_0 - \delta - \gamma_2 t} \right)\\
&\leq \delta_0 e^{-\delta_0 / 2}\sup_{\tilde{t} < 1 / \Gamma} \sup_{\tilde{\delta} < 1 - \Gamma \tilde{t}} \big( 1 - \tilde{\delta} - \Gamma \tilde{t} \big)^{1/2} e^{-(1 - \tilde{\delta} - \gamma_2 \tilde{t} )/2}\log\left( \frac{1 - \tilde{\delta} - \gamma_1 \tilde{t}}{1 - \tilde{\delta} - \gamma_2 \tilde{t}} \right),
\end{align*}
and that
\[
\sup_{\tilde{t} < 1 / \Gamma} \sup_{\tilde{\delta} < 1 - \Gamma \tilde{t}} \big( 1 - \tilde{\delta} - \Gamma \tilde{t} \big)^{1/2}e^{-(1 - \tilde{\delta} - \gamma_2 \tilde{t} )/2}  \log\left( \frac{1 - \tilde{\delta} - \gamma_1 \tilde{t}}{1 - \tilde{\delta} - \gamma_2 \tilde{t}} \right)
\]
only depends on $\Gamma$. So we can include this factor in the constant $C$ and because $\delta_0 e^{-\delta_0/2} \leq 2$, we get
\begin{equation}
\label{estimdp1}
 \sqrt{\delta_0} \big( \delta_0 - \delta - \Gamma t \big)^{1/2} \| \DD \pp_1(t) \|_{\delta} \leq K \eta^2.
\end{equation}
Gathering \eqref{estimp1} and \eqref{estimdp1}, we conclude that
\begin{equation*}
\VERT \pp_1 \VERT \leq K \eta^2.
\end{equation*}
\underline{Step two:} estimating $\pp_n-\pp_{n-1}$ by induction.

Now, we prove by induction that when $\eps_0$ is sufficiently small, then for all $n \geq 1$,
\begin{equation}
\label{inductionstep}
\VERT \pp_{n} - \pp_{n-1} \VERT \leq 2^{-(n-1)} \VERT \pp_1 \VERT,
\end{equation}
and
\begin{equation}
\label{pnleqeta}
\VERT \pp_n \VERT \leq \eta.
\end{equation}
\underline{Basis step}. When $n = 1$, \eqref{inductionstep} is automatically true, and as under condition \eqref{conditionetadelta0},
 \[
 \VERT \pp_1 \VERT \leq K \eta^2 \leq K \eps_0 \eta,
 \]
as soon as 
 \begin{equation}
 \label{condition1eps0}
 K \eps_0 \leq \frac{1}{2},
 \end{equation}
 we have
 \begin{equation}
 \label{p1leqeta/2}
  \VERT \pp_1 \VERT \leq \frac{\eta}{2}.
 \end{equation}
 Equation \eqref{pnleqeta} is then trivially true for $n=1$.\\
\underline{Induction step}. If the result is true for $k=1, \dots, n$, let us estimate the norm $\VERT \pp_{n+1} - \pp_n \VERT$. First of all, as soon as 
\begin{equation}
\label{condition2eps0}
\eps_0 \leq \frac{r_0}{C+1},
\end{equation}
then for all $s \leq \delta_0 / \Gamma$, 
\begin{equation*}
C \eta + \| \pp_n(s), \pp_{n-1}(s) \|_{\delta_0 - \Gamma s} \leq (C+1) \eta \leq r_0,
\end{equation*}
so we will be able to use \eqref{estiminductionabstract} with $r := (C+1) \eta$. Take $t \leq \delta_0 / \Gamma$. By setting for all $s \in [0,t]$
\[
\delta(s) := \delta_0 -(\Gamma - \gamma_2 ) t- \gamma_2 s ,
\]
we have by \eqref{estimS}, \eqref{estiminductionabstract}, \eqref{pnleqeta}, \eqref{deftriplenorm} and Lemma \ref{lemdeltatodelta'}:
\begin{align*}
\| \pp&_{n + 1}(t) - \pp_n(t) \|_{\delta_0 - \Gamma t} \\
& \leq \int_0^t \| S_{t-s} \big\{ \mathcal{F}_{\eta}(s, \pp_n(s)) - \mathcal{F}_{\eta}(s, \pp_{n-1}(s)) \big\} \|_{\delta_0 - \Gamma t} \D s \\
&\leq K \int_0^t \|\mathcal{F}_{\eta}(s, \pp_n(s)) - \mathcal{F}_{\eta}(s, \pp_{n-1}(s))  \|_{\delta(s)} \D s \\
&\leq K \int_0^t \Big\{ \Big( C\eta + \| \DD \pp_{n-1}(s), \DD \pp_n(s) \|_{\delta(s)}\Big)\| \pp_n(s) - \pp_{n-1}(s) \|_{\delta( s)}\\
&\hspace{5cm} + (C+1) \eta \| \DD \, (\pp_n(s) - \pp_{n-1}(s)) \|_{\delta(s)} \Big\} \D s \\
&\leq K \int_0^t\bigg\{ \bigg( C \eta + \frac{\eta}{\sqrt{\delta_0} \big( \delta_0 - \delta(s) - \Gamma s\big)^{1/2}} \bigg) \frac{\VERT \pp_n - \pp_{n-1} \VERT}{\exp\big(\delta_0 - \Gamma s - \delta(s)\big)} \\
&\hspace{5cm} + (C+1) \eta \frac{\VERT \pp_n - \pp_{n-1} \VERT}{\sqrt{\delta_0} \big( \delta_0 - \delta(s) - \Gamma s\big)^{1/2}} \bigg\} \D s \\
&\leq K \VERT \pp_n - \pp_{n-1} \VERT \eta \int_0^t \hspace{-3pt} \left\{\frac{C}{\exp\big(\delta_0 - \Gamma s - \delta(s)\big)} \hspace{-2pt} + \hspace{-2pt} \frac{C+2}{\sqrt{\delta_0}\big( \delta_0 - \Gamma s - \delta(s) \big)^{1/2}}  \right\}\! \D s.
\end{align*}
But using 
\[
\delta_0 - \Gamma s - \delta(s) = (\Gamma - \gamma_2) t - (\Gamma - \gamma_2)s,
\]
we get
\begin{align*}
\| \pp_{n + 1}(t) - \pp_n(t) \|_{\delta_0 - \Gamma t} &\leq K \VERT \pp_n - \pp_{n-1} \VERT \eta \left\{ C + (C+2)\sqrt{\frac{t}{\delta_0}}\right\} \\
&\leq K \VERT \pp_n - \pp_{n-1} \VERT \eta. 
\end{align*}
In particular, 
\begin{equation}
\label{estiminduction1}
\| \pp_{n + 1}(t) - \pp_n(t) \|_{\delta} \leq \frac{1}{4} \VERT \pp_n - \pp_{n-1} \VERT
\end{equation}
as soon as
\begin{equation}
\label{condition3eps0}
K \eps_0 \leq \frac{1}{4}.
\end{equation}
On the other hand, if $t < \delta_0 / \Gamma$ and $\delta < \delta_0 - \Gamma t$, by using Lemma \ref{lemDS=SD},
\[
\| \DD\,(\pp_{n + 1}(t) - \pp_n(t)) \|_{\delta} \leq K \int_0^t \| \DD \, ( \mathcal{F}_{\eta}(s, \pp_n(s)) - \mathcal{F}_{\eta}(s, \pp_{n-1}(s)))\|_{\delta + \gamma_2 (t-s)} \D s.
\]
We get rid of the $\DD$ by using Proposition \ref{propestimnabla} with for all $s \in [0,t]$,
\[
\delta(s) := \frac{\delta + \gamma_2(t-s)+ \delta_0 - \Gamma s}{2}, 
\]
With this choice, for all $s \in [0,t]$,
\begin{equation}
\label{Delta(s)}
\Delta(s) := \delta(s) - \delta - \gamma_2 (t-s) = \delta_0 - \Gamma s - \delta(s) = \frac{\delta_0 - \delta - \gamma_2 t -(\Gamma - \gamma_2) s }{2}.
\end{equation}
Consequently,
\[
\| \DD\,(\pp_{n + 1}(t) - \pp_n(t)) \|_{\delta} \leq K \int_0^t \frac{\|  \mathcal{F}_{\eta}(s, \pp_n(s)) - \mathcal{F}_{\eta}(s, \pp_{n-1}(s))\|_{\delta(s)}}{\Delta(s)} \D s.
\]
By the exact same estimations of $\|  \mathcal{F}_{\eta}(s, \pp_n(s)) - \mathcal{F}_{\eta}(s, \pp_{n-1}(s))\|_{\delta(s)}$ as before,
\begin{align*}
\| \DD\,(\pp&_{n + 1}(t) - \pp_n(t)) \|_{\delta} \\
&\leq K \VERT \pp_n - \pp_{n-1} \VERT \eta \int_0^t  \left\{\frac{C}{\Delta(s) \exp\big(\Delta(s)\big)}  + \frac{C+2}{\sqrt{\delta_0} \Delta(s)^{3/2} }  \right\} \D s.
\end{align*}
We just have to use
\[
\Delta(s) \geq \frac{\delta_0 - \delta - \Gamma t}{2}
\]
to get the bound
\begin{align*}
\int_0^t \frac{\D s}{ \Delta(s) \exp\big( \Delta(s) \big)} &\leq  \exp\left( - \frac{\delta_0 - \delta - \Gamma t}{2} \right) \int_0^t \frac{\D s}{ \Delta(s)} \\
&\leq 2 \exp\left( - \frac{\delta_0 - \delta - \Gamma t}{2} \right) \int_0^t \frac{\D s}{ \delta_0 - \delta - \gamma_2 t - (\Gamma - \gamma_2)s } \\
&\leq 2 \exp\left( - \frac{\delta_0 - \delta - \Gamma t}{2} \right) \log\left( \frac{\delta_0 - \delta - \gamma_2 t}{\delta_0 - \delta - \Gamma t} \right).
\end{align*}
And besides, by \eqref{Delta(s)},
\[
\int_0^t \frac{\D s}{\Delta(s)^{3/2}} \leq \frac{4 \sqrt{2}}{\Gamma - \gamma_2} \frac{1}{\big(  \delta_0 - \delta - \Gamma t\big)^{1/2}}. 
\]
In the end,
\[
\sqrt{\delta_0} \big( \delta_0 - \delta - \Gamma t \big)^{1/2}\| \DD\,(\pp_{n + 1}(t) - \pp_n(t)) \|_{\delta} \leq K \VERT \pp_n - \pp_{n-1} \VERT \eta ( L + 1 ),
\]
with 
\begin{align*}
L &=  \sup_{\substack{t < \delta_0/ \Gamma \\ \delta < \delta_0 - \Gamma t}} \sqrt{\delta_0} \big( \delta_0 - \delta - \Gamma t \big)^{1/2} \exp\left( - \frac{\delta_0 - \delta - \Gamma t}{2} \right) \log\left( \frac{\delta_0 - \delta - \gamma_2 t}{\delta_0 - \delta - \Gamma t} \right)\\
&\leq \delta_0 \exp( - \delta_0/2) \sup_{\substack{\tilde{t} < 1/ \Gamma \\ \tilde{\delta} < 1 - \Gamma \tilde{t}}} \big( 1 - \tilde{\delta} - \Gamma \tilde{t} \big)^{1/2} \log\left( \frac{1 - \tilde{\delta} - \gamma_2 \tilde{t}}{1 - \tilde{\delta} - \Gamma \tilde{t}} \right)\\
&\leq 2 \sup_{\substack{\tilde{t} < 1/ \Gamma \\ \tilde{\delta} < 1 - \Gamma \tilde{t}}} \big( 1 - \tilde{\delta} - \Gamma \tilde{t} \big)^{1/2} \log\left( \frac{1 - \tilde{\delta} - \gamma_2 \tilde{t}}{1 - \tilde{\delta} - \Gamma \tilde{t}} \right)
\end{align*}
and 
\[
\sup_{\substack{\tilde{t} < 1 / \Gamma \\ \tilde{\delta} < 1 - \Gamma \tilde{t}}} \big( 1 - \tilde{\delta} - \Gamma \tilde{t} \big)^{1/2} \log\left( \frac{1 - \tilde{\delta} - \gamma_2 \tilde{t}}{1 - \tilde{\delta} - \Gamma \tilde{t}} \right)
\]
only depending on $\Gamma$. Thus,
\begin{align*}
\sqrt{\delta_0}\big( \delta_0 - \delta - \Gamma t\big)^{1/2}\| \DD\,(\pp_{n + 1}(t) - \pp_n(t)) \|_{\delta} &\leq K \eta \VERT \pp_n - \pp_{n-1} \VERT\\
&\leq K \eps_0 \VERT \pp_n - \pp_{n-1} \VERT.
\end{align*}
Once again, 
\begin{equation}
\label{estiminduction2}
\left( \frac{\delta_0 - \delta - \Gamma t}{\delta_0} \right)^{1/2} \| \DD \,(\pp_{n + 1}(t) - \pp_n(t)) \|_{\delta} \leq \frac{1}{4} \VERT \pp_n - \pp_{n-1} \VERT
\end{equation}
as soon as
\begin{equation}
\label{condition4eps0}
K \eps_0 \leq \frac{1}{4}.
\end{equation}
So under conditions \eqref{condition3eps0}, \eqref{condition4eps0}, then \eqref{estiminduction1} and \eqref{estiminduction2} hold and thus by summing them, 
\[
\VERT \pp_{n+1} - \pp_{n} \VERT \leq \frac{1}{2} \VERT \pp_n - \pp_{n-1} \VERT \leq 2^{-n}\VERT \pp_1 \VERT. 
\]
To get
\[
\VERT \pp_{n+1} \VERT \leq \eta,
\] 
it suffices to sum \eqref{inductionstep} for all the integers up to $n+1$ and \eqref{p1leqeta/2}. So we are done with our induction as soon as $\eps_0$ satisfies \eqref{condition1eps0}, \eqref{condition2eps0}, \eqref{condition3eps0}, \eqref{condition4eps0}.\\
\underline{Step three:} conclusion of the first and third point of the statement. 

We have shown that when $\eps_0$ is small enough, under condition \eqref{conditionetadelta0}, $(\pp_n)_{n \in \N}$ is a Cauchy sequence in the Banach space of functions having finite norm $||| \bullet |||$. So it converges to a certain $\pp$ which turns out to be a solution to equation \eqref{abstractintegraleq} and which belongs to $\boldsymbol{L}_0$ for all of its times of existence ($\boldsymbol{L}_0$ is closed even for the topology of distributions). Moreover, by summing \eqref{inductionstep} for all $n\geq 1$ and by using \eqref{estimp1} and \eqref{condition1eps0}, we get
\[
\VERT \pp \VERT \leq 2 \VERT \pp_1 \VERT \leq \frac{ \eta^2}{\eps_0}.
\]
Inequalities \eqref{estimsolution} and \eqref{estimsolutionbis} follow easily.\\
\underline{Step four:} continuity of $\pp$.

Here, we show the second point of the statement. First, because of \eqref{estiminductionabstract} (with $\pp_1 = \pp$ and $\pp_2 = 0$), \eqref{estimsolution} and \eqref{estimsolutionbis}, we can estimate $\mathcal{F}$: for all $t < \delta_0 / \Gamma$ and for all $\delta < \delta_0 - \Gamma t$, there is a constant $K$ such that 
\begin{equation}
\label{eq:estim_F_sol} 
\| \mathcal{F}_{\eta}(t, \pp(t)) \|_{\delta} \leq \frac{K}{(\delta_0 - \delta - \Gamma t)^{1/2}}.
\end{equation} 
(For this part of the proof, we do not need to be cautious with the dependences of $K$; \emph{a priori}, it depends on everything except for $t$ and $\delta$.) In what follows, $K$ will be a large constant growing from line to line.

Then, if $0 \leq t_1 < t_2 \leq \delta_0/\Gamma$, using \eqref{abstractintegraleq}, we get:
\begin{equation}
\label{eq:increment_p}
\pp(t_2) - \pp(t_1) = \int_0^{t_1} \{ S_{t_2 - s} - S_{t_1 - s}\} \mathcal{F}_{\eta}(s, \pp(s)) \D s + \int_{t_1}^{t_2} S_{t_2 - s} \mathcal{F}_{\eta}(s, \pp(s)) \D s
\end{equation}

Now, we chose $\delta< \delta_0$ and we suppose $0 \leq t_1 < t_2 \leq T := (\delta_0 - \delta)/\Gamma$. The goal is to show that the two terms in \eqref{eq:increment_p} tend to zero in $\XX_{\delta}$ when $t_2 - t_1$ goes to zero. The easiest term is the second one:
\begin{align*}
\left\| \int_{t_1}^{t_2} S_{t_2 - s} \mathcal{F}_{\eta}(s, \pp(s)) \D s \right\|_{\delta} &\leq \int_{t_1}^{t_2} \| S_{t_2 - s} \mathcal{F}_{\eta}(s, \pp(s))\|_{\delta} \D s\\
&\leq K  \int_{t_1}^{t_2} \| \mathcal{F}_{\eta}(s, \pp(s))\|_{\delta + \gamma_2 (t_2 - s)} \D s &\mbox{by \eqref{estimS}}, \\
&\leq K \int_{t_1}^{t_2} \frac{\D s}{(\delta_0 - \delta - \Gamma s - \gamma_2(t_2 - s))^{1/2}} &\mbox{by \eqref{eq:estim_F_sol}},\\
&= K\int_{t_1}^{t_2} \frac{\D s}{(\Gamma (T-s) - \gamma_2(t_2 - s))^{1/2}}\\
&\leq K\int_{t_1}^{t_2} \frac{\D s}{((\Gamma - \gamma_2) (T-s))^{1/2}}\\
&\leq K\int_{T - (t_2-t_1)}^{T} \frac{\D s}{((\Gamma - \gamma_2) (T-s))^{1/2}}\\
&\leq K \sqrt{t_2 - t_1},
\end{align*}
which tends to zero when $t_2 - t_1$ tends to zero.

We treat the first term in two stages: $t_2 \searrow t_1$ and then $ t_1 \nearrow t_2$. \\
\underline{The case $t_2 \searrow t_1$}. We have:
\begin{align*}
\int_0^{t_1} \{ S_{t_2 - s} - S_{t_1 - s}\} \mathcal{F}_{\eta}(s, \pp(s)) \D s &= \int_0^{t_1} \{ S_{t_2 - t_1} - \Id\}S_{t_1 - s} \mathcal{F}_{\eta}(s, \pp(s)) \D s \\
& = \{S_{t_2 - t_1} - \Id\} \pp(t_1).
\end{align*}
According to the continuity part of Theorem~\ref{sharpestimate}, this term tends to zero in $\XX_{\delta}$ as $t_2 \searrow t_1$ provided $\pp(t_1) \in \XX_{\delta'}$ for some $\delta' > \delta$. But this is the case as we know that $p(t_1) \in \XX_{\delta_0 - \Gamma t_1}$, and $\delta_0 - \Gamma t_1 > \delta_0 - \Gamma T = \delta$.\\
\underline{The case $t_1 \nearrow t_2$}. We also have:
\begin{equation*}
\int_0^{t_1} \{ S_{t_2 - s} - S_{t_1 - s}\} \mathcal{F}_{\eta}(s, \pp(s)) \D s = \int_0^{t_1} S_{t_1 - s}\{ S_{t_2 - t_1} - \Id\} \mathcal{F}_{\eta}(s, \pp(s)) \D s  
\end{equation*}
Consequently,
\begin{align}
\notag \Bigg\|\int_0^{t_1} \{ S_{t_2 - s} -& S_{t_1 - s}\} \mathcal{F}_{\eta}(s, \pp(s)) \D s  \Bigg\|_{\delta} \\
\notag & \leq \int_0^{t_1} \| S_{t_1 - s}\{ S_{t_2 - t_1} - \Id\} \mathcal{F}_{\eta}(s, \pp(s)) \|_{\delta}\D s\\
\notag &\leq \int_0^{t_1} \| \{ S_{t_2 - t_1} - \Id\} \mathcal{F}_{\eta}(s, \pp(s)) \|_{\delta + \gamma_2 (t_1 -s)}\D s &\mbox{by \eqref{estimS}},\\
\label{eq:t_1_to_t_2}&= \int_0^{t_2} \1_{s \leq t_1} \| \{ S_{t_2 - t_1} - \Id\} \mathcal{F}_{\eta}(s, \pp(s)) \|_{\delta + \gamma_2 (t_1 -s)}\D s.
\end{align}
On the one hand, if $s \in [0,t_2)$, $\mathcal{F}_{\eta}(s,\pp(s)) \in \XX_{\delta'}$ for all $\delta' < \delta_0 - \Gamma s$, and 
\begin{equation*}
\delta + \gamma_2 (t_2 - s) < \delta + \Gamma(t_2 - s) \leq \delta + \Gamma (T - s) = \delta_0 - \Gamma s.
\end{equation*}
So by the continuity part of Theorem~\ref{sharpestimate}
\begin{align*}
\1_{s \leq t_1} \| \{ S_{t_2 - t_1} - \Id\} \mathcal{F}_{\eta}(s, \pp(s)) \|_{\delta + \gamma_2 (t_1 -s)} &\leq \| \{ S_{t_2 - t_1} - \Id\} \mathcal{F}_{\eta}(s, \pp(s)) \|_{\delta + \gamma_2 (t_2 -s)}\\
&\hspace{-6pt}\underset{t_1 \nearrow t_2}{\longrightarrow} 0.
\end{align*}
On the other hand, for all $s \in [0,t_2)$, 
\begin{align*}
\1_{s \leq t_1} \|& \{ S_{t_2 - t_1} - \Id\} \mathcal{F}_{\eta}(s, \pp(s)) \|_{\delta + \gamma_2 (t_1 -s)} \\
& \leq \1_{s \leq t_1} \|  S_{t_2 - t_1}  \mathcal{F}_{\eta}(s, \pp(s)) \|_{\delta + \gamma_2 (t_1 -s)} + \1_{s \leq t_1} \|    \mathcal{F}_{\eta}(s, \pp(s)) \|_{\delta + \gamma_2 (t_1 -s)}\\
&\leq 2 \| \mathcal{F}_{\eta}(s, \pp(s)) \|_{\delta + \gamma_2 (t_2 -s)}\\
&\leq \frac{K}{ (\delta_0 - \delta - \Gamma s - \gamma_2 (t_2-s))^{1/2}}\\
&\leq \frac{K}{ ((\Gamma - \gamma_2)(T-s))^{1/2}}
\end{align*}
where we used \eqref{eq:estim_F_sol} to get the fourth line. This bound does not depend on $t_1$ and is summable between $0$ and $t_2 \leq T$. So the dominated convergence theorem applies, and we can pass to the limit $t_1 \nearrow t_2$ in \eqref{eq:t_1_to_t_2}.\\
\underline{Step five:} uniqueness.

Suppose $\eta \leq \eps_0$. The computations of the induction part of step two show that if $\qq_1$ and $\qq_2$ are two solutions to \eqref{abstractintegraleq} up to time $\delta_0 / \Gamma$ that satisfy
\begin{equation}
\label{eq:neighborhood}
\sup_{t \in [0,\delta_0 / \Gamma]} \| \qq_1(t), \qq_2(t) \|_{\delta_0 - \Gamma t} \leq \eta,
\end{equation}
then  
\begin{equation*}
\VERT \qq_2 - \qq_1 \VERT \leq \frac{1}{2} \VERT \qq_2 - \qq_1 \VERT.
\end{equation*} 
Consequently, in that case, $\qq_1(t) = \qq_2(t)$ for all $t \in [0, \delta_0 / \Gamma]$.

Moreover, we have seen that $\eps_0$ does not depend on $\delta_0$. So if we replace $\delta_0$ by some $\iota_0$ in \eqref{eq:neighborhood}, then the conclusion holds up to time $\iota_0 / \Gamma$.

Now take $\pp$ as built in step three and $\qq$, $T \leq \delta_0 / \Gamma$ and $\delta >0$ as in the statement of the theorem. Let $t_0 \in [0,T]$ be defined by:
\begin{equation*}
t_0 := \sup \{ t \in [0,T] \, | \, \pp(t) = \qq(t) \}.
\end{equation*}
Suppose by contradiction that $t_0 < T$. 

By the previous considerations, it suffices to find $\iota_0 \leq \delta_0$ such that $t_0 < \iota_0 / \Gamma \leq T$ and
\begin{equation}
\label{eq:estim_q}
\sup_{t \in [t_0,\iota_0 / \Gamma]} \| \qq \|_{\iota_0 - \Gamma t} \leq \eta.
\end{equation}
Indeed, if such a $\iota_0$ exists, then \eqref{eq:neighborhood} holds with $\iota_0$ instead of $\delta_0$, $\qq_1 := \pp$ and $\qq_2 := \qq$. (The estimate for $\pp$ and the estimate for $\qq$ before $t_0$ are due to $\iota_0 \leq \delta_0$ and $\VERT \pp \VERT \leq \eta$.) Hence, for all $t \leq \iota_0 / \Gamma$, $\qq(t) = \pp(t)$ and as $\iota_0 / \Gamma > t_0$, the maximality of $t_0$ is contradicted.

As $\pp(t_0) \in \boldsymbol{L}_0$, $\delta' \mapsto \| \pp \|_{\delta'}$ is an increasing function. So as
\begin{equation*}
\| \pp(t_0) \|_{\delta_0 - \Gamma t_0} \leq \eta,
\end{equation*}
for all $\iota_0 < \delta_0$,
\begin{equation*}
\| \qq(t_0) \|_{\iota_0 - \Gamma t_0} = \| \pp(t_0) \|_{\iota_0 - \Gamma t_0} < \eta.
\end{equation*}
In addition, if $\iota_0$ is sufficiently small, then $\iota_0 - \Gamma t_0 \leq \delta$, where $\delta$, given in the statement of the theorem, is such that $\qq$ is continuous in $\XX_{\delta}$. For such $\iota_0$ there is $t_1 > t_0$ such that for all $t \in [t_0,t_1]$,
\begin{equation*}
\| \qq(t) \|_{\iota_0 - \Gamma t} \leq \| \qq(t) \|_{\iota_0 - \Gamma t_0} \leq  \eta.
\end{equation*}
Up to taking an even lower $\iota_0 > t_0  \Gamma$ we can suppose furthermore that $\iota_0/ \Gamma \leq t_1$. For such a $\iota_0$, \eqref{eq:estim_q} holds and the result follows.
\end{proof} 
\subsection{Conclusion: proof of Theorem \ref{existenceofsolutions}}
Theorem \ref{existenceofsolutions} is a direct application of Theorem \ref{cauchykovalevskaia}. 

The fact that $(\rhorho(t), \vv(t))$ stays real is a consequence of \eqref{Areal}. Indeed, with this assumptions, on the one hand, with the notations of the proof of Lemma \ref{fourierlemma}:
\[
\forall n \in \Z^d, \quad \boldsymbol{A}_{-n} = \overline{\boldsymbol{A}_n} \quad \mbox{ and } \quad \boldsymbol{B}_{-n} = \overline{\boldsymbol{B}_n}.
\]
So the linear solutions are real if $(\rhorho_0, \vv_0)$ is real. On the other hand, the fixed point procedure developed in the proof of theorem \ref{cauchykovalevskaia} send real functions on real functions. So the nonlinear solutions also stay real.

The fact that $\rhorho(t)$ stays nonnegative is a classical fact in the theory of the continuity equation and can be understood for example through the characteristics method.

\qed
\section{Consequences}
\label{consequences}
In this section, we will give some consequences of our results. First, we will prove that the solution to equations \eqref{abstracteq}-\eqref{abstractforce} are almost Lyapounov unstable in the neighbourhood of any linearly unstable stationary profile. Theorem \ref{kineticstatementVP} for the Vlasov-Poisson equation will be a direct application of this result. Then, we show an ill-posedness result implying Theorem \ref{kineticstatementKEuVB} when the unstable spectrum grows linearly with the frequency of the exponential growing modes, as it does in the kinetic Euler equation and in the Vlasov-Benney equation.
\subsection{Almost Lyapounov instability}
\label{sectionalmostlyapounov}
Take $\mu$ an unstable profile, as defined in Definition \ref{def:unstability}. We consider $\gamma_0$ as defined in \eqref{defgamma0}. We also take $\gamma \in (0, \gamma_0)$ and 
\[
\Gamma := 2\gamma_0 - \gamma,
\]
chosen so that $\Gamma - \gamma_0 = \gamma_0 - \gamma$.

From now on, we take $(n, \lambda) \in \Z^d\backslash\{0\} \times \C$ such that \eqref{penrosecondition} holds and such that $\Re(\lambda) / |n| \in [\gamma, \gamma_0]$.

Because of Subsection \ref{spectralstudy}, taking the real part of the exponential growing mode associated to $(n, \lambda)$, and using the notations: 
\begin{equation}
\label{notationseigenvalues}
u := \frac{n}{|n|}, \quad  r + i \varphi := \frac{\lambda}{|n|}, \quad \theta(w) := \arctan\left(\frac{\varphi + u \cdot w}{r}\right),
\end{equation}
we obtain that for all $c \in \R$,
\begin{equation}
\label{explicitlinearsolution}
\left\{
\begin{aligned}
r_c^w(t,x) &:= - c \frac{\cos\Big[n\cdot x + |n| \varphi t - 2 \theta(w)\Big]}{r^2 + (\varphi + u \cdot w)^2} \exp(|n| r t),\\
u_c^w(t,x) &:= c \frac{\sin\Big[n\cdot x + |n| \varphi t - \theta(w)\Big]}{\sqrt{ r^2 + (\varphi + u \cdot w)^2} } \exp(|n| r t) \times u,
\end{aligned}
\right.
\end{equation}
is a real solution to the linearized system \eqref{abstractlineq}. Its initial data is clearly in $\boldsymbol{L}_0$. We see thanks to remark \ref{eigenpotential} that it corresponds to the linear potential
\begin{equation}
\label{linearpotential}
V[\rr_c(t), \uu_c(t)](x) = c \cos( n \cdot x + |n|\varphi t) \exp(|n|r t).
\end{equation}
We deduce our "almost Lyapounov instability" result.
\begin{Thm}[Almost Lyapounov instability]
\label{almostLyapounovThm}
Take $s \in \N $ and $\alpha \in (0,1]$. Then there exists $(\rhorho^k_0, \vv^k_0)_{k \in \N}$ a family of analytic initial data tending uniformly in $w$ towards the stationary solution $(\boldsymbol{1}, \boldsymbol{w})$ in $W^{s, \infty}$, $(T_k)_{k \in \N}$ a family of positive times tending to $+ \infty$ such that for all $k \in \N$, the unique analytic solution $(\rhorho^k, \vv^k)$ to \eqref{abstracteq}-\eqref{abstractforce} starting from $(\rhorho^k_0, \vv^k_0)$, is defined up to time $T_k$, and satisfies
\begin{equation}
\label{almostLyapounov}
\frac{\displaystyle{\int} \min \Big(  \| \rho^{k,w} - 1 \|_{L^1((0,T_k) \times \T^d)}, \| v^{k, w} - w \|_{L^1((0,T_k) \times \T^d)} \Big)\D \mu(w)}{\displaystyle{\sup_{w \in \R^d}} \Big\{ \max \Big(  \| \rho^{k,w}_0 - 1 \|_{W^{s, \infty}}, \| v^{k, w}_0 - w \|_{W^{s, \infty}} \Big)\Big\}^{\alpha}} \underset{k \to + \infty}{\longrightarrow} + \infty.
\end{equation}
Moreover, we have the following asymptotics for the potential:
\begin{equation}
\label{growthpotential}
\frac{\|U[\rhorho^k, \vv^k] - U[\boldsymbol{1}, \boldsymbol{w}]\|_{L^1((0,T_k)\times \T^d)}}{\displaystyle{\sup_{w \in \R^d}} \Big\{ \max \Big(  \| \rho^{k,w}_0 - 1 \|_{W^{s, \infty}}, \| v^{k, w}_0 - w \|_{W^{s, \infty}} \Big)\Big\}^{\alpha}} \underset{k \to + \infty}{\longrightarrow} + \infty,
\end{equation}
where $(\boldsymbol{1}, \boldsymbol{w})$ is a notation for the homogeneous stationary solution.

Moreover
\begin{equation}
\label{eq:estim_Tk_VP_abstract}
\begin{gathered}
T_k  \underset{k \to + \infty}{\sim}  |\log \eps_k | \\
\mbox{with} \quad \eps_k := \sup_{w \in \R^d} \Big\{ \max \Big(  \| \rho^{k,w}_0 - 1 \|_{L^1}, \| v^{k, w}_0 - w \|_{L^1} \Big)\Big\} .
\end{gathered}
\end{equation}
\end{Thm}
\begin{Rem}
\label{remalmostLyapounov}
\begin{itemize}
\item A classical Lyapounov instability result would mean some discontinuity of the numerator of \eqref{almostLyapounov} in the topology generated by the norm in the denominator. Here we show instead that the numerator cannot be H\"older continuous with any H\"older exponent with respect to the denominator. That is why we call this result almost Lyapounov instability.
\item Equation \eqref{growthpotential} shows that the instability does not come from the multiphasic representations of the solutions. Indeed, the potential does not depend on this representation.
\end{itemize}
\end{Rem}
\begin{proof}
Chose $\gamma < \gamma_0$ sufficiently close to $\gamma_0$ to have
\begin{equation}
\label{defGamma}
\alpha > 1 - \left( 1- \frac{\alpha}{2} \right)\frac{\gamma}{\Gamma},
\end{equation}
with $\Gamma := \gamma_0 + (\gamma_0 - \gamma)$. Now take $(n, \lambda)$ as in the beginning of the subsection, $(c_k)_{k \in \N} \in (0,1]^{\N}$ a sequence converging to $0$, satisfying:
\[
c_k |n| < c_k^{\alpha/2}.
\]
Take $T_k>0$ the unique positive number such that
\begin{equation}
\label{defTk}
c_k |n| \exp(|n|\Gamma T_k) = c_k^{\alpha/2}.
\end{equation} 
Then we define $(\rr^k, \uu^k) := (\uu_{c_k}, \rr_{c_k})$ using \eqref{explicitlinearsolution}. Remark that
\[
\| \DD \rr^k_0, \DD \uu^k_0 \|_{\Gamma T_k} \lesssim c_k |n| \exp(|n|\Gamma T_k) = c_k^{\alpha/2} \underset{k \to + \infty}{\longrightarrow} 0.
\]
(The symbol $\lesssim$ means "lower than up to a constant which is independent of $k$".) So when $k$ is sufficiently large, condition \eqref{initialdatasmall} of Theorem \ref{existenceofsolutions} is satisfied with $\delta_0 = \Gamma T_k$. As a consequence, the unique analytic solution $(\rhorho^k, \vv^k)$ to \eqref{abstracteq}-\eqref{abstractforce} starting from the initial data $(\rhorho^k_0, \vv^k_0) =( \boldsymbol{1} + \rr^k_0,\boldsymbol{w} + \uu^k_0)$ is well defined up to time $T_k$, and is of the form
\[
\forall w \in \R^d, \qquad \rho^{k, w} = 1 + r^{k,w} + \sigma^{k,w} \quad \mbox{and} \quad v^{k,w} = w + u^{k,w} + \xi^{k,w}.
\]
Moreover, $(\sigsig^k, \xixi^k)$ satisfies the estimate \eqref{estimsigmaxi}, which gives:
\begin{equation}
\label{estimsigmaxialmostLyapounov}
\sup_{t \leq T_k} \| \sigsig(t), \xixi(t) \|_{\Gamma (T_k - t)} \lesssim c_k^2 \exp(2|n|\Gamma T_k).
\end{equation}

Let us move on to the proof of the asymptotics \eqref{almostLyapounov}. First, we estimate $(\rhorho_0^k, \vv_0^k)$. It is given explicitly by $(\boldsymbol{1} + \rr^k, \boldsymbol{w} + \uu^k)$ and formula \eqref{explicitlinearsolution} with $t=0$. We deduce:
\begin{equation}
\label{estiminitialdata}
\sup_{w \in \R^d} \max \Big(  \| \rho^{k,w}_0 - 1 \|_{W^{s, \infty}}, \| v^{k, w}_0 - w \|_{W^{s, \infty}} \Big) \underset{k \to + \infty}{\sim} c_k.
\end{equation}

Now we have to estimate $(\rhorho^k, \vv^k)$ in $L^1((0,T_k) \times \T^d)$ (we denote by $\| \bullet \|_{L^1}$ its norm to lighten the notations). First, remark that for all $w \in \R^d$,
\begin{align*}
\min \Big(  &\| \rho^{k,w} - 1 \|_{L^1}, \| v^{k, w} - w \|_{L^1} \Big) \\
&\geq \min\Big( \| r^{k,w} \|_{L^1} - \| \sigma^{k,w} \|_{L^1}, \| u^{k,w} \|_{L^1} - \| \xi^{k,w} \|_{L^1} \Big)\\
&\geq \min\Big( \| r^{k,w} \|_{L^1} , \| u^{k,w} \|_{L^1} \Big) - \max\Big( \| \sigma^{k,w} \|_{L^1},  \| \xi^{k,w} \|_{L^1} \Big).
\end{align*}
But now, with formula \eqref{explicitlinearsolution} and equation \eqref{defTk}, we easily see that
\begin{align}
\notag \int \min\Big( \| r^{k,w} \|_{L^1} ,& \| u^{k,w} \|_{L^1} \Big)\D \mu(w)\\
\notag &\gtrsim c_k \exp(|n| r T_k) - c_k \\
\notag &\gtrsim c_k \exp( |n| \gamma T_k) - c_k \\
\notag &\gtrsim c_k \exp(|n| \Gamma T_k)^{\gamma / \Gamma} - c_k \\
\notag &\gtrsim c_k^{1 - (1 - \alpha/2)\gamma/\Gamma} - c_k \\
\label{estimlinear} &\gtrsim c_k^{1 - (1 - \alpha/2)\gamma/\Gamma} + \underset{k \to + \infty}{o}\Big(c_k^{1 - (1 - \alpha/2)\gamma/\Gamma}\Big).
\end{align}
On the other hand, with the help of \eqref{defdoublenorm}, \eqref{defnormcouple}, \eqref{estimsigmaxialmostLyapounov} and Lemma \ref{lemdeltatodelta'}, we get that for all $t \in [0, T_k]$,
\begin{align*}
\sup_{w \in \R^d}\max\Big(\| \sigma^{k,w}(t) &\|_{L^1(\T^d)}, \| \xi^{k,w}(t) \|_{L^1(\T^d)}\Big) \\
&\leq \| \sigsig^k(t), \xixi^k(t) \|_0 \\
&\leq \exp\big( - \Gamma(T_k - t)  \big)\| \sigsig^k(t), \xixi^k(t) \|_{\Gamma (T_k - t)}\\
&\lesssim c_k^2 \exp\big (\Gamma (2 |n| - 1) T_k \big) \exp(\Gamma t).
\end{align*}
Integrating over time, we get by \eqref{defGamma}
\begin{align}
\notag \sup_{w \in \R^d} \max\Big(\| \sigma^{k,w} \|_{L^1}, \| \xi^{k,w} \|_{L^1}\Big) &\lesssim c_k^2 \exp(2 \Gamma |n| T_k)\\
\notag &\lesssim c_k^{\alpha}\\
\label{estimrest}&=\underset{k \to + \infty}{o}\Big(c_k^{1 - (1 - \alpha/2)\gamma/\Gamma}\Big).
\end{align}
Gathering \eqref{estiminitialdata}, \eqref{estimlinear} and \eqref{estimrest}, we get
\begin{align*}
&\frac{ \displaystyle{\int} \min \Big(  \| \rho^{k,w} - 1 \|_{L^1((0,T_k) \times \T^d)}, \| v^{k, w} - w \|_{L^1((0,T_k) \times \T^d)} \Big)\D \mu(w)}{\displaystyle{\sup_{w \in \R^d}} \Big\{ \max \Big(  \| \rho^{k,w}_0 - 1 \|_{W^{s, \infty}}, \| v^{k, w}_0 - w \|_{W^{s, \infty}} \Big)\Big\}^{\alpha}}\\
&\hspace{2cm} \gtrsim \frac{c_k^{1 - (1-\alpha/2) \gamma / \Gamma} + \underset{k \to + \infty}{o}\Big(c_k^{1 - (1 - \alpha/2)\gamma/\Gamma}\Big) }{c_k^{\alpha}}\\
&\hspace{1.7cm} \underset{k \to + \infty}{\longrightarrow} + \infty,
\end{align*}
by \eqref{defGamma}, which gives \eqref{almostLyapounov}.

To prove \eqref{growthpotential}, remark that for all $k \in \N$,
\begin{equation}
\label{decompositionofpotential}
\begin{aligned}
U[\rhorho^k, \vv^k ] &= U[\boldsymbol{1}, \boldsymbol{w}] + V[\rr^k, \uu^k] + V[\sigsig^k, \xixi^k]\\
&\hspace{5pt}+ \underbrace{A \int \{ \Phi(w + u^{k,w} + \xi^{k,w}) - \Phi(w) \} (r^{k,w} + \sigma^{k,w} ) \D \mu(w)}_{W_1}\\
&\hspace{5pt}+ \underbrace{A \int \{ \Phi(w + u^{k,w} + \xi^{k,w}) - \Phi(w) - \D \Phi(w)\cdot (u^{k,w} + \xi^{k,w})  \}\D \mu(w)}_{W_2}. 
\end{aligned}
\end{equation}

On the one hand, $V[\rr^k, \uu^k]$ is given by \eqref{linearpotential} and we compute easily as in \eqref{estimlinear} that
\begin{equation}
\label{estimpotentiallinearpart}
\| V[\rr^k, \uu^k] \|_{L^1} \gtrsim c_k^{1 - (1 - \alpha/2)\gamma/\Gamma} + \underset{k \to + \infty}{o}\Big(c_k^{1 - (1 - \alpha/2)\gamma/\Gamma}\Big).
\end{equation}

On the other hand, using the definition of $V$ in \eqref{abstractlineq} and estimates \eqref{intPhi}, \eqref{intdPhi} and \eqref{assumptionP}, 
\begin{align}
\notag \| V[\sigsig^k, \xixi^k] \|_{L^1} &\lesssim \sup_{w \in \R^d} \max\Big( \| \sigma^{k,w} \|_{L^1} , \| \xi^k \|_{L^1} \Big)\\
\label{estimpotentiallinearrest} &\lesssim c_k^{\alpha} = \underset{k \to + \infty}{o}\Big(c_k^{1 - (1 - \alpha/2)\gamma/\Gamma}\Big),
\end{align}
by \eqref{estimrest}.

Let us show how to treat $W_1$ defined in \eqref{decompositionofpotential}, $W_2$ being treated in the same way and satisfying the same estimate. In the second line, we use \eqref{assumptionP} and Proposition \ref{product}, in the third line, we use \eqref{Phi-Phi}, in the fourth line, we use \eqref{intdPhiseriesbis}, and finally in the last line, we use the same arguments as for \eqref{estimrest}:
\begin{align}
\notag \| W_1 \|_{L^1} &\leq \int_0^{T_k} \Big| A \int \{ \Phi(w + u^{k,w} + \xi^{k,w}) - \Phi(w) \} (r^{k,w} + \sigma^{k,w} ) \D \mu(w) \Big|_0 \D t\\
\notag &\lesssim\int_0^{T_k} \int | \Phi(w + u^{k,w} + \xi^{k,w}) - \Phi(w) |_0 |r^{k,w} + \sigma^{k,w} |_0 \D \mu(w) \D t \\
\notag &\lesssim \int_0^{T_k} \sup_{w \in \R^d}|r^{k,w} + \sigma^{k,w} |_0|u^{k,w} + \xi^{k,w} |_0 \D t \int |\Phi|'(|w| + r_0) \D \mu(w)\\
\label{estimW1} &\lesssim c_k^{\alpha} = \underset{k \to + \infty}{o}\Big(c_k^{1 - (1 - \alpha/2)\gamma/\Gamma}\Big).
\end{align}
We get \eqref{growthpotential} by gathering \eqref{estimpotentiallinearpart}, \eqref{estimpotentiallinearrest}, \eqref{estimW1} and \eqref{estiminitialdata}.

Finally, \eqref{eq:estim_Tk_VP_abstract} is a consequence of \eqref{defTk} and the explicit estimate
\begin{equation*}
\sup_{w \in \R^d} \max \Big(  \| \rho^{k,w}_0 - 1 \|_{L^1}, \| v^{k, w}_0 - w \|_{L^1} \Big) \underset{k \to + \infty}{\sim} c_k.
\end{equation*}
\end{proof}
We can now go back to the kinetic formulation and give a corollary which implies Theorem \ref{kineticstatementVP}. We recall that in Theorem \ref{kineticstatementVP}, we suppose that we control a certain number of macroscopic observables at the initial time. But with \eqref{multiphasictokinetic}, it is easy to go from a multiphasic representation of the system to macroscopic observables. This remark leads to the following statement.
\begin{Cor}
\label{almostLyapounovkinetic}
Take $\mu$ an unstable profile, $N \in \N^*$, $\varphi_1, \dots, \varphi_N \in C_c^{\infty}(\R^d)$, $s \in \N$ and $\alpha \in (0,1]$. Take $(\rhorho^k_0, \vv^k_0)_{k \in \N}$ and $(T_k)_{k \in \N}$ as in the previous theorem, and $(\rhorho^k, \vv^k)_{k \in \N}$ the corresponding solutions. For each $i$, call
\[
\cg f_0^k, \varphi_i \cd := \int \varphi_i (v_0^{k,w}) \rho_0^{k,w} \D \mu(w).
\] 
Then, we have:
\begin{equation}
\label{eq:growth_potential_abstract_corollary}
\frac{\|U[\rhorho^k, \vv^k] - U[\boldsymbol{1}, \boldsymbol{w}]\|_{L^1((0,T_k)\times \T^d)}}{\sum_{i=1}^N \| \cg f^k_0 , \varphi_i \cd - \cg \mu, \varphi_i\cd \|_{W^{s, \infty}(\T^d)}^{\alpha}} \underset{k \to + \infty}{\longrightarrow} + \infty.
\end{equation}
Moreover,
\begin{equation}
\label{eq:estim_Tk_abstract_corollary}
T_k \underset{k \to + \infty}{\sim}  |\log \eps_k |,
\end{equation}
where $\eps_k := \|U[\rhorho^k, \vv^k]|_{t=0} - U[\boldsymbol{1}, \boldsymbol{w}]\|_{L^1}$.
\end{Cor}
\begin{proof}
In view of \eqref{growthpotential}, to prove \eqref{eq:growth_potential_abstract_corollary} it suffices to show that if $\varphi \in C_c^{\infty}(\R^d)$, there exists $C >0$ such that for all smooth $(\rhorho, \vv)$,
\begin{align*}
\bigg\| \int \varphi(v^w)\rho^w &\D \mu(w) - \int \varphi(w) \D \mu(w) \bigg\|_{W^{s, \infty}}\\
& \leq C \sup_{w \in \R^d} \Big\{ \max \Big(  \| \rho^w - 1 \|_{W^{s, \infty}}, \| v^w - w \|_{W^{s, \infty}} \Big)\Big\}. 
\end{align*}
This is an easy consequence of the following decomposition:
\begin{align*}
 \int \varphi(v^w)\rho^w &\D \mu(w) - \int \varphi(w) \D \mu(w) \\
  &= \int \{ \varphi(v^w)\rho^w - \varphi(w) \} \D \mu(w) \\
 &= \int \varphi(v^w)\{ \rho^w - 1\} \D \mu(w) + \int \{ \varphi(v^w) - \varphi(w) \}\D \mu(w).
\end{align*}
To prove \eqref{eq:estim_Tk_abstract_corollary}, just remark that because of \eqref{linearpotential}, taking $(c_k)$ as in the previous proof,
\begin{equation*}
\|U[\rhorho^k, \vv^k]|_{t=0} - U[\boldsymbol{1}, \boldsymbol{w}]\|_{L^1} \underset{k \to + \infty}{\sim} \sup_{w \in \R^d} \max \Big(  \| \rho^{k,w}_0 - 1 \|_{L^1}, \| v^{k, w}_0 - w \|_{L^1} \Big) \underset{k \to + \infty}{\sim} c_k.
\end{equation*}
\end{proof}
\subsection{Ill-posedness when the spectrum is highly unbounded}
\label{paragraphhiglyunbounded}
With an additional assumption, we can show an ill-posedness result for equations \eqref{abstracteq}-\eqref{abstractforce}. The assumption is the following.

\subsubsection*{Assumption on the structure of the unbounded spectrum} We assume that the number $\gamma_0$ defined in \eqref{defgamma0} satisfies
\begin{equation}
\label{unboundedspectrum}
\gamma_0 = \limsup_{ |n| \to + \infty } \sup_{\lambda \in S_n} \frac{\Re(\lambda)}{|n|}.
\end{equation}
(We recall that $S_n$ is defined in \eqref{defSn}.) This assumption means that there exist exponential growing modes of frequency $n$ with growing rates of order $|n| \gamma_0$ for arbitrary large $|n|$.
 \subsubsection*{Examples} The kinetic Euler system \eqref{KEu'} and the Vlasov-Benney system \eqref{VB} satisfy the following property: if $n \in \Z^d$ and $\lambda \in \C$ are such that $\lambda \in S_n$, then for all $k \in \N^*$, $k \lambda \in S_{kn}$. This can be directly checked using the Penrose conditions \eqref{penroseKEumeasure} and \eqref{penroseVBmeasure}. As a consequence, for all $n \in \Z^d$ and $k \in \N$,
 \[
 k S_n \subset S_{kn}.
 \]
 Equation \eqref{unboundedspectrum} follows easily.
 This property is a consequence of the following scaling for \eqref{KEu} and \eqref{VB}:
 \[
 \mbox{if } f(t,x,v) \mbox{ is a solution and }k \in \N,\mbox{ then } f(kt, kx, v) \mbox{ is also a solution.}
 \]
 
Under this assumption, the instability proved at Theorem \ref{almostLyapounovThm} is true even in small times.
\begin{Thm}
\label{illposednesstheorem}
Take $\mu$ an unstable profile satisfying assumption \eqref{unboundedspectrum}, $s \in \N$ and $\alpha \in (0,1]$. Then there exists $(\rhorho^k_0, \vv^k_0)_{k \in \N}$ a family of analytic initial data tending uniformly in $w$ towards the stationary solution $(\boldsymbol{1}, \boldsymbol{w})$ in $W^{s, \infty}$, $(T_k)_{k \in \N}$ a family of positive times \textbf{tending to zero} such that for all $k \in \N$, the unique analytic solution $(\rhorho^k, \vv^k)$ to \eqref{abstracteq}-\eqref{abstractforce} starting from $(\rhorho^k_0, \vv^k_0)$, is defined up to time $T_k$, and satisfies \eqref{almostLyapounov} and \eqref{growthpotential}.

Moreover
\begin{equation}
\label{eq:estim_Tk_KEu_VB_abstract}
T_k \underset{k \to + \infty}{\sim} \left( \frac{ |\log \eps_k |}{|n_k|} \right),
\end{equation}
where
\begin{equation*}
\eps_k := \sup_{w \in \R^d} \Big\{ \max \Big(  \| \rho^{k,w}_0 - 1 \|_{L^1}, \| v^{k, w}_0 - w \|_{L^1} \Big)\Big\} 
\end{equation*}
and where $n_k$ is the spatial frequency of the nearest exponential growing mode.
\end{Thm}
\begin{proof}
The proof is very similar to the one of Theorem \ref{almostLyapounovThm}, except that here the eigenvalue depends on $k$. Thanks to assumption \eqref{unboundedspectrum}, we choose $(n_k, \lambda_k)_{k \in \N}$ a family of solutions to \eqref{penrosecondition} with
\begin{gather*}
|n_k| \underset{k \to + \infty}{\longrightarrow} + \infty,\\
\forall k \in \N, \quad r_k := \frac{\Re(\lambda_k)}{|n_k|} \underset{k \to + \infty}{\longrightarrow} \gamma_0.
\end{gather*}
Now we take for all $k$:
\begin{equation}
\label{defckbeta}
c_k: = \frac{1}{|n_k|^{\beta}}\underset{k \to + \infty}{\longrightarrow} 0, \quad \beta := \frac{5}{\alpha}(\alpha s + 2).
\end{equation}
We take $\gamma$ sufficiently close to $\gamma_0$ to have (with $\Gamma = 2 \gamma_0 - \gamma$)
\begin{align}
\label{defGammanumdomin}
a := 1 + \frac{1}{\beta} - \frac{\gamma}{\Gamma} \left( 1 - \frac{\alpha}{4} - \frac{1}{\beta} \right) &< \alpha\left( 1 - \frac{s}{\beta} \right),\\
\label{defGammalindomin} a = 1 + \frac{1}{\beta}- \frac{\gamma}{\Gamma} \left( 1 - \frac{\alpha}{4} - \frac{1}{\beta} \right) &< \frac{\alpha}{2}.
\end{align}
We suppose up to forgetting the first terms that for all $k \in \N$,
\begin{equation}
\label{rkgeqgamma}
r_k \geq \gamma.
\end{equation}
Using the notations \eqref{notationseigenvalues} indexed by $k$ for $(n_k, \lambda_k)$, we define
\begin{equation}
\label{explicitlinearsolutionk}
\left\{
\begin{aligned}
r^{k,w}(t,x) &:= - c_k \frac{\cos\Big[n_k\cdot x + |n_k| \varphi_k t - 2 \theta_k(w)\Big]}{r_k^2+ (\varphi_k + u_k \cdot w)^2} \exp(|n_k| r_k t),\\
u^{k,w}(t,x) &:= c_k \frac{\sin\Big[n_k\cdot x + |n_k| \varphi_k t - \theta_k(w)\Big]}{\sqrt{ r_k^2 + (\varphi_k + u_k \cdot w)^2} } \exp(|n_k| r_k t) \times u_k.
\end{aligned}
\right.
\end{equation}
Remark that 
\[
\frac{c_k |n_k|}{c_k^{\alpha/4}} = c_k^{1 - \alpha/\{4(\alpha s + 1)\} - \alpha/4} \leq c_k^{1 - \alpha / 2} \underset{k \to + \infty}{\longrightarrow} 0.
\]
Consequently, we can suppose that for all $k \in \N$,
\[
c_k |n_k| < c_k^{\alpha/4}.
\]
In that case, we take $T_k$ the unique positive number satisfying
\begin{equation}
\label{defTkbis}
c_k |n_k| \exp(|n_k| \Gamma T_k) = c_k^{\alpha / 4}.
\end{equation}
Remark that automatically, as $1 < \beta(1-\alpha/4)$,
\[
T_k \underset{k \to + \infty}{\sim} \left( \frac{|\log c_k |}{|n_k|} \right) \underset{k \to + \infty}{\longrightarrow} 0.
\]
Formula \eqref{eq:estim_Tk_KEu_VB_abstract} follows easily.
Then,
\[
\| \DD \rr^k_0, \DD \uu^k_0 \|_{\Gamma T_k} \lesssim c_k |n_k| \exp(|n_k|\Gamma T_k) = c_k^{\alpha/4} \underset{k \to + \infty}{\longrightarrow} 0.
\]
 So when $k$ is sufficiently large, condition \eqref{initialdatasmall} of Theorem \ref{existenceofsolutions} is satisfied with $\delta_0 = \Gamma T_k$. As a consequence, the unique analytic solution $(\rhorho^k, \vv^k)$ to \eqref{abstracteq}-\eqref{abstractforce} starting from the initial data $(\rhorho^k_0, \vv^k_0) =( \boldsymbol{1} + \rr^k_0,\boldsymbol{w} + \uu^k_0)$ is well defined up to time $T_k$, and is of the form
\[
\forall w \in \R^d, \qquad \rho^{k, w} = 1 + r^{k,w} + \sigma^{k,w} \quad \mbox{and} \quad v^{k,w} = w + u^{k,w} + \xi^{k,w}.
\]
Moreover, $(\sigsig^k, \xixi^k)$ satisfies the estimate \eqref{estimsigmaxi}, which gives:
\begin{equation}
\label{estimsigmaxiillposedness}
\sup_{t \leq T_k} \| \sigsig(t), \xixi(t) \|_{\Gamma (T_k - t)} \lesssim c_k^2 |n_k|^2\exp(2|n_k|\Gamma T_k).
\end{equation}

In this context \eqref{estiminitialdata} becomes (using \eqref{defckbeta}):
\begin{equation}
\label{estiminitialdatabis} \sup_{w \in \R^d}\max \Big(  \| \rho^{k,w}_0 - 1 \|_{W^{s, \infty}}, \| v^{k, w}_0 - w \|_{W^{s, \infty}} \Big) \sim c_k|n_k|^s = c_k^{1 - s/\beta}.
\end{equation}
Equation \eqref{estimlinear} becomes (thanks to \eqref{rkgeqgamma} and \eqref{defTkbis} and the definition of $a$ in \eqref{defGammanumdomin})
\begin{align*}
\int \min\Big( \| r^{k,w} \|_{L^1} , \| u^{k,w} \|_{L^1} \Big) \D \mu(w) &\gtrsim \frac{ c_k \exp(|n| \Gamma T_k)^{\gamma / \Gamma} - c_k }{|n_k|} \\
&\gtrsim c_k^a - c_k^{1 + 1/\beta}.
\end{align*}
But clearly $a < 1 + 1/\beta$, so that
\begin{equation}
\label{estimlinearbis}
\int \min\Big( \| r^{k,w} \|_{L^1} , \| u^{k,w} \|_{L^1} \Big) \D \mu(w) \gtrsim c_k^a + \underset{k \to + \infty}{o}( c_k^a ).
\end{equation}
Finally, \eqref{estimrest} becomes because of \eqref{defTkbis}, \eqref{defGammalindomin} and \eqref{estimsigmaxiillposedness}:
\begin{align}
\notag \max\Big(\| \sigma^{k,w} \|_{L^1}, \| \xi^{k,w} \|_{L^1}\Big) &\lesssim c_k^2 |n_k|^2 \exp(2 \Gamma |n_k| T_k) \\
\label{estimrestbis} &\lesssim c_k^{\alpha/2} = \underset{k \to + \infty}{o}( c_k^a ),.
\end{align}
Gathering \eqref{estiminitialdatabis}, \eqref{estimlinearbis} and \eqref{estimrestbis}, we get
\begin{align*}
&\frac{\displaystyle{\int} \min \Big(  \| \rho^{k,w} - 1 \|_{L^1((0,T_k) \times \T^d)}, \| v^{k, w} - w \|_{L^1((0,T_k) \times \T^d)} \Big)\D \mu(w)}{\displaystyle{\sup_{w \in \R^d}}\Big\{ \max \Big(  \| \rho^{k,w}_0 - 1 \|_{W^{s, \infty}}, \| v^{k, w}_0 - w \|_{W^{s, \infty}} \Big)\Big\}^{\alpha}}\\
&\hspace{2cm} \gtrsim \frac{c_k^a + \underset{k \to + \infty}{o}( c_k^a ) }{c_k^{\alpha(1-s/\beta) }} \underset{k \to + \infty}{\longrightarrow}+ \infty,
\end{align*}
using \eqref{defGammanumdomin} in the last line. Estimate \eqref{growthpotential} is proved in the exact same way as in the previous proof.
\end{proof}
As in the previous subsection, Theorem \ref{illposednesstheorem} has a kinetic counterpart. The following corollary implies Theorem \ref{kineticstatementKEuVB} in the Vlasov-Benney case. For the kinetic Euler case, the next subsection (in particular Theorem \ref{rigorousstatementKEu}) is also needed.

\begin{Cor}
\label{illposednesskinetic}
Take $\mu$ an unstable profile satisfying assumption \ref{unboundedspectrum}, $N \in \N^*$, $\varphi_1, \dots, \varphi_N \in C_c^{\infty}(\R^d)$, $s \in \N$ and $\alpha \in (0,1]$. Take $(\rhorho^k_0, \vv^k_0)_{k \in \N}$ and $(T_k)_{k \in \N}$ as in the previous theorem (in particular
$(T_k)$ converges to zero), and $(\rhorho^k, \vv^k)_{k \in \N}$ the corresponding solutions. For each $i$, call
\[
\cg f_0^k, \varphi_i \cd := \int \varphi_i (v_0^{k,w}) \rho_0^{k,w} \D \mu(w).
\] 
Then, we have:
\[
\frac{\|U[\rhorho^k, \vv^k] - U[\boldsymbol{1}, \boldsymbol{w}]\|_{L^1((0,T_k)\times \T^d)}}{\sum_{i=1}^N \| \cg f^k_0 , \varphi_i \cd - \cg \mu, \varphi_i\cd \|_{W^{s, \infty}(\T^d)}^{\alpha}} \underset{k \to + \infty}{\longrightarrow} + \infty.
\]
Moreover
\begin{equation*}
T_k \underset{k \to + \infty}{\sim} \left( \frac{|\log \eps_k |}{|n_k|} \right),
\end{equation*}
where $\eps_k := \|U[\rhorho^k, \vv^k]|_{t=0} - U[\boldsymbol{1}, \boldsymbol{w}]\|_{L^1}$ and $n_k$ is the spatial frequency of the nearest exponential growing mode.
\end{Cor}
\begin{proof}
The proof is the same than the one of Corollary \ref{almostLyapounovkinetic}.
\end{proof}
\subsection{The specific case of the kinetic Euler equation}
\label{sectionkineticEuler}
As already said in the introduction, in the case of the kinetic Euler equation, our abstract framework let us solve \eqref{KEu'}, but not \eqref{KEu}: our method applies when we have a formula for the force field, and not when it is defined through a constraint. So an argument must be added to prove Theorem \ref{kineticstatementKEuVB}. It is done in three steps. 

First we will show in Theorem \ref{KEu'impliesKEu} that the measure-valued solutions to \eqref{KEu'} built in Theorem \ref{existenceofsolutions} are in fact measure-valued solutions to \eqref{KEu} provided the initial data satisfies \eqref{initialdataconstraint} (which makes sense in a measure-valued setting).

Unfortunately, the initial data used in the proof of Theorem \ref{illposednesstheorem} (the initial data of the exponential growing modes) do not satisfy this property. We will give in Lemma \ref{modifiedinitialdata} a way to add a quadratic perturbation to these initial data in order to regain \eqref{initialdataconstraint}.

Finally, Theorem \ref{rigorousstatementKEu} will be nothing but an adaptation of Theorem \ref{illposednesstheorem} in the case of the kinetic Euler equation. It can be seen as a stability result for this theorem: if we modify the initial data chosen in the proof of Theorem \ref{illposednesstheorem} by a quadratic perturbation, then the result is still true.
\subsubsection*{The two kinetic Euler equations coincide in analytic regularity}
Take $f$ a smooth solution to \eqref{KEu'}. Such a solution $f$ satisfies
\begin{gather*}
\partial_t \left( \int f(t,x,v) \D v \right) + \Div\left( \int v f (t,x,v) \D v \right) = 0, \\
\partial_t \Div\left( \int v f (t,x,v) \D v \right) + \Div \left( -\nabla p(t,x) \left\{ \int f(t,x,v) \D v - 1 \right\} \right) = 0,
\end{gather*}
and so
\[
\partial_{tt}  \left(\int f(t,x,v) \D v \right) + \Div \left( -\nabla p(t,x) \left\{ \int f(t,x,v) \D v - 1 \right\} \right) = 0,
\]
This equation holds in the measure-valued setting for the solutions of Theorem \ref{existenceofsolutions} taking successively $\varphi \equiv 1$ and $\varphi(v):= v$ in the weak formulation. (The second one is not bounded but $f$ must have a finite second order moment in virtue of \eqref{intPhi}, \eqref{estimS} and \eqref{estimsigmaxi}.) We call $R = (R(t,x))$ the scalar function defined by
\[
R(t,x) := \int f(t,x,v) \D v -1.
\]
Considering the previous computation, as soon as
\begin{equation}
\label{initialconstraints}
\int f_0(\bullet,v) \D v \equiv 1 \quad \mbox{ and } \quad \Div\left( \int v f_0 (\bullet,v) \D v \right) \equiv 0,
\end{equation}
then $R$ must be a solution to the linear equation (once $p$ is known)
\begin{equation}
\label{secondordercontinuity}
\left\{
\begin{gathered}
\partial_{tt} R(t,x) + \Div \Big(- \nabla p(t,x) R(t,x)\Big) = 0,\\
R|_{t=0} \equiv 0, \qquad \partial_t R|_{t = 0} \equiv 0.
\end{gathered}
\right.
\end{equation}
Remark that because of \eqref{incompressibilityconstraint} and \eqref{nulldivergence}, the initial condition of a solution to \eqref{KEu} must satisfy \eqref{initialconstraints}.

In Theorem \ref{existenceofsolutions}, we have built solutions to \eqref{abstracteq}-\eqref{abstractforce} satisfying for some $\delta_0>0$ and $\Gamma>0$ (among other estimates) 
\begin{gather*}
\sup_{0 \leq t \leq \delta_0 / \Gamma} \sup_{0 \leq \delta \leq \delta_0 - \Gamma t } \| \rhorho(t), \vv(t) \|_{\delta} < + \infty,\\
\sup_{0 \leq t \leq \delta_0 / \Gamma} \sup_{0 \leq \delta \leq \delta_0 - \Gamma t } (\delta_0 - \delta - \Gamma t)^{1/2}|\DD p(t)|_{\delta} < + \infty.
\end{gather*}
(The second one is an easy consequence of \eqref{abstractforce}, Proposition \ref{propestimchainrule}, \eqref{intdPhiseries} and \eqref{estimsolutionbis}.) Therefore, there is $C>0$ such that for all $t \in [0, \delta_0/\Gamma)$, for all $\delta \in [0, \delta_0 - \Gamma t)$, 
\begin{equation}
\label{estimpressure}
|\DD p(t)|_{\delta} \leq \frac{C}{(\delta_0 - \delta - \Gamma t)^{1/2}}.
\end{equation}
When integrating the estimate we have on $\rhorho$, we also get
\begin{equation}
\label{Ranalytic}
\sup_{0 \leq t \leq \delta_0 / \Gamma} \sup_{0 \leq \delta \leq \delta_0 - \Gamma t } |R(t)|_{\delta} < + \infty
\end{equation}
We are now able to prove the following.
\begin{Thm}
\label{KEu'impliesKEu}
If \eqref{estimpressure} holds, the only solution to \eqref{secondordercontinuity} satisfying \eqref{Ranalytic} is $0$.
\end{Thm}
In particular, the solutions to \eqref{KEu'} built in Theorem \ref{existenceofsolutions} and for which \eqref{initialconstraints} holds are solutions to \eqref{KEu}, as announced.
\begin{proof}
We call $ T := \delta_0 / \Gamma$ the time of existence of our solution and
\[
t_0 := \sup\{ t < T \mbox{ such that } R(t) = 0 \mbox{ and } \partial_t R(t) = 0 \}.
\]
The goal is to show that $t_0 = T$. By contradiction if it is not the case, we can do the change of variable $t \leftarrow (t-t_0)$, $T \leftarrow (T-t_0) >0$, $\delta_0 \leftarrow (\delta_0 - \Gamma t_0)$ and suppose that $t_0 = 0$. Then, we just have to show that there exists $\eps \in (0,T)$ such that
\[
\forall t \leq \eps, \quad R(t) = 0.
\]
Indeed, if so, for all $t< \eps$, $\partial_t R(t) = 0$ and the definition of $t_0$ would be contradicted. For $\eps \in (0, T)$, we define
\[
N(\eps) := \sup_{0 \leq t \leq \eps} \sup_{0 \leq \delta \leq \delta_0 - \Gamma t } |R(t)|_{\delta} < + \infty.
\]
We will show that if $\eps$ is sufficiently small, then 
\[
N(\eps) \leq \frac{1}{2} N(\eps).
\]
The result follows easily. Because $R(0) = \partial_t R(0) = 0$, for all $t < T$,
\[
R(t,x) = \int_0^t \int_0^s \Div(- \nabla p(\tau,x) R(\tau,x))\D \tau \D s.
\]
Thus, if $\delta < \delta_0 - \Gamma t$,
\[
|R(t)|_{\delta} = \int_0^t \int_0^s |\Div(- \nabla p(\tau) R(\tau))|_{\delta}\D \tau \D s.
\]
But using Proposition \ref{product} and Proposition \ref{propestimnabla}, and defining 
\[
\delta'(\tau) := \delta + \frac{\delta_0 -\Gamma \tau - \delta}{2},
\]
 we get
\begin{align*}
|R(t)|_{\delta} &\leq \int_0^t \int_0^s \frac{| \nabla p(\tau) R(\tau) |_{\delta'(\tau)}}{\delta'(\tau) - \delta}\D \tau \D s\\
&\leq 2 \int_0^t \int_0^s \frac{| \DD p(\tau)|_{\delta'(\tau)} |R(\tau) |_{\delta'(\tau)}}{\delta_0 - \delta - \Gamma \tau}\D \tau \D s.
\end{align*}
By \eqref{estimpressure} and the definition of $N$, if $\eps >0$ and $t \leq \eps$,
\begin{align*}
|R(t)|_{\delta} &\leq 2\sqrt{2} CN(\eps) \int_0^t \int_0^s \frac{1}{(\delta_0 - \delta - \Gamma \tau)^{3/2}}\D \tau \D s \\
&\leq \frac{4\sqrt{2} C}{\Gamma}N(\eps) \int_0^t \frac{1}{(\delta_0 - \delta - \Gamma s)^{1/2}} \D s\\
&= \frac{8\sqrt{2} C}{\Gamma^2}N(\eps) \Big( (\delta_0 - \delta)^{1/2} - (\delta_0 - \delta - \Gamma t)^{1/2} \Big).
\end{align*}
Taking the supremum on $\delta \leq \delta_0 - \Gamma t$, and then on $t \leq \eps$, we get
\[
N(\eps) \leq \frac{8\sqrt{2} C}{\Gamma^2} \sqrt{\Gamma \eps} N(\eps).
\] 
We obtain the result by taking
\[
\eps \leq \frac{\Gamma^3}{512 C^2}.
\]
\end{proof}

\subsubsection*{Choosing appropriate initial conditions} We recall that the initial conditions used in the proof of Theorem \ref{illposednesstheorem} are of the form $(\boldsymbol{1} + \rr_0, \boldsymbol{w} + \uu_0)$, $\rr_0$ and $\uu_0$ being given for all $w \in \R^d$ and $x \in \T^d$ by the formulae:
\begin{equation}
\label{initialdataEGM}
\left\{
\begin{aligned}
r_0^w(x) &:= - c \frac{\cos\Big[n\cdot x - 2 \theta(w)\Big]}{r^2 + (\varphi + u \cdot w)^2} = - c \, \Re \left( \frac{\exp(in\cdot x)}{(r + i \varphi + iu\cdot w)^2} \right),\\
u_0^w(x) &:= c \frac{\sin\Big[n\cdot x - \theta(w)\Big]}{\sqrt{ r^2 + (\varphi + u \cdot w)^2} } \times u = -c\, \Re \left( \frac{i \exp(in \cdot x)}{r + i \varphi + iu\cdot w } \right) \times u,
\end{aligned}
\right.
\end{equation}
where $c$ and $\varphi \in \R$, $r >0$, $n \in \Z^d$, $u = n / |n|$ and $\theta(w) := \arctan\left(\{\varphi + u \cdot w\}/r\right)$. In the case of the kinetic Euler equation, these are initial data of an exponential growing mode corresponding to the eigenvalue $\lambda = |n|(r + i \varphi)$ provided \eqref{penroseKEumeasure} holds. We suppose it is the case.

The first condition in \eqref{initialconstraints} holds for these data. Indeed, in this context (use \eqref{multiphasictokinetic}), we have to check that for all $x \in \T^d$, 
\[
\int r_0^w(x) \D \mu(w) = 0.
\]
But 
\begin{align}
\notag \int r_0^w(x) \D \mu(w) &= - c \, \Re\left( \exp(in\cdot x) \int \frac{\D \mu(w)}{(r + i \varphi + i u\cdot w)^2} \right)\\
\notag &= -c|n|^2 \Re\left( \exp(in\cdot x) \int \frac{\D \mu(w)}{(\lambda + i n\cdot w)^2} \right)\\
\label{rincompressible} &=0, 
\end{align}
by \eqref{penroseKEumeasure}.

However, the second condition in \eqref{initialconstraints} does not hold in general. In this setting, it would mean that
\[
\Div_x \left(  \int \big(w + u_0^w(x)\big) \big(1 + r_0^w(x)\big) \D \mu(w) \right)
\] 
cancels. But
\begin{align*}
\int \big(&w + u_0^w(x)\big) \big(1 + r_0^w(x)\big) \D \mu(w)\\
&= \int w \D \mu(w) \\
&\quad- c\, \Re\left( \exp(in\cdot x) \int \left\{ \frac{w}{(r + i \varphi + i u\cdot w)^2} + \frac{iu}{r + i \varphi + i u\cdot w}\right\}\D \mu(w)  \right)\\
&\quad - c^2 \int \frac{\cos\Big[n\cdot x - 2 \theta(w)\Big] \sin\Big[n\cdot x - \theta(w)\Big]}{\big( r^2 + (\varphi + u \cdot w)^2\big)^{3/2} } \D \mu(w) \times u.
\end{align*}
Taking the divergence, we get
\begin{align*}
\Div_x &\bigg(\int \big(w + u_c^w(0,x)\big) \big(1 + r_c^w(0,x)\big) \D \mu(w)\bigg)\\
&= - c|n|^2 \Re\left( \exp(in\cdot x) \int \left\{ \frac{i n \cdot w }{(\lambda + i n\cdot w)^2} - \frac{1}{\lambda + i n\cdot w}\right\}\D \mu(w)  \right)\\
&\quad -c^2 n\cdot u \int \frac{\cos\Big[2n\cdot x - 3 \theta(w)\Big]}{\big( r^2 + (\varphi + u \cdot w)^2\big)^{3/2}}\D \mu(w).
\end{align*}
But the first term can be rewritten
\begin{equation}
\label{divu=0firstorder}
c |n|^2 \Re\left( \lambda \exp(in\cdot x) \int \frac{\D \mu(w) }{(\lambda + in\cdot w)^2}\right) = 0
\end{equation} 
because of \eqref{penroseKEumeasure}. Finally, we end up with
\begin{align*}
\Div_x \bigg(\int &\big(w + u_0^w(x)\big) \big(1 + r_0^w(x)\big) \D \mu(w)\bigg)\\
& = - c^2 |n|^4 \Re\left( \exp(2 i n \cdot x) \int\frac{\D \mu(w)}{(\lambda + in\cdot w)^3}\right),
\end{align*}
which does not cancel in general. Nevertheless, the crucial point is that the first order (in $c$) cancels. We give the initial data we shall consider in the following lemma.
\begin{Lem}
\label{modifiedinitialdata}
Take $(\rr_0, \uu_0)$ the couple defined in \eqref{initialdataEGM} and suppose \eqref{penroseKEumeasure} holds with $\lambda = |n|(r + i \varphi)$. We call 
\begin{gather}
\label{defV} V := \int \left\{ r_0^w u_0^w - \int r_0^w(y) u_0^w(y) \D y \right\} \D \mu(w)  \\
\label{defutilde} \forall w \in \R^d, \quad \tilde{u}^w_0 = u_0^w - V.
\end{gather}
Then $(\rhorho_0, \tilde{\vv}_0) := (\boldsymbol{1} + \rr_0, \boldsymbol{w} + \tilde{\uu}_0)$ belongs to $\boldsymbol{L}_0$ and satisfies the multiphasic version of \eqref{initialconstraints}.
\end{Lem}
In particular, according to Theorem \ref{KEu'impliesKEu}, the solution to \eqref{KEu'} built in Theorem \ref{existenceofsolutions} starting from these data are also solution to \eqref{KEu}.
\begin{Rem}
The vector field $V$ is a quadratic function of $(\rr, \uu)$. We can even give the following explicit formula using \eqref{initialdataEGM} and trigonometric identities. For all $x \in \T^d$, we have
\begin{equation}
\label{formulaV}
V(x) = -\frac{ c^2}{2}  \int \frac{\sin[2n\cdot x - 3 \theta(w)]}{\big( r^2 + (\varphi + u \cdot w)^2\big)^{3/2} } \D \mu(w) \times u.
\end{equation}
\end{Rem}
\begin{proof}
The first condition in \eqref{initialconstraints} only involves $\rr_0$, which is unchanged, and has already been checked in \eqref{rincompressible}. We just have to check the second one. We have
\begin{align*}
\Div_x &\left(\int (1 + r_0^w)(w + \tilde{u}_0^w) \D \mu(w)\right)\\
 &= \Div_x \left(\int (1 + r_0^w)(w + u_0^w - V) \D \mu(w)\right) \\
&= \Div_x \left(\int \{ r_0^w w + u_0^w \} \D \mu(w)\right) + \Div_x \left(\int r_0^w u_0^w  \D \mu(w)\right)\\
&\quad - \Div_x \left( V \int (1 + r_0^w) \D \mu(w)\right).
\end{align*}
We have checked in \eqref{divu=0firstorder} that the first term cancels. The second one equals the third one because of \eqref{rincompressible} and because by the definition of $V$,
\[ 
\Div V = \Div \left( \int r_0^w u_0^w \D \mu(w)  \right).
\]
Finally, $(0,V) \in \boldsymbol{L}_0$ because $V$ is the gradient of the function defined for all $x \in \T^d$ by:
\[
\frac{c^2}{4 |n|} \int \frac{\cos[2n\cdot x - 3 \theta(w)]}{\big( r^2 + (\varphi + u \cdot w)^2\big)^{3/2}}\D \mu(w).
\]
\end{proof}
\subsubsection*{Stability of Theorem \ref{illposednesstheorem}} We are now ready to state and prove Theorem \ref{illposednesstheorem} in the case of the kinetic Euler equation. Of course as in the previous cases, this theorem has a kinetic version that implies Theorem \ref{kineticstatementKEuVB} in the kinetic Euler case.
\begin{Thm}
\label{rigorousstatementKEu}
Take $\mu$ an unstable profile for the kinetic Euler equation (satisfying the Penrose condition \eqref{penroseKEumeasure}), $s \in \N$ and $\alpha \in (0,1]$. Consider $(\rhorho^k_0, \vv^k_0)_{k \in \N}$ and $(T_k)_{k \in \N}$ the families of data and times given by Theorem \ref{illposednesstheorem}. 

Then for all $k$, the unique analytic solution $(\tilde{\rhorho}^k, \tilde{\vv}^k)$ to the multiphasic kinetic Euler equation starting from $(\rhorho_0^k,  \tilde{\vv}^k_0)$ is defined up to time $T_k$ ($\tilde{\vv}^k_0$ being chosen as in Lemma \ref{modifiedinitialdata}). This family of solutions still satisfy the asymptotics \eqref{almostLyapounov} and \eqref{growthpotential}, and $(T_k)_{k \in \N}$ still satisfies \eqref{eq:estim_Tk_KEu_VB_abstract}.
\end{Thm}
\begin{Rem}
In particular, thanks to Theorem \ref{KEu'impliesKEu} and Lemma \ref{modifiedinitialdata}, Theorem \ref{illposednesstheorem} holds for Equation \eqref{KEu}, and not only for \eqref{KEu'}.
\end{Rem}
\begin{proof}
Let us take $(n_k)_{k \in \N}$, $(\lambda_k)_{k \in \N}$, $(c_k)_{k \in \N}$, $\beta$, $\gamma$, $\Gamma$ and $(T_k)_{k \in \N}$ as in the proof of Theorem \ref{illposednesstheorem}.

First, let us check that for $k$ sufficiently large, there exist a multiphasic solution to \eqref{KEu'} starting from $(\rhorho_0^k, \tilde{\vv}^k_0)$ up to time $T_k$. To use Theorem \ref{existenceofsolutions}, we need to check condition \eqref{initialdatasmall}. We just have to consider the velocity part because the density part is unchanged. Using the notations of Lemma \ref{modifiedinitialdata} with the index $k$ ($V_k$ is defined in \eqref{defV} and $\tilde{\uu}_0^k$ is defined in \eqref{defutilde}), we have (using \eqref{formulaV}):
\begin{align*}
\| \DD \tilde{\uu}^k_0 \|_{\Gamma T_k} &\leq \| \DD \uu^k_0 \|_{\Gamma T_k} + |\DD V_k|_{\Gamma T_k}\\
&\lesssim c_k|n_k| \exp(|n_k| \Gamma T_k) +  c_k^2 |n_k| \exp(2|n_k| \gamma T_k)\\
&\lesssim \frac{c_k^{\alpha/4}}{|n_k|} + \frac{c_k^{\alpha/2}}{ |n_k|},
\end{align*}
the last line being obtained thanks to \eqref{defTkbis}. In particular, if $k$ is sufficiently large, $\| \tilde{\uu}^k_0 \|_{\Gamma T_k} \leq \eps_0$. So for such $k$, Theorem \ref{existenceofsolutions} guarantees the existence of a unique analytic multiphasic solution $(\tilde{\rhorho}_0^k, \tilde{\vv}^k)$ to \eqref{KEu'} up to time $T_k$. It has the following form: for all $w \in \R^d$,
\begin{gather*}
\tilde{\rho}^{k,w} = 1 + r^{k,w} + (\tilde{r}^{k,w} - r^{k,w}) + \tilde{\sigma}^{k,w},\\
\tilde{v}^{k,w} = w + u^{k,w} + (\tilde{u}^{k,w} - u^{k,w}) + \tilde{\xi}^{k,w},
\end{gather*}
where:
\begin{itemize}
\item $(\rr^k(t), \uu^k(t)) = S_t(\rr^k_0, \uu^k_0)$ is given by \eqref{explicitlinearsolutionk}, 
\item $(\tilde{\rr}^k(t), \tilde{\uu}^k(t)) = S_t(\rr_0^k, \tilde{\uu}^k_0)$,
\item $(\tilde{\sigsig}^k, \tilde{\xixi}^k)$ satisfies the same estimates as $(\sigsig^k, \xixi^k)$ in the proof of Theorem \ref{illposednesstheorem}.
\end{itemize}
Comparing with the proof of Theorem \ref{illposednesstheorem}, we just need to show that the additional term
\[
(\tilde{\rr}^k(t), \tilde{\uu}^k(t)) - (\rr^k(t), \uu^k(t)) = S_t(0,V_k)
\]
is negligible both in the estimate of the initial condition and in the $L^1$ estimate. Thus, the two things we have to prove are:
\begin{itemize}
\item for the initial condition
\[
\| V_k \|_{W^{s,\infty}} = \underset{k \to + \infty}{o}\big(c_k^{1-s/\beta}\big),
\]
\item for the $L^1$ estimate
\[
\| S_t(0,V_k) \|_{L^1((0,T_k) \times \T^d)} = \underset{k \to + \infty}{o}(c_k^a),
\]
where $a$ is defined in \eqref{defGammanumdomin}.
\end{itemize}
For the first one, thanks to \eqref{formulaV} and \eqref{defckbeta},
\[
\|V_k\|_{W^{s,\infty}} \lesssim c_k^2 \times  2^s |n_k|^s \lesssim (c_k^{1 - s/\beta})^2 = \underset{k \to + \infty}{o}\big(c_k^{1-s/\beta}\big).
\]
For the second one,
\begin{align*}
\| S_t(0, V_k) \|_{L^1_{t,x}} & \leq \int_0^{T_k} \| S_t(0, V_k) \|_0 \D t \\
&\lesssim \int_0^{T_k} |V_k|_{\Gamma t} \D t &\mbox{by \eqref{sharpestimate}},\\
&\lesssim c_k^2 \int_0^{T_k} \exp(2 |n_k| \Gamma t) \D t &\mbox{by \eqref{formulaV}},\\
&\lesssim \frac{c_k^2}{|n_k|} \exp(2|n_k| \Gamma T_k)\\
&\lesssim \frac{c_k^{\alpha/2}}{|n_k|^3} &\mbox{by \eqref{defTkbis}},\\
&\lesssim c_k^{\alpha/2 + 3/\beta} = \underset{k \to + \infty}{o}\big(c_k^a\big) &\mbox{by \eqref{defGammalindomin}}.
\end{align*}
This concludes the proof.
\end{proof}
\noindent \textbf{Acknowledgments.} This work is part of my PhD thesis supervised by Yann Brenier and Daniel Han-Kwan. I would like to thank both of them for the careful reading and advice. I would also like to thank Toan Nguyen for answering my questions on semigroups.
\begin{appendices}
\section{Superpositions of Diracs are unstable}
\label{penrosediracs}
In this section, we give ourselves an integer $p \geq 2$, $p$ positive numbers $\alpha_1, \dots, \alpha_p$, and $p$ distinct points of $\R^d$, $a_1, \dots, a_p$. We define
\[
\mu := \alpha_1 \delta_{a_1} + \dots + \alpha_p \delta_{a_p}.
\]
The measure $\mu$ is unstable for the three models presented in the introduction. (In the case of the kinetic Euler equation, we must have $\alpha_1 + \dots + \alpha_p = 1$.)
\begin{Thm}
The measure $\mu$ is unstable in the sense of the Penrose conditions \eqref{penroseVPmeasure}, \eqref{penroseKEumeasure} and \eqref{penroseVBmeasure}.
\end{Thm}
\begin{proof}
Take $n \in \Z^d$ such that $n \cdot a_1, \dots, n \cdot a_p$ are distinct. We will show that whatever $e \in \R_+$, there exists $\lambda \in \C$ with $\Re(\lambda) >0$ such that
\begin{equation}
\label{discretepenrose}
\int \frac{\D \mu(w)}{(\lambda + in \cdot w)^2}= \sum_{k=1}^p \frac{\alpha_k}{(\lambda + in \cdot a_k)^2} =-e.
\end{equation}
Applying this property to $e = 1$, $0$ and $1/|n|^2$, we get the Penrose conditions \eqref{penroseVPmeasure}, \eqref{penroseKEumeasure} and \eqref{penroseVBmeasure} respectively.

For all $\lambda \in \C$ with $\Re(\lambda) \neq 0$, \eqref{discretepenrose} holds if and only if $P(\lambda)=0$, where $P$ is the following polynomial:
\[
P(X) := \sum_{k=1}^p \alpha_k \prod_{l \neq k} (X + in\cdot a_l)^2 + e \prod_{k=1}^p (X + i n \cdot a_k)^2.
\]
This polynomial is of degree $2p$ if $e \neq 0$ and $2(p-1)$ if $e=0$. In each case, according to the fundamental theorem of algebra, it admits at least one complex root $z$.

If $z = ix$ with $x \in \R$, then either there is exactly one $k_0$ such that $x = -n\cdot a_{k_0}$ or for all $l$, $x \neq -n\cdot a_l$ (if so, we set $k_0:=1$). In each case, for all $l \neq k_0$, $x \neq -n \cdot a_l$, so
\begin{align*}
P(ix) &= - \sum_{k=1}^p \alpha_k \prod_{l \neq k} (x + n\cdot a_l)^2 - e \prod_{k=1}^p (x + n \cdot a_k)^2 \\
&\leq - \alpha_{k_0} \prod_{l \neq k_0} (x + n\cdot a_l)^2 \\
&< 0.
\end{align*}
Hence, we get a contradiction and $\Re(z) \neq 0$. 

Moreover remark that $P(-\overline{X}) = \overline{P(X)}$. So we also have $P(-\overline{z}) = 0$. But necessarily, $\Re(z) > 0$ or $\Re(-\overline{z})>0$. We conclude that there is $\lambda \in \C$ with $\Re(\lambda)>0$ with $P(\lambda) = 0$. For this $\lambda$, \eqref{discretepenrose} holds.
\end{proof}

\section{Proofs of the properties of the analytic norms}
\label{appendixproof}
We give in this appendix the proofs of the results stated in Subsection \ref{propertiesanalyticnorms}.
\begin{proof}[Proof of Proposition \ref{product}]
We have for all $w \in \R^d$ and $x \in \T^d$:
\[
f^w(x) = \sum_{k \in \Z^d} \hat{f}_k(w) \exp(i k \cdot x) \quad \mbox{and} \quad g^w(x) = \sum_{l \in \Z^d} \hat{g}_l(w) \exp(i l \cdot x),
\]
with for all $n$ $\hat{f}_n$ and $\hat{g}_n$ in $L^{\infty}$. Consequently,
\[
f^w(x) g^w(x) = \sum_{n \in \Z^d} \left(\sum_{k + l = n} \hat{f}_k(w) \hat{g}_l(w) \right) \exp(in\cdot x).
\]
It follows with \eqref{defdoublenorm} that
\begin{align*}
\| \FF \GG \|_{\delta} &= \sum_{n \in \Z^d} \left| \sum_{k + l = n} \hat{f}_k \hat{g}_l \right|_{\infty} \exp(\delta |n|)\\
&\leq \sum_{n \in \Z^d} \sum_{k + l = n} |\hat{f}_k|_{\infty} \exp(\delta |k|) |\hat{g}_l|_{\infty} \exp(\delta |l|) \\
&\leq \left( \sum_{k \in \Z^d} |\hat{f}_k|_{\infty} \exp(\delta |k|) \right) \left( \sum_{l \in \Z^d} |\hat{g}_l|_{\infty} \exp(\delta |l|) \right) \\
&= \| \FF \|_{\delta} \| \GG \|_{\delta}.
\end{align*}
\end{proof}

\begin{proof}[Proof of Proposition \ref{propestimnabla}]
For all $w \in \R^d$ and $x \in \T^d$,
\[
f^w(x) = \sum_{n \in \Z^d} \hat{f}_n(w) \exp(i n \cdot x),
\]
with for all $n$, $\hat{f}_n \in L^{\infty}$. Consequently, if $\delta' < \delta$,
\begin{align*}
\| \DD f \|_{\delta'} &= \sum_{n \in \Z^d} |n| |\hat{f}_n|_{\infty} \exp(\delta' |n|) \\
&=  \sum_{n \in \Z^d} |\hat{f}_n|_{\infty} \exp(\delta |n|) \times |n| \exp\big(-(\delta - \delta')|n| \big)\\
&\leq \left( \sum_{n \in \Z^d} |\hat{f}_n|_{\infty} \exp(\delta |n|) \right)\times \sup_{a \in \R_+} \Big\{ a \exp\big(-(\delta - \delta')a \big) \Big\}\\
&\leq \frac{1}{\delta - \delta'} \| \FF \|_{\delta}.
\end{align*}
\end{proof}

\begin{proof}[Proof of Proposition \ref{propestimchainrule}]
We have already seen in the proof of Proposition \ref{product} that with the same notations, for all $x \in \T^d$ and $w \in \R^d$,
\[
f^w(x) g^w(x) = \sum_{n \in \Z^d} \left(\sum_{k + l = n} \hat{f}_k(w) \hat{g}_l(w) \right) \exp(in\cdot x).
\]
So
\begin{align*}
\| \DD\, (\FF \GG) \|_{\delta} &= \sum_{n \in \Z^d} |n| \left|\sum_{k + l = n} \hat{f}_k \hat{g}_l \right|_{\infty} \exp(\delta |n|)\\
&\leq \sum_{n \in \Z^d}  \sum_{k + l = n} \big( |k| + |l| \big) |\hat{f}_k|_{\infty} |\hat{g}_l|_{\infty} \exp(\delta |n|)\\
&= \sum_{n \in \Z^d}  \sum_{k + l = n} \Big\{|k| |\hat{f}_k|_{\infty} |\hat{g}_l|_{\infty} + |\hat{f}_k|_{\infty} |l| |\hat{g}_l|_{\infty} \Big\} \exp(\delta |k|) \exp(\delta |l|) \\
&= \sum_{k \in \Z^d} |k| |\hat{f}_k|_{\infty}e^{\delta |k|} \sum_{l \in \Z^d}|\hat{g}_l|_{\infty}e^{\delta |l|}+\sum_{k \in \Z^d}  |\hat{f}_k|_{\infty}e^{\delta |k|}\sum_{l \in \Z^d} |l| |\hat{g}_k|_{\infty}e^{\delta |l|}\\
&= \| \FF \|_{\delta} \| \DD \GG \|_{\delta} + \| \GG \|_{\delta} \| \DD \FF \|_{\delta}.
\end{align*}
Inequality \eqref{estimiteratechainrule} is simply obtained by induction on $|\alpha| + |\beta|$.
\end{proof}

\begin{proof}[Proof of Proposition \ref{propestimA}]
For all $x \in \T^d$,
\[
f(x) = \sum_{n \in \Z^d} \hat{f}_n \exp(i n \cdot x),
\]
with for all $n$ $\hat{f}_n \in E$. Consequently
\[
\nabla A f (x) = -i \sum_{n \in \Z^d} P(n) \cdot \hat{f}_n \exp(in\cdot x) n.
\]
In particular, using \eqref{assumptionP},
\begin{align*}
| \nabla  A f |_{\delta} &= \sum_{n \in \Z^d} |n| |P(n) \cdot \hat{f}_n| \exp(\delta |n|)\\
&\leq M \sum_{n \in \Z^d} |n| |\hat{f}_n| \exp(\delta|n|)\\
&\leq M|\DD f|_{\delta}.
\end{align*}
\end{proof}
\begin{proof}[Proof of Lemma \ref{lemDS=SD}]
The unique classical solution to \eqref{abstractlineq} is given by \eqref{rufourier} with $(\hat{\rr}_n, \hat{\uu}_n)$ given by Lemma \ref{fourierlemma}. As a consequence,
\begin{align*}
\big\| \DD\, (\rr(t), \uu(t)) \big\|_{\delta - \gamma t} &\leq \sum_{n \in \Z^d} |n| |(\hat{r}_n(t),\hat{u}_n(t))|_{\infty} \exp\Big(\{\delta - \gamma t\}|n|\Big) \\
&\leq C \sum_{n \in \Z^d}|n| |(\hat{r}_n(0),\hat{u}_n(0))|_{\infty}  \exp(\delta|n|) \\
&\leq C \|\DD \, (\rr_0, \uu_0)\|_{\delta}.
\end{align*}
\end{proof}

\begin{proof}[Proof of Lemma \ref{lemdeltatodelta'}]
For all $w \in \R^d$ and $x \in \T^d$,
\[
f^w(x) = \sum_{n \in \Z^d} \hat{f}_n(w) \exp(i n \cdot x),
\]
with for all $n$, $\hat{f}_n \in L^{\infty}$. But by \eqref{zeromass}, $\hat{f}_0(w) = 0$ for all $w \in \R^d$. Consequently, if $\delta' < \delta$,
\begin{align*}
\| f \|_{\delta'} &= \sum_{n \in \Z^d\backslash \{0\}} |\hat{f}_n|_{\infty} \exp(\delta' |n|) \\
&=  \sum_{n \in \Z^d\backslash \{0\}} |\hat{f}_n|_{\infty} \exp(\delta |n|) \times \exp\big(-(\delta - \delta')|n| \big)\\
&\leq \left( \sum_{n \in \Z^d} |\hat{f}_n|_{\infty} \exp(\delta |n|) \right)\times \exp \big(-(\delta - \delta')\big)\\
&= \frac{\| \FF \|_{\delta}}{\exp(\delta - \delta')} .
\end{align*}
\end{proof}

\begin{proof}[Proof of Lemma \ref{lemestimPhi}]
To prove these three estimates, we only need to consider the case when $\Phi : w \mapsto w^k$ for some $k \in \N^d$. For the general case, it suffices then to multiply the inequalities obtained by $|a_k|$ and to sum over $k$. For $i \in \{1, \dots, d\}$, we denote by $\1_i \in \R^d$ the vector whose only nonzero coordinate is a one at position $i$. Then, if $\alpha$ and $\beta \in \N^d$ are such that $\alpha + \beta = k - \1_i$, we chose $\gamma^i_{\alpha, \beta} \in \N$ in such a way that for all $X$ and $Y \in \R^d$,
\begin{equation}
\label{bernoulli}
X^k - Y^k = \sum_{i=1}^d(X_i - Y_i) \sum_{\alpha + \beta = k -\1_i} \gamma^i_{\alpha, \beta}X^{\alpha} Y^{\beta}.
\end{equation}
(Use Bernoulli's formula to find such $\gamma^i_{\alpha, \beta}$.) Up to now, we omit to specify in each line $i = 1, \dots, d$ and $\alpha + \beta = k - \1_i$. Remark that taking for $h \in \R$, $X = (1 + h) \1_i $ and $Y = \1_i$, the previous formula leads to
\[
(1 + h)^{k_i} - 1 = h \sum_{\alpha, \beta} \gamma^i_{\alpha, \beta} (1 + h)^{\alpha_i}.
\]
Derivating at $h=0$, we obtain
\begin{equation}
\label{gammaki}
\sum_{\alpha, \beta} \gamma^i_{\alpha, \beta} = k_i.
\end{equation}
In particular, summing over $i$,
\begin{equation}
\label{gammak}
\sum_{i, \alpha, \beta} \gamma^i_{\alpha, \beta} = |k|.
\end{equation}
\underline{Proof of \eqref{Phi-Phi}}. We want to show
\[
|(a + f)^k - (a + g)^k|_{\delta} \leq |f-g|_{\delta} \times |k| (|a| + |f,g|_{\delta})^{|k|-1}.
\]
But by \eqref{bernoulli}, the triangle inequality and Proposition \ref{product},
\begin{align*}
|(a + f)^k - (a + g)^k|_{\delta} &= \left| \sum_i (f_i - g_i) \sum_{\alpha, \beta} \gamma^i_{\alpha, \beta} (a + f)^{\alpha} (a + g)^{\beta} \right|_{\delta}\\
&\leq |f-g|_{\delta} \sum_{i, \alpha, \beta }\gamma^i_{\alpha, \beta} (|a| + |f|_{\delta})^{|\alpha|} (|a| + |g|_{\delta})^{|\beta|}\\
&\leq |f-g|_{\delta} (|a| + |f,g|_{\delta})^{|k| - 1} \sum_{i, \alpha, \beta} \gamma^i_{\alpha, \beta}\\
&= |f-g|_{\delta} \times |k| (|a| + |f,g|_{\delta})^{|k| - 1},
\end{align*}
the last line being obtained by \eqref{gammak}.\\
\underline{Proof of \eqref{DPhi}}. To prove
\[
|\DD \,(a + f)^k|_{\delta} \leq |k| |\DD f|_{\delta} (|a| + |f|)^{|k|-1},
\]
it suffices to develop $(a + f)^k$, use the triangle inequality and the fact that for all $\alpha \in \N^d$,
\[
|\DD f^{\alpha}|_{\delta} \leq |\alpha| |f|_{\delta}^{|\alpha|-1} |\DD f|_{\delta}, 
\]
which is \eqref{estimiteratechainrule} of Proposition \ref{propestimchainrule} with $\beta = 0$. One can then re-factorize and get the result. The proof of \eqref{Phi-Phi-dPhi} follows the same path.\\
\underline{Proof of \eqref{DPhi-Phi}}. Here we need to prove
\[\begin{aligned}
|\DD \, \{(a + f)^k - (a + g)^k\}|_{\delta} \leq |f&-g|_{\delta} |\DD f, \DD g|_{\delta} \times |k| (|k| - 1) (|a| + |f,g|_{\delta})^{|k| - 2}\\
& + |\DD\,(f-g)|_{\delta} \times |k| (|a| + |f,g|_{\delta})^{|k|-1} .
\end{aligned}
\]
but using \eqref{bernoulli} and then Proposition \ref{propestimchainrule}, we get
\begin{align*}
|\DD \, \{(a + f)^k - (a + g)^k\}|_{\delta} &\leq \sum_{i, \alpha, \beta} \gamma^i_{\alpha, \beta} |\DD\,\{ (f_i - g_i)(a + f)^{\alpha} (a + g)^{\beta} \}|_{\delta}\\
&\leq |f-g|_{\delta} \underbrace{\sum_{i, \alpha, \beta} \gamma^i_{\alpha, \beta} |\DD\,\{(a + f)^{\alpha} (a + g)^{\beta} \}|_{\delta}}_{:= S_1}\\
&\quad  + |\DD\, (f-g)|_{\delta} \underbrace{ \sum_{i, \alpha, \beta} \gamma^i_{\alpha, \beta} |(a + f)^{\alpha} (a + g)^{\beta}|_{\delta}}_{:= S_2}.
\end{align*}
To estimate $S_1$, remark that thanks to Proposition \ref{propestimchainrule}, for all $\alpha$ and $\beta$ such that $|\alpha| + |\beta| = |k|-1$,
\begin{align*}
|\DD\,\{(a + f)^{\alpha} (a + g)^{\beta} \}|_{\delta} &\leq |\alpha| (|a| + |f|_{\delta})^{|\alpha| - 1} (|a| + |g|_{\delta})^{|\beta|} |\DD f|_{\delta} \\
&\qquad + |\beta| (|a| + |f|_{\delta})^{|\alpha|} (|a| + |g|_{\delta})^{|\beta|-1} |\DD g|_{\delta}\\
&\leq (|\alpha| + |\beta|) (|a| + |f,g|_{\delta})^{|\alpha| + |\beta| - 1} |\DD f, \DD g|_{\delta} \\
&= (|k|-1) (|a| + |f,g|_{\delta})^{|k|-2}|\DD f, \DD g|_{\delta}.
\end{align*}
It remains to sum over $\alpha, \beta$ and to use \eqref{gammak} to get
\[
S_1 \leq |k|(|k|-1) (|a| + |f,g|_{\delta})^{|k|-2}|\DD f, \DD g|_{\delta}.
\]
The sum $S_2$ is estimated in the same way than in the proof of \eqref{Phi-Phi}. The result follows.\\
\underline{Proof of \eqref{DPhi-Phi-dPhi}}. At last, we have to show
\begin{align*}
\Big|\DD \, \{ &(a + f)^k - (a + g)^k - \sum_i k_i a^{k - \1_i} (f_i - g_i)\Big|_{\delta} \\
&\leq \{ |\DD \, (f-g)|_{\delta}|f,g|_{\delta} +  |f-g|_{\delta} |\DD f, \DD g|_{\delta} \} |k|(|k| - 1) (|a| + |f,g|_{\delta})^{|k| - 2}.
\end{align*}
Thanks to \eqref{bernoulli} and \eqref{gammaki},
\begin{align*}
\Big|\DD \, \{ &(a + f)^k - (a + g)^k - \sum_i k_i a^{k - \1_i} (f_i - g_i)\Big|_{\delta} \\
& = \Big| \DD \, \sum_i (f_i - g_i) \sum_{\alpha, \beta} \gamma^i_{\alpha, \beta} \Big( (a + f)^{\alpha}(a + g)^{\beta} - a^{\alpha + \beta} \Big)\Big|_{\delta}.
\end{align*}
By similar computations than before,
\begin{align*}
\Big|\DD \, \{ &(a + f)^k - (a + g)^k - \sum_i k_i a^{k - \1_i} (f_i - g_i)\Big|_{\delta} \\
&\leq  |f-g|_{\delta} \underbrace{ \sum_{i, \alpha, \beta} \gamma^i_{\alpha, \beta} |\DD \, \{ (a + f)^{\alpha}(a + g)^{\beta} \} |_{\delta}}_{:= S_1} \\
&\qquad +|\DD \, (f-g)|_{\delta} \underbrace{\sum_{i, \alpha, \beta} \gamma^i_{\alpha, \beta} \Big| (a + f)^{\alpha}(a + g)^{\beta} - a^{\alpha + \beta} \Big|_{\delta} }_{:=S_2}.
\end{align*}
The sum $S_1$ has already been estimated in the proof of \eqref{DPhi-Phi}. For $S_2$, remark that
\begin{align*}
 \Big| (a + f)^{\alpha}(a + g)^{\beta} - a^{\alpha + \beta} \Big|_{\delta} &\leq (|a| + |f|_{\delta})^{|\alpha|} (|a| + |g|_{\delta})^{|\beta|} - |a|^{|\alpha| + |\beta|}\\
 &\leq (|a| + |f,g|_{\delta})^{|k| - 1} - |a|^{|k|-1}\\
 &\leq (|k| - 1)(|a| + |f,g|_{\delta})^{|k|-2} |f,g|_{\delta}.
\end{align*}
Indeed, for the first line, you just have to develop the product, simplify the term $a^{\alpha + \beta}$, use the triangle inequality and Proposition \ref{product}, and finally re-factorize. Our estimation does not depend on $i, \alpha, \beta$, so taking the sum, by \eqref{gammak},
\[
S_2 \leq |k|(|k|-1)(|a| + |f,g|_{\delta})^{|k|-2} |f,g|_{\delta}.
\]
Hence, the result.
\end{proof}

\end{appendices}
  \bibliography{../../biblio/bibliographie}
 \bibliographystyle{plain}
 \end{document}